%% file: Main.tex
\documentclass[a4paper,twoside]{article}
\usepackage{import}

\usepackage{amssymb}
\usepackage{amsmath}
\usepackage{amsthm}

\usepackage{float}
\usepackage{tabularx}
\usepackage{longtable}
\usepackage{pgfplots}

\usepackage{graphicx}
\usepackage{xcolor}
\usepackage{subfig}
\usepackage{epstopdf}
\DeclareGraphicsExtensions{.pdf, .eps, .png, .jpg}
\usepackage[ngerman,english]{babel}
\usepackage{color}
\usepackage{siunitx}
\usepackage{cite}
\usepackage{comment} %zum auskommentieren von Textstellen
\usepackage{todonotes}

\usepackage{fancyhdr}
\usepackage{datetime}
\newcommand{\tensor}[1]{\mbox{\boldmath{$#1$}\unboldmath}}
\newcommand{\bu}[1]{{\underline{\boldsymbol #1}}}
\newcommand{\ul}[1]{{\underline{ #1}}}

\newtheorem{problem}{Problem}

\newtheorem{remark}{Remark}

\newtheorem{lemma}{Lemma}
\newtheorem{theorem}{Theorem}

\newcommand{\keywords}{\textbf{Key words.}  }
\newcommand{\acknowledgment}{\textbf{Acknowledgments.}  }

\title{A posteriori estimator for the adaptive solution of a quasi-static fracture phase-field model with irreversibility constraints}

\author{Mirjam Walloth and Winnifried Wollner\\*[3\baselineskip]
\small Fachbereich Mathematik, TU Darmstadt \\
\small  Dolivostra{\ss}e 15, 64293 Darmstadt\\
\small {\it  walloth@mathematik.tu-darmstadt.de, wollner@mathematik.tu-darmstadt.de}}

\date{\today}

\begin{document}
\maketitle
\begin{abstract}
  Within this article, we develop a residual type a posteriori
  error estimator for a time discrete quasi-static
  phase-field fracture model. Particular emphasize is given to
  the robustness of the error estimator for the variational inequality
  governing the phase-field evolution with respect to the phase-field
  regularization parameter $\epsilon$.
  The article concludes with numerical examples highlighting the performance
  of the proposed a posteriori error estimators on three standard
  test cases; the single edge notched tension and shear test as well as the
  L-shaped panel test.
 \end{abstract}
\thispagestyle{plain}

\section{Introduction}
\input{Introduction}

\section{A quasi-static fracture phase-field model}\label{Sec:Problem}
\input{ProblemSetting}

\section{Residual-type a posteriori estimator for the variational inequality}\label{Sec:EstimatorVI}
\input{EstimatorVI}
\section{Reliability of the estimator}\label{Sec:Reliability}
\input{Reliability}
\section{Efficiency of the estimator}\label{Sec:Efficiency}
\input{Efficiency}

\section{Residual a posteriori estimator for the equation}\label{Sec:EstimatorEqu}
\input{EstimatorEqu}

\section{Numerical results}\label{Sec:Numerics}
\input{Numerik}

\small{\acknowledgment}{This work was funded by the Deutsche Forschungsgemeinschaft (DFG, German Research
Foundation) -- Projektnummer 392587580 -- SPP 1748}

\bibliographystyle{abbrv}
\bibliography{lit}   % name your BibTeX data base

\end{document}

%% file: Introduction.tex
Modeling of fracture propagation by variational models has a long history.
~\cite{Francfort_Marigo_1998} provided a variational formulation of Griffith's model for
brittle fracture~\cite{Griffith_1921}. See also~\cite{Bourdin_Francfort_Marigo_2008} for
a summary. More recently, such phase-field models
have increased in complexity incorporating different phenomena, see, e.g.,~\cite{Ambati_Gerasimov_Lorenzis_2015,Borden_Verhoosel_Scott_Hughes_Landis_2012,Miehe_Welschinger_Hofacker_2010a,
Miehe_Welschinger_Hofacker_2010,Schlueter_Willenbuecher_Kuhn_Mueller_2014}
and higher order methods have been proposed, e.g.,~\cite{Borden_Hughes_Landis_Verhoosel_2014}.

Since the interface, where a transition between the broken and unbroken
material occurs, is often very narrow adaptive finite element methods have been
proposed for the solution of such problems. 
\cite{Burke_Ortner_Sueli_2010}~started by showing that an alternating refinement
procedure according to a posteriori error estimators for the elastic material and the
phase-field equation in each time step gives rise to a convergent algorithm.
This analysis was extended to more general energy functionals in~\cite{Burke_Ortner_Sueli_2013}.
Improvements towards anisotropic refinements where proposed
in~\cite{Artina_Fornasier_Micheletti_Perotto_2015}, all these contributions dealt with the
irreversibility condition by fixing the phase-field to $0$, i.e., fracture, once a tolerance value
had been reached by the phase-field variable. Thereby avoiding a variational inequality
for the description of the time-discrete fracture.
More heuristic methods, such as a predictor-corrector scheme based on refinement
near the computed fracture~\cite{Heister_Wheeler_Wick_2015} or dual-weighted residual error
estimates~\cite{Wick_2016} have also been proposed.

Within this article, we will analyze the residual based error estimator
proposed in~\cite{Mang_Walloth_Wick_Wollner_2019} for the
a posteriori error estimation within a phase-field fracture model.
In contrast to prior work the analysis will treat the irreversibility
condition of the phase-field by a variational inequality. Due to the modeling
and time discretization this variational inequality is a singularly perturbed
obstacle problem and consequently we will show that our estimates are
robust~\cite{Verfuerth_1998b} with respect
to the singular perturbation. Moreover, we will sketch how stress-splitting
approaches~\cite{Miehe_Welschinger_Hofacker_2010a} can be incorporated into the error estimates.

Various methods for a posteriori error estimation of the obstacle problem
can be found in the literature, see,
e.g.,~\cite{Chen_Nochetto_2000,Veeser_2001,Bartels_Carstensen_2004,Weiss_Wohlmuth_2010,Zou_Veeser_Kornhuber_Graeser_2011,Gudi_Porwal_2014}.
Here we focus on the approach by~\cite{Veeser_2001} utilizing a suitable Galerkin functional and
a useful definition of the discrete constraining forces. 

The rest of the paper is structured as follows. In Section~\ref{Sec:Problem}, we will introduce the
time-discrete phase-field fracture model under consideration and briefly state its discretization.
In Section~\ref{Sec:EstimatorVI}, we introduce some suitable auxiliary problems utilized to
decouple the discretization error for the elasticity equation and the phase-field inequality.
We continue by defining a discrete counterpart of the constraining force and state the error estimator for the
phase-field variable and the Lagrange multiplier for the obstacle.
In Section~\ref{Sec:Reliability}, we show the robust reliability of the proposed estimator.
This is complemented by the efficiency in Section~\ref{Sec:Efficiency}, indeed efficiency is not always robust. It will become robust once the semi-contact zone, near the fracture,
is sufficiently resolved. For completeness, in Section~\ref{Sec:EstimatorEqu}, we state a standard
residual estimator for the elasticity equation in each time step.
The paper concludes with numerical examples in Section~\ref{Sec:Numerics}. Here we demonstrate
the robustness of the proposed error estimators on three standard test cases, the single edge
notched shear and tension tests as well as an L-shaped panel test.

%% file: ProblemSetting.tex
Let $\Omega\subset\mathbb{R}^2$ be a polygonal domain of a linear elastic body in which a lower dimensional fracture $\mathfrak{C}$ may exist and propagate. Let $I=(0,T)$ be the time interval. The displacements are given by the function $\boldsymbol{u}: \Omega\times I \rightarrow \mathbb{R}^2$. Based on the phase-field approach the fracture is approximated by the phase-field variable $\varphi: \Omega\times I \rightarrow [0,1]$ where $\varphi=1$ characterizes the unbroken material and $\varphi=0$ the broken material. The intermediate values constitute a smooth transition zone dependent on a small regularization parameter $\epsilon$. The physics of the underlying problem ask to enforce that the fracture cannot heal. This condition  is called irreversibility condition. 

The boundary $\Gamma = \partial\Omega$ is subdivided in Dirichlet $\Gamma^D$ and Neumann boundary $\Gamma^N$ where we enforce Dirichlet and Neumann boundary values for the displacements $\boldsymbol{u}$. For the phase-field variable, we have Neumann values $\boldsymbol{\nabla}\varphi\cdot \tensor{n}= 0$ on the whole boundary where $\boldsymbol{n}$ is the unit outward normal to the boundary.

We denote the critical energy release rate by $G_c$.
A degradation function is defined as $g(\varphi):=(1-\kappa)\varphi^2+\kappa$ where $\kappa$ is a small regularization parameter. 
The stress tensor $\boldsymbol{\sigma}(\boldsymbol{u})$ is given by
\begin{align*}
\boldsymbol{\sigma}(\boldsymbol{u}) := 2 \mu \boldsymbol{E}_{\text{lin}}(\boldsymbol{u}) + \lambda \operatorname{tr} (\boldsymbol{E}_{\text{lin}}(\boldsymbol{u})) \textbf{id}.
\end{align*} 
Here, $\lambda$ and $\mu$ are the Lam{\'e} constants, $\boldsymbol{E}_{\text{lin}}(\boldsymbol{u})$ is the linearized strain tensor:
\begin{align*}
\boldsymbol{E}_{\text{lin}}(\boldsymbol{u}):=\frac{1}{2} (\nabla \boldsymbol{u} + \nabla \boldsymbol{u}^T),
\end{align*}
and $\textbf{id}$ denotes the two-dimensional identity matrix.
Often the relation between $\boldsymbol{\sigma}$ and $\boldsymbol{E}_{\mathrm{lin}}$ is given by means of Hooke's tensor, i.e. 
\begin{align*}
\sigma_{ij}(\boldsymbol{u})= C_{ijml}(E_{\mathrm{lin}}(\boldsymbol{u}))_{ml}
\end{align*}
where $C_{ijml}$ are the components of Hooke's tensor which is symmetric, elliptic and bounded. 

We consider a time discrete formulation on
a fixed subdivision $0 =t_0 < t_1 < \ldots < t_N = T$ of the interval
$I$. We define approximations $(\boldsymbol{u}^n,\varphi^n) \approx
(\boldsymbol{u}(t_n),\varphi(t_n))$ and enforce a so-called
discrete irreversibility condition given by
$\varphi^{n}\leq\varphi^{n-1}$ for all $n = 1, \ldots, N$. The discrete irreversibility condition is an approximation of the condition that the fracture cannot heal.

In each time step, we seek the displacement variable in
$\boldsymbol{\mathcal{H}}_D^n:=\{\boldsymbol{v}\in \boldsymbol{H}^1(\Omega)\mid \mbox{tr}|_{\Gamma_D}(\boldsymbol{v})=\boldsymbol{u}_D(t^n)\mbox{ a.e. on }\Gamma_D \}$. Further, we need the test space 
$\boldsymbol{\mathcal{H}}_{0}:=\{\boldsymbol{w}\in \boldsymbol{H}^1(\Omega)\mid \mbox{tr}|_{\Gamma_D}(\boldsymbol{w})=\boldsymbol{0}\mbox{ a.e. on }\Gamma_D\}$. 
To give the weak formulation in each time step $n$, we define the feasible set $\mathcal{K}(\varphi^{n-1}):=\{\psi\in H^1(\Omega)\mid \psi\leq \varphi^{n-1}\leq 1\}$ for the phase-field variable. 
We denote the $L^2$-scalar product by $\left<\cdot,\cdot\right>$ and dual pairings by $\left<\cdot,\cdot\right>_{-1,1}$.

Thus, the weak problem formulation in each time step $n$ is given by
\begin{problem}[Weak formulation in each time step]\label{WeakFormulation}
Find $(\boldsymbol{u}^n,
\varphi^n)\in\boldsymbol{\mathcal{H}}_D^n\times\mathcal{K}(\varphi^{n-1})$
such that
\begin{equation}\label{eq:VI_orig}
\begin{aligned}
\left<g(\varphi^{n-1})\boldsymbol{\sigma}(\boldsymbol{u}^n), \boldsymbol{E}_{\mathrm{lin}}(\boldsymbol{w})\right>& = 0\quad \forall \boldsymbol{w}\in \boldsymbol{\mathcal{H}}_0\\ 
\left<(1-\kappa)\varphi^n \boldsymbol{\sigma}(\boldsymbol{u}^n):\boldsymbol{E}_{\mathrm{lin}}(\boldsymbol{u}^n),\psi -\varphi^n\right>\qquad\qquad\qquad  &\\ 
- \frac{G_c}{\epsilon}\left<1-\varphi^n, \psi -\varphi^n \right> + \epsilon G_c\left<\nabla\varphi^n,\nabla(\psi -\varphi^n)\right>  &\ge 0\quad \forall \psi\in\mathcal{K}(\varphi^{n-1})
\end{aligned}
\end{equation}
\end{problem}

In Miehe et al.~\cite{Miehe_Welschinger_Hofacker_2010} a stress splitting into a crack driving and a non crack driving part has been proposed for fracture phase-field models. 
The linearized strain tensor is decomposed into its tensile and compressive parts, i.e., $\boldsymbol{E}_{\mathrm{lin}} := \boldsymbol{E}^+_{\mathrm{lin}}  + \boldsymbol{E}^-_{\mathrm{lin}}$ with
\begin{align*}
\boldsymbol{E}^+_{\mathrm{lin}} := \boldsymbol{Q}\boldsymbol{D}^+ \boldsymbol{Q}^T
\end{align*} 
where $\boldsymbol{Q}$ is the matrix of eigenvectors, of $\boldsymbol{E}_{\mathrm{lin}}$, and $\boldsymbol{D}$ the matrix with the corresponding eigenvalues on the diagonal. Further, $(\cdot)^+$ denotes the positive part, i.e., on the diagonal of $\boldsymbol{D}^+$ are either the positive eigenvalues or zeros. 
We use the stress splitting of~\cite{Miehe_Welschinger_Hofacker_2010} which is given by
\begin{align*}
\boldsymbol{\sigma}^+(\boldsymbol{u}) :=& 2\mu\; \boldsymbol{E}^+_{\mathrm{lin}}(\boldsymbol{u}) + \lambda\;\mathrm{max}\{0,\operatorname{tr}(\boldsymbol{E}_{\mathrm{lin}}(\boldsymbol{u}))\}\textbf{id}, \\[3pt]
\boldsymbol{\sigma}^-(\boldsymbol{u}) := &2\mu\; \boldsymbol{E}^-_{\mathrm{lin}}(\boldsymbol{u}) + \lambda\;\mathrm{min}\{0,\operatorname{tr}(\boldsymbol{E}_{\mathrm{lin}}(\boldsymbol{u}))\}\textbf{id}
\end{align*}
where $\boldsymbol{\sigma}^+$ is the crack driving part of the stress. 
With these definitions and notations the time discrete weak formulation of the quasi-static fracture phase-field model according to~\cite{Miehe_Welschinger_Hofacker_2010}
reads as follows
\begin{problem}[Weak formulation in each time step with Miehe stress splitting]\label{WeakFormulationWithSplit}
Find $(\boldsymbol{u}^n,
\varphi^n)\in\boldsymbol{\mathcal{H}}_D^n\times\mathcal{K}(\varphi^{n-1})$
such that
\begin{equation}\label{eq:VI_orig_split}
\begin{aligned}
\left<g(\varphi^{n-1})\boldsymbol{\sigma}^+(\boldsymbol{u}^n) + \boldsymbol{\sigma}^-(\boldsymbol{u}^n), \boldsymbol{E}_{\mathrm{lin}}(\boldsymbol{w})\right>& = 0\quad \forall \boldsymbol{w}\in \boldsymbol{\mathcal{H}}_0\\ 
\left<(1-\kappa)\varphi^n \boldsymbol{\sigma}^+(\boldsymbol{u}^n):\boldsymbol{E}_{\mathrm{lin}}(\boldsymbol{u}^n),\psi -\varphi^n\right> \qquad\qquad\qquad &\\ 
- \frac{G_c}{\epsilon}\left<1-\varphi^n, \psi -\varphi^n \right>
+ \epsilon G_c\left<\nabla\varphi^n,\nabla(\psi -\varphi^n)\right>  &\ge 0\quad \forall \psi\in\mathcal{K}(\varphi^{n-1})
\end{aligned}
\end{equation}
\end{problem}

\subsection{Discrete formulation}
In the discrete setting, at each time step $n = 1,\ldots, N$, we decompose
the polygonal domain $\Omega$ by a (family of) meshes $\mathfrak{M}^n$ 
consisting of shape regular parallelograms or triangles $\mathfrak{e}$, such that all meshes share a
common coarse mesh. To allow for local
refinement, in particular of rectangular elements, we allow for one hanging node per edge at which degrees of freedom
will be eliminated to assert conformity of the discrete
spaces. Further, we assume that the boundary of the domain is resolved by the chosen meshes.

To each mesh, we associate the mesh size function $h^n$, i.e.,
$h^n_{\mathfrak{e}} = h^n\lvert_{\mathfrak{e}}= \operatorname{diam}{\mathfrak{e}}$ for any
element $\mathfrak{e} \in \mathfrak{M}^n$.
The set of nodes $p$ is given by $\mathfrak{N}$ and we distinguish between the set
$\mathfrak{N}^{\Gamma}$ of nodes at the boundary
and the set of interior nodes $\mathfrak{N}^I$.

Further, for a point $p \in \mathfrak{N}$, we define a patch $\omega_p$ as the interior of the union of all elements sharing the node $p$. We call the union of all sides in the interior of $\omega_p$, not including the boundary of $\omega_p$, skeleton and denote it by $\gamma_p^I$. For boundary nodes, we denote the intersections between $\Gamma$ and $\partial\omega_p$ by  $\gamma_p^{\Gamma}:=\Gamma \cap\partial\omega_p$.
Further, we will make use of $\omega_{\mathfrak{s}}$ which is the union of all elements sharing a side $\mathfrak{s}$.
Later on, we need the definition of the jump term $[\nabla\psi_h]:= \nabla|_{\mathfrak{e}}\psi_h\cdot\boldsymbol{n}_{\mathfrak{e}}- \nabla|_{\tilde{\mathfrak{e}}}\psi_h\cdot\boldsymbol{n}_{\mathfrak{e}}$ where $\mathfrak{e}, \tilde{\mathfrak{e}}$ are neighboring elements and $\boldsymbol{n}_{\mathfrak{e}}$ is the unit outward normal on the common side of the two elements. 
For the discretization, we consider linear finite elements on triangles and bilinear finite elements on parallelograms. 
We abbreviate
\begin{displaymath}
\mathbb{S}_1(\mathfrak{e}):=\left\{\begin{array}{cccc} \mathbb{P}_1(\mathfrak{e}), &\mbox{if}\; &\mathfrak{e} \;\mbox{is a}&\ \mbox{triangle,}\\  \mathbb{Q}_1(\mathfrak{e}), &\mbox{if}\;& \mathfrak{e} \;\mbox{is a}&\ \mbox{parallelogram}. \end{array}\right. 
\end{displaymath}
We define the space of continuous (bi-)linear finite elements by 
\begin{equation*}
\mathcal{H}_{\mathfrak{m}}:= \{\zeta_{\mathfrak{m}} \in \mathcal{C}^0(\bar{\Omega})\mid\forall\mathfrak{e}\in\mathfrak{M}, \;\zeta_{\mathfrak{m}}|_{\mathfrak{e}}\in \mathbb{S}_1(\mathfrak{e})\}.
\end{equation*}
The nodal basis functions of the finite element spaces are denoted by $\phi_p$. 
Hence, a finite element function has the representation
\begin{equation*}
\zeta_{\mathfrak{m}}=\sum_{p\in \mathfrak{N}}\zeta_{\mathfrak{m}}(p)\phi_p.
\end{equation*}

We assume the Dirichlet data $\boldsymbol{u}_D(t^n)\in\boldsymbol{\mathcal{C}}^0$ to be continuous and piecewise linear on the coarsest meshes. 
Thus, we seek the discrete displacements in the subset
\begin{equation*}
\boldsymbol{\mathcal{H}}^n_{\mathfrak{m},D}:= \{\boldsymbol{v}_{\mathfrak{m}} \in \mathcal{C}^0(\bar{\Omega})\mid\forall\mathfrak{e}\in\mathfrak{M}^n, \;\boldsymbol{v}_{\mathfrak{m}}|_{\mathfrak{e}}\in \mathbb{S}_1(\mathfrak{e}) \text{ and }\boldsymbol{v}_{\mathfrak{m}}=\boldsymbol{u}_D(t^n)\;\mbox{on }\Gamma^D\}.
\end{equation*}
The corresponding discrete test space is given by
\begin{equation*}
\boldsymbol{\mathcal{H}}^n_{\mathfrak{m},0}:= \{\boldsymbol{v}_{\mathfrak{m}} \in \mathcal{C}^0(\bar{\Omega})\mid\forall\mathfrak{e}\in\mathfrak{M}^n, \;\boldsymbol{v}_{\mathfrak{m}}|_{\mathfrak{e}}\in \mathbb{S}_1(\mathfrak{e}) \text{ and }\boldsymbol{v}_{\mathfrak{m}}=\boldsymbol{0}\;\mbox{on }\Gamma^D\}.
\end{equation*}
For the discrete phase-field variable the discrete feasible set is given by 
\begin{equation*}
\mathcal{K}^n_{\mathfrak{m}}:=\{\psi_{\mathfrak{m}}\in \mathcal{H}_{\mathfrak{m}}\mid \psi_{\mathfrak{m}}(p)\leq I_{\mathfrak{m}}^n(\varphi_{\mathfrak{m}}^{n-1})(p), \; \forall p\in\mathfrak{N}\}
\end{equation*}
where $I_{\mathfrak{m}}^n$ is the nodal interpolation operator on the mesh $\mathfrak{M}^n$.

Thus, the discrete formulation of Problem~\ref{WeakFormulation} is given by
\begin{problem}[Discrete formulation in each time step]\label{DiscreteFormulation}
Find $(\boldsymbol{u}^n_{\mathfrak{m}},
\varphi^n_{\mathfrak{m}})\in\boldsymbol{\mathcal{H}}^n_{\mathfrak{m},D}\times\mathcal{K}^n_{\mathfrak{m}}$
such that
\begin{equation}\label{eq:VI_discrete_orig}
\begin{aligned}
\left<g(\varphi^{n-1}_{\mathfrak{m}})\boldsymbol{\sigma}(\boldsymbol{u}^n_{\mathfrak{m}}), \boldsymbol{E}_{\mathrm{lin}}(\boldsymbol{w}_{\mathfrak{m}})\right>& = 0\quad \forall \boldsymbol{w}_{\mathfrak{m}}\in \boldsymbol{\mathcal{H}}^n_{\mathfrak{m},0}\\ 
\left<(1-\kappa)\varphi^n_{\mathfrak{m}}
\boldsymbol{\sigma}(\boldsymbol{u}^n_{\mathfrak{m}}):\boldsymbol{E}_{\mathrm{lin}}(\boldsymbol{u}^n_{\mathfrak{m}}),\psi_{\mathfrak{m}}
  -\varphi^n_{\mathfrak{m}}\right>  \qquad\qquad\qquad&\\
- \frac{G_c}{\epsilon}\left<1-\varphi^n_{\mathfrak{m}}, \psi_{\mathfrak{m}} -\varphi^n_{\mathfrak{m}} \right>
 + \epsilon G_c\left<\nabla\varphi^n_{\mathfrak{m}},\nabla(\psi_{\mathfrak{m}} -\varphi^n_{\mathfrak{m}})\right>  &\ge 0\quad \forall \psi_{\mathfrak{m}}\in\mathcal{K}^n_{\mathfrak{m}}
\end{aligned}
\end{equation}
\end{problem}
Using the splitting proposed in~\cite{Miehe_Welschinger_Hofacker_2010}, we get 
\begin{problem}[Discrete formulation in each time step with Miehe stress splitting]\label{DiscreteFormulationWithSplit}
Find $(\boldsymbol{u}^n_{\mathfrak{m}},
\varphi^n_{\mathfrak{m}})\in\boldsymbol{\mathcal{H}}^n_{\mathfrak{m},D}\times\mathcal{K}^n_{\mathfrak{m}}$
such that
\begin{equation} \label{eq:VI_discrete_orig_split}
\begin{aligned}
\left<g(\varphi^{n-1}_{\mathfrak{m}})\boldsymbol{\sigma}^+(\boldsymbol{u}^n_{\mathfrak{m}}) + \boldsymbol{\sigma}^-(\boldsymbol{u}^n_{\mathfrak{m}}), \boldsymbol{E}_{\mathrm{lin}}(\boldsymbol{w}_{\mathfrak{m}})\right>& = 0\quad \forall \boldsymbol{w}_{\mathfrak{m}}\in \boldsymbol{\mathcal{H}}^n_{\mathfrak{m},0}\\
\left<(1-\kappa)\varphi^n_{\mathfrak{m}}
  \boldsymbol{\sigma}^+(\boldsymbol{u}^n_{\mathfrak{m}}):\boldsymbol{E}_{\mathrm{lin}}(\boldsymbol{u}^n_{\mathfrak{m}}),\psi_{\mathfrak{m}}
  -\varphi^n_{\mathfrak{m}}\right> \qquad\qquad\qquad&\\
 - \frac{G_c}{\epsilon}\left<1-\varphi^n_{\mathfrak{m}}, \psi_{\mathfrak{m}} -\varphi^n_{\mathfrak{m}} \right>
 + \epsilon G_c\left<\nabla\varphi^n_{\mathfrak{m}},\nabla(\psi_{\mathfrak{m}} -\varphi^n_{\mathfrak{m}})\right>  &\ge 0\quad \forall \psi_{\mathfrak{m}}\in\mathcal{K}^n_{\mathfrak{m}}
\end{aligned}
\end{equation}
\end{problem}

%% file: EstimatorVI.tex
In this section, we propose a residual-type a posteriori estimator for the adaptive solution of the quasi-static phase-field model (Problems~\ref{WeakFormulation} and~\ref{DiscreteFormulation}). We comment on how the estimator changes for the problem formulations with the stress splitting (Problem~\ref{WeakFormulationWithSplit} and Problem~\ref{DiscreteFormulationWithSplit}). As the structure remains the same for all
time steps, we consider one time step $n$, only. We drop the
now superfluous superscript $n$ for the solution and for other quantities as e.g., $h_{\mathfrak{e}}:= h^{n}_{\mathfrak{e}}$. 

The proofs of reliability and efficiency are given in Sections~\ref{Sec:Reliability} and~\ref{Sec:Efficiency}.

\subsection{Auxiliary problem formulation}
The residual-type a posteriori estimator proposed in this section is derived for the solution of the following variational inequality (Problem~\ref{AuxiliaryProblemCont}).
\begin{problem}\label{AuxiliaryProblemCont}
Let $\boldsymbol{u}^n_{\mathfrak{m}}$ and $\varphi^{n-1}_{\mathfrak{m}}$ be given, then find $\hat{\varphi}\in \mathcal{K}( I_{\mathfrak{m}}^n(\varphi_{\mathfrak{m}}^{n-1}))$ such that
\begin{equation}\label{eq:VI_DiscreteBilinearForm}
a_{\mathfrak{m},\epsilon}(\hat{\varphi} , \psi -\hat{\varphi} ) \ge\left< \frac{G_c}{\epsilon}, \psi - \hat{\varphi}\right>\quad \forall \psi \in \mathcal{K}( I_{\mathfrak{m}}^n(\varphi_{\mathfrak{m}}^{n-1}))
\end{equation}
\end{problem}
where the bilinear form is given by 
\begin{equation}\label{eq:EpsBilinearFormDiscrete}
a_{\mathfrak{m},\epsilon}(\zeta,\psi) := \left<\left(\frac{G_c}{\epsilon} + \left(1-\kappa\right)\left(\boldsymbol{\sigma}\left(\boldsymbol{u}^n_{\mathfrak{m}}\right):\boldsymbol{E}_{\mathrm{lin}}\left(\boldsymbol{u}^n_{\mathfrak{m}}\right)\right)\right)\zeta,\psi\right>  + G_c\epsilon \left<\nabla \zeta,\nabla \psi \right>,
\end{equation}
and $\mathcal{K}(I^n_{\mathfrak{m}}(\varphi^{n-1}_{\mathfrak{m}})):=\{\psi\in \mathcal{H}\mid \psi\leq I^n_{\mathfrak{m}}(\varphi_{\mathfrak{m}}^{n-1})\}$.

It exists a distribution $\hat{\Lambda}\in H^{-1}$, called constraining force density, which turns the variational inequality~\eqref{eq:VI_DiscreteBilinearForm} into an equation
\begin{equation*}
\left<\hat{\Lambda} ,\psi\right>_{-1,1}:=\left<\frac{G_c}{\epsilon}, \psi \right> -a_{\mathfrak{m},\epsilon}(\hat{\varphi}, \psi )\quad \forall \psi\in H^1.
\end{equation*}

As discrete approximation of Problem~\ref{AuxiliaryProblemCont}, we consider the following Problem
\begin{problem}\label{AuxiliaryProblemDisc}
Let $\boldsymbol{u}^n_{\mathfrak{m}}$ and $\varphi^{n-1}_{\mathfrak{m}}$ be given, then find $\hat{\varphi}_{\mathfrak{m}}\in \mathcal{K}_{\mathfrak{m}}^n$ such that
\begin{equation}\label{eq:VI_DiscreteBilinearForm_Discrete}
a_{\mathfrak{m},\epsilon}(\hat{\varphi}_{\mathfrak{m}} , \psi_{\mathfrak{m}} -\hat{\varphi}_{\mathfrak{m}}) \ge\left< \frac{G_c}{\epsilon}, \psi_{\mathfrak{m}} - \hat{\varphi}_{\mathfrak{m}}\right>\quad \forall \psi_{\mathfrak{m}} \in \mathcal{K}_{\mathfrak{m}}^n
\end{equation}
\end{problem}
We define the corresponding discrete constraining force density
$\hat{\Lambda}_{\mathfrak{m}}\in \mathcal{H}_{\mathfrak{m}}^*$ as
\begin{equation}\label{eq:DiscreteConstrainingForce}
\left< \hat{\Lambda}_{\mathfrak{m}},\psi_{\mathfrak{m}} \right>_{-1,1} := \left<\frac{G_c}{\epsilon},\psi_{\mathfrak{m}}\right>-a_{{\mathfrak{m}},\epsilon}(\hat{\varphi}_{\mathfrak{m}}, \psi_{\mathfrak{m}})\quad\forall \psi_{\mathfrak{m}}\in\mathcal{H}_{{\mathfrak{m}}}.
\end{equation}

We note that the discrete solution $\hat{\varphi}_{\mathfrak{m}}$ of~\eqref{eq:VI_DiscreteBilinearForm_Discrete} equals the discrete solution $\varphi^n_{\mathfrak{m}}$ of Problem~\ref{DiscreteFormulation} in time step $n$. Further, as the bilinear form $a_{\mathfrak{m},\epsilon}(\cdot,\cdot)$ depends on the approximation $\boldsymbol{u}^n_{\mathfrak{m}}$ of $\boldsymbol{u}^n$ and the constraints depend on the approximation $I^n_{\mathfrak{m}}(\varphi_{\mathfrak{m}}^{n-1})$  of $\varphi^{n-1}$, the solution $\hat{\varphi}$ of~\eqref{eq:VI_DiscreteBilinearForm} 
is an approximation to the solution $\varphi^n$ of~\eqref{eq:VI_orig}.

\subsection{Error measure and quasi-discrete constraining force}
The error will be measured in the solution of the variational inequality as well as in the constraining forces as has been proposed in~\cite{Veeser_2001} for the obstacle problem. We measure the error of the solution $\hat{\varphi}$ in the energy norm
\begin{equation}\label{eq:EnergyNorm}
\|\cdot\|_{\epsilon}:= \left\{G_c\epsilon\|\nabla (\cdot)\|^2 + \|\left(\frac{G_c}{\epsilon}+(1-\kappa)\boldsymbol{\sigma}(\boldsymbol{u}^n_{\mathfrak{m}}):\tensor{E}_{\mathrm{lin}}(\boldsymbol{u}^n_{\mathfrak{m}})\right)^{\frac{1}{2}}(\cdot)\|^2 \right\}^{\frac{1}{2}}
\end{equation}
which corresponds to the bilinear form $a_{\mathfrak{m},\epsilon}(\cdot,\cdot)$.
\begin{remark}
We note that in the case of stress splitting
$\boldsymbol{\sigma}:\boldsymbol{E}_{\mathrm{lin}}$
in~\eqref{eq:EpsBilinearFormDiscrete} and in~\eqref{eq:EnergyNorm} is
replaced by $\boldsymbol{\sigma}^+:\boldsymbol{E}_{\mathrm{lin}}$. The
resulting bilinear form is positive definite as
\begin{equation}\label{eq:posdef}
\begin{aligned}
\boldsymbol{\sigma}^+:\boldsymbol{E}_{\mathrm{lin}} &= 2\mu \boldsymbol{E}^+_{\mathrm{lin}}:\boldsymbol{E}_{\mathrm{lin}} + \lambda\mathrm{max}\{0,\mathrm{tr}(\boldsymbol{E}_{\mathrm{lin}})\}\boldsymbol{id}:\boldsymbol{E}_{\mathrm{lin}}\\ 
&= 2\mu (\boldsymbol{Q}\boldsymbol{D}^+\boldsymbol{Q}^T):(\boldsymbol{Q}\boldsymbol{D}\boldsymbol{Q}^T) +  \lambda\mathrm{max}\{0,\mathrm{tr}(\boldsymbol{E}_{\mathrm{lin}})\}\boldsymbol{id}:\boldsymbol{E}_{\mathrm{lin}}\\ 
& = 2\mu\; \mathrm{tr}(\boldsymbol{D}^+)^2 + \lambda\mathrm{max}\{0,\mathrm{tr}(\boldsymbol{E}_{\mathrm{lin}})\}^2>0.
\end{aligned}
\end{equation}
Thus, the energy norm is well defined.
\end{remark}
The error in the constraining forces is measured in the corresponding dual norm $\|\cdot\|_{\ast,\epsilon}:= \frac{\mathrm{sup}_{\psi\in H^1}\left<\cdot, \psi \right>_{-1,1}}{\|\psi\|_{\epsilon}}$.

In order to compare the continuous and discrete constraining forces, we cannot simply take $\hat{\Lambda}_{\mathfrak{m}}$ given by definition~\eqref{eq:DiscreteConstrainingForce} as it is a functional on the space of discrete functions, only, and not a functional on $H^1$. 
There is no unique definition how $\hat{\Lambda}_{\mathfrak{m}}$ acts on a function in $H^1$ which is not in $\mathcal{H}_{\mathfrak{m}}$. Thus, we have to define a suitable approximation of $\hat{\Lambda}$ as a functional on $H^1$ on the basis of the properties of the discrete solution $\hat{\varphi}_{\mathfrak{m}}$ and $\hat{\Lambda}_{\mathfrak{m}}$. We call it quasi-discrete constraining force and denote it by $\widetilde{\hat{\Lambda}}_{\mathfrak{m}}$.
In~\cite{Veeser_2001} such a functional on $H^1$ has been proposed by means of lumping 
$\sum_{p\in\mathfrak{N}^C}s_p\phi_p$,
where $s_p = \frac{\left<\hat{\Lambda}_{\mathfrak{m}},\phi_p\right>_{-1,1}}{\int_{\omega_p}\phi_p}\ge 0$
are the node values of the lumped discrete constraining force. The sign condition follows from the discrete variational inequality. 
As the lumped discrete constraining force  is a discrete function a complementarity condition, i.e., $\hat{\Lambda}_{\mathfrak{m}}\cdot(\hat{\varphi}_{\mathfrak{m}}-I_{\mathfrak{m}}^n(\varphi^{n-1}_{\mathfrak{m}}))=0$, cannot be fulfilled in the so-called semi-contact zone which consists of elements having nodes which are in contact and nodes which are not in contact. It is only valid in so-called full-contact areas where $\hat{\varphi}_{\mathfrak{m}}=I_{\mathfrak{m}}^n(\varphi^{n-1}_{\mathfrak{m}})$  and in non-actual-contact areas where $\hat{\varphi}_{\mathfrak{m}}<I_{\mathfrak{m}}^n(\varphi^{n-1}_{\mathfrak{m}})$. 

Especially for the efficiency and the localization of a posteriori error estimation it is very advantageous, if the quasi-discrete constraining force density can be defined differently for the different areas of full- and semi-contact to reflect local properties. Such an approach has been used first for the derivation of an a posteriori error estimator in~\cite{Fierro_Veeser_2003} and applied to obstacle and contact problems in, e.g.,~\cite{Nochetto_Siebert_Veeser_2005, Moon_Nochetto_Petersdorff_Zhang_2007, Krause_Veeser_Walloth_2015, Gudi_Porwal_2014, Gudi_Porwal_2016,Walloth_2019,Walloth_2020}.
Following this approach, we distinguish between full-contact nodes $p\in\mathfrak{N}^{fC}$ and semi-contact nodes $p\in\mathfrak{N}^{sC}$. Full-contact nodes are those nodes for which the solution is fixed to the obstacle, i.e., $\hat{\varphi}^n_{\mathfrak{m}}=I_{\mathfrak{m}}^n(\varphi^{n-1}_{\mathfrak{m}})$ on $\omega_p$, and the sign condition
\[
0\leq\left<\mathcal{R}^{lin}_{\mathfrak{m}},\psi\right>_{-1,1,\omega_p}:=\left<\frac{G_c}{\epsilon},\psi\right>-a_{{\mathfrak{m}},\epsilon}(\hat{\varphi}_{\mathfrak{m}}, \psi)\qquad \forall \psi\ge 0\in \mathcal{H}_0(\omega_p)
\]
is fulfilled. The latter condition means that the solution is locally not improvable, see the explanation in~\cite{Moon_Nochetto_Petersdorff_Zhang_2007}. Semi-contact nodes are those nodes for which $\hat{\varphi}^n_{\mathfrak{m}}(p)=I_{\mathfrak{m}}^n(\varphi^{n-1}_{\mathfrak{m}})(p)$ holds but not the above conditions of full-contact. 
Based on this classification, we define the quasi-discrete constraining force 
\begin{align}\label{eq:QuasiDiscreteConstrainingForce}
&\left< \widetilde{\hat{\Lambda}}_{\mathfrak{m}},\psi \right>_{-1,1} := \sum_{p\in\mathfrak{N}^{sC}}\left< \widetilde{\hat{\Lambda}}_{\mathfrak{m}}^p,\psi\phi_p \right>_{-1,1} + \sum_{p\in\mathfrak{N}^{fC}}\left< \widetilde{\hat{\Lambda}}_{\mathfrak{m}}^p,\psi\phi_p \right>_{-1,1}.
\end{align}
For the definition of the local contributions, we abbreviate the element residual 
\begin{equation}\label{eq:ElementResidual}
r(\hat{\varphi}_{\mathfrak{m}}) :=\frac{G_c}{\epsilon} + G_c\epsilon\Delta \hat{\varphi}_{\mathfrak{m}} -  \frac{G_c}{\epsilon}\hat{\varphi}_{\mathfrak{m}}- (1-\kappa)(\tensor{\sigma}(\tensor{u}^n_{\mathfrak{m}}):\tensor{E}_{\mathrm{lin}}(\tensor{u}^n_{\mathfrak{m}}))\hat{\varphi}_{\mathfrak{m}}.
\end{equation}
For semi-contact nodes we consider the following local contribution in~\eqref{eq:QuasiDiscreteConstrainingForce}
\begin{align*}
\bigl< \widetilde{\hat{\Lambda}}_{\mathfrak{m}}^p &,\psi\phi_p \bigr>_{-1,1}
:= \bigl<\hat{\Lambda}_{\mathfrak{m}} ,\phi_p \bigr>_{-1,1}c_p(\psi)\\
& =\int_{\gamma_p^I}G_c\epsilon[\nabla \hat{\varphi}_{\mathfrak{m}}]c_p(\psi)\phi_p  - \int_{\gamma_p^\Gamma}(G_c\epsilon\nabla \hat{\varphi}_{\mathfrak{m}}\cdot\boldsymbol{n}_{\mathfrak{e}})c_p(\psi)\phi_p + \int_{\omega_p} r(\hat{\varphi}_{\mathfrak{m}})c_p(\psi)\phi_p 
\end{align*} 
with $c_p(\psi)= \frac{\int_{\widetilde{\omega}_p}\psi\phi_p}{\int_{\widetilde{\omega}_p}\phi_p}$, where $\widetilde{\omega}_p$ is the patch around $p$ with respect to a three times uniformly red-refined mesh.

For full-contact nodes we define the following local contribution in~\eqref{eq:QuasiDiscreteConstrainingForce}
\begin{align*}
\bigl< \widetilde{\hat{\Lambda}}_{\mathfrak{m}}^p&,\psi\phi_p \bigr>_{-1,1}
 :=  \bigl<\mathcal{R}^{lin}_{\mathfrak{m}}, \psi\phi_p\bigr>_{-1,1}\\
& :=\int_{\gamma_p^I}G_c\epsilon[\nabla
\hat{\varphi}_{\mathfrak{m}}]\psi\phi_p  -
\int_{\gamma_p^\Gamma}(G_c\epsilon\nabla
\hat{\varphi}_{\mathfrak{m}}\cdot\boldsymbol{n}_{\mathfrak{e}})\psi\phi_p
+ \int_{\omega_p} r(\hat{\varphi}_{\mathfrak{m}})\psi\phi_p. 
\end{align*}

With these definitions, we define the error measure 
\begin{equation}\label{eq:ErrorMeasure}
\|\hat{\varphi}-\hat{\varphi}_{\mathfrak{m}}\|_{\epsilon} + \|\hat{\Lambda}-\widetilde{\hat{\Lambda}}_{\mathfrak{m}}\|_{\ast,\epsilon}.
\end{equation}

\subsection{Error estimator}
In order to state the error estimator for the error measure~\eqref{eq:ErrorMeasure}, we define for each node $p$
\begin{equation}\label{eq:Def_alpha_p}
\alpha_p:= \mathrm{min}_{x\in\omega_p}\{\frac{G_c}{\epsilon}+(1-\kappa) (\boldsymbol{\sigma}(\boldsymbol{u}^n_{\mathfrak{m}}):\boldsymbol{E}_{\mathrm{lin}}(\boldsymbol{u}^n_{\mathfrak{m}}))\}
\end{equation}
and $h_p:=\mathrm{diam}(\omega_p)$. We note that for linear finite
elements on triangles the quantity $\left(\frac{G_c}{\epsilon}+(1-\kappa) (\boldsymbol{\sigma}^+(\boldsymbol{u}^n_{\mathfrak{m}}):\boldsymbol{E}_{\mathrm{lin}}(\boldsymbol{u}^n_{\mathfrak{m}}))\right)$ is constant on each element.
The error estimator 
\begin{equation}\label{eq:Def_Estimator}
\eta^{\varphi} := \sum_{k=1}^4\eta^{\varphi}_{k}
\end{equation}
for which we prove reliability and efficiency in Sections~\ref{Sec:Reliability} and~\ref{Sec:Efficiency} consists of the following local contributions
\begin{align*}
(\eta^{\varphi}_{1})^2:=&\sum_{p\in\mathfrak{N}\backslash\mathfrak{N}^{fC}}(\eta^{\varphi}_{1,p})^2, & \eta^{\varphi}_{1,p}:=&\mathrm{min}\{\frac{h_p}{\sqrt{G_c\epsilon}},\alpha_p^{-\frac{1}{2}}\}\|r(\hat{\varphi}_{\mathfrak{m}})\|_{\omega_p}\\ 
(\eta^{\varphi}_{2})^2:=&\sum_{p\in\mathfrak{N}\backslash\mathfrak{N}^{fC}}(\eta^{\varphi}_{2,p})^2, &\eta^{\varphi}_{2,p}:=& \mathrm{min}\{\frac{h_p}{\sqrt{G_c \epsilon}},\alpha_p^{-\frac{1}{2}}\}^{\frac{1}{2}}(G_c\epsilon)^{-\frac{1}{4}}\|G_c\epsilon[\nabla \hat{\varphi}_{\mathfrak{m}}]\|_{\gamma_p^I}\\ 
(\eta^{\varphi}_{3})^2:=&\sum_{p\in\mathfrak{N}\backslash\mathfrak{N}^{fC}}(\eta^{\varphi}_{3,p})^2, &\eta^{\varphi}_{3,p}:=&\mathrm{min}\{\frac{h_p}{\sqrt{G_c\epsilon}},\alpha_p^{-\frac{1}{2}}\}^{\frac{1}{2}}(G_c\epsilon)^{-\frac{1}{4}}\|G_c\epsilon\nabla \hat{\varphi}_{\mathfrak{m}}\cdot\boldsymbol{n}_{\mathfrak{e}}\|_{\gamma_p^\Gamma}\\  
(\eta^{\varphi}_{4})^2:=&\sum_{p\in\mathfrak{N}^{sC}}(\eta^{\varphi}_{4,p})^2, &\eta^{\varphi}_{4,p}:=&\left(s_p \int_{\widetilde{\omega}_p}(I_{\mathfrak{m}}^n(\varphi^{n-1}_{\mathfrak{m}})-\hat{\varphi}_{\mathfrak{m}})\phi_p\right)^{\frac{1}{2}}
\end{align*}
with $s_p := \frac{\left<\hat{\Lambda}_{\mathfrak{m}},\phi_p\right>_{-1,1}}{\int_{\omega_p}\phi_p}$.
We emphasize that the estimator contributions related to the constraints are localized to the area of semi-contact. In the absence of any contact, we have $\eta^{\varphi}_{k,p}=0$ for $k=4$ such that $\eta^{\varphi}$ reduces to a robust residual estimator, see, e.g.,~\cite{Verfuerth_1998b} 
for the prototype of a singularly perturbed reaction-diffusion equation.
\begin{remark}
If stress splitting of $\boldsymbol{\sigma}$ is used, the definitions of $r(\hat{\varphi}_{\mathfrak{m}})$ in~\eqref{eq:ElementResidual} and $\alpha_p$ in~\eqref{eq:Def_alpha_p} need to consider $\boldsymbol{\sigma}^+$, which thus enters into the error estimator.  
\end{remark}

In Section~\ref{Sec:Reliability}, we prove that $\eta^{\varphi}$ constitutes a robust upper bound where robust means that the constant in the bound does not depend on $\epsilon$ such that the validity of the estimator holds for arbitrary choices of $\epsilon$.
\begin{theorem}{\bf Reliability of the error estimator}\label{Theorem:UpperBound}\\
The error estimator $\eta^{\varphi}$ provides a robust upper bound of the error measure~\eqref{eq:ErrorMeasure}:
\begin{equation*}
\|\hat{\varphi}-\hat{\varphi}_{\mathfrak{m}}\|_{\epsilon} + \|\hat{\Lambda}-\widetilde{\hat{\Lambda}}_{\mathfrak{m}}\|_{\ast,\epsilon}\lesssim \eta^{\varphi}.
\end{equation*}
\end{theorem} 

In order to formulate the local lower bounds we denote by $\bar{r}(\hat{\varphi}_{\mathfrak{m}})$ a piecewise linear approximations of $r(\hat{\varphi}_{\mathfrak{m}})$ and we abbreviate $\mathrm{osc}_p(r):=\mathrm{min}\{\frac{h_p}{\sqrt{G_c\epsilon}},\alpha_{p}^{-\frac{1}{2}}\}\|\bar{r}(\hat{\varphi}_{\mathfrak{m}}) - r(\hat{\varphi}_{\mathfrak{m}}\|_{\omega_p}$.
In Section~\ref{Sec:Efficiency}, we derive the local lower bounds which are summarized in the following Theorems.
\begin{theorem}{\bf Local lower bounds by $\eta^{\varphi}_{1,p},\eta^{\varphi}_{2,p},\eta^{\varphi}_{3,p}$}\label{Theorem:LowerBound}\\
The error estimator contributions $\eta^{\varphi}_{k,p}$, $k=1,2,3$ constitute the following robust local lower bounds
\begin{equation*}
\eta^{\varphi}_{k,p}\lesssim\|\hat{\varphi}-\hat{\varphi}_{\mathfrak{m}}\|_{\epsilon,\omega_p} + \|\hat{\Lambda}-\widetilde{\hat{\Lambda}}_{\mathfrak{m}}\|_{\ast,\epsilon,\omega_p} + \mathrm{osc}_p(r).
\end{equation*}
\end{theorem}
To formulate the local lower bound by $\eta^{\varphi}_{4,p}$ we make use of the definition $\overline{\nabla|_{\mathfrak{e}} v_{\mathfrak{m}}}:= \nabla|_{\mathfrak{e}} v_{\mathfrak{m}}(\zeta_{\mathfrak{e}}) $ as a piecewise constant approximation of  $\nabla|_{\mathfrak{e}} v_{\mathfrak{m}}$ for $v_{\mathfrak{m}}\in \mathcal{H}_{\mathfrak{m}}$, where $\zeta_{\mathfrak{e}} \in \mathfrak{e}$ is a suitably chosen point that will be defined in the proof of Theorem~\ref{Theorem:LowerBound2}.
\begin{theorem}{\bf Local lower bound by $\eta^{\varphi}_{4,p}$}\label{Theorem:LowerBound2}\\
For nodes $p\in\mathfrak{N}^{sC}$ with $\frac{h_p}{\sqrt{G_c\epsilon}} \leq \alpha_p^{-\frac{1}{2}}$ we have the robust local lower bound
\begin{equation}\label{eq:LowerBound4_Robust}
\begin{split}
\eta^{\varphi}_{4,p} \lesssim&\; \|\hat{\varphi}-\hat{\varphi}_{\mathfrak{m}}\|_{\epsilon,\omega_p} + \|\hat{\Lambda}-\widetilde{\hat{\Lambda}}_{\mathfrak{m}}\|_{\ast,\epsilon,\omega_p} + \mathrm{osc}_p(r) \\
&\;  + \mathrm{min}\{\frac{h_p}{\sqrt{G_c\epsilon}},\alpha_p^{-\frac{1}{2}}\}^{\frac{1}{2}}(G_c\epsilon)^{-\frac{1}{4}}\|G_c\epsilon[\overline{\nabla (I_{\mathfrak{m}}^n(\varphi^{n-1}_{\mathfrak{m}})-\hat{\varphi}_{\mathfrak{m}})}]\|_{\gamma^I_p}\\
\end{split}
\end{equation}
Otherwise, for nodes $p\in\mathfrak{N}^{sC}$ with $\alpha_p^{-\frac{1}{2}}< \frac{h_p}{\sqrt{G_c\epsilon}}$
we have the local lower bound
\begin{equation}\label{eq:LowerBound4_NonRobust}
\begin{split}
\eta^{\varphi}_{4,p} \lesssim  &\;\|\hat{\varphi}-\hat{\varphi}_{\mathfrak{m}}\|_{\epsilon,\omega_p} + \|\hat{\Lambda}-\widetilde{\hat{\Lambda}}_{\mathfrak{m}}\|_{\ast,\epsilon,\omega_p}+ \mathrm{osc}_p(r)\\ &\; + \mathrm{max} \{\alpha_p(G_c\epsilon)^{-2}, \alpha_p^{\frac{1}{2}} (G_c\epsilon)^{-\frac{3}{2}}\}\|G_c\epsilon[\overline{\nabla (I_{\mathfrak{m}}^n(\varphi^{n-1}_{\mathfrak{m}})-\hat{\varphi}_{\mathfrak{m}})}]\|^2_{\gamma^I_p}.
\end{split}
\end{equation}
\end{theorem}
\begin{remark}
We note that the additional term in the bound~\eqref{eq:LowerBound4_Robust} only occurs for $p\in\mathfrak{N}^{sC}$ and is of the same order as the other estimator contributions. 
In the application, we expect the semi-contact zone to be well resolved, especially with respect to $\epsilon$ meaning $\frac{h_p}{\sqrt{G_c\epsilon}} \leq \alpha_p^{-\frac{1}{2}}$ 
after a finite number of adaptive refinement steps such that the local lower bound is robust everywhere. 
 \end{remark}

%% file: Reliability.tex
To derive the error estimator, we replace the linear residual which is
used in the derivation of a posteriori estimators for linear elliptic
equations by a so-called Galerkin functional which takes into account
the errors in both unknowns
\begin{equation}\label{eq:DefGalerkinFunctional}
 \begin{aligned}
 \bigl<\mathcal{G}_{\mathfrak{m}}&,\psi\bigr>_{-1,1} := a_{\mathfrak{m}, \epsilon}(\hat{\varphi}-\hat{\varphi}_{\mathfrak{m}},\psi) + \left<\hat{\Lambda}-\widetilde{\hat{\Lambda}}_{\mathfrak{m}},\psi\right>_{-1,1}\\ 
&=\left<\frac{G_c}{\epsilon},\psi\right> - a_{\mathfrak{m},\epsilon}( \hat{\varphi}_{\mathfrak{m}}, \psi) - \left< \widetilde{\hat{\Lambda}}_{\mathfrak{m}},\psi \right>_{-1,1}\\ 
&= \sum_{p\in\mathfrak{N}\backslash \mathfrak{N}^{fC}} \Biggl( \int_{\gamma_p^I}G_c\epsilon[\nabla \hat{\varphi}_{\mathfrak{m}}](\psi-c_p(\psi))\phi_p\\
&\qquad\qquad - \int_{\gamma_p^\Gamma}(G_c\epsilon\nabla \hat{\varphi}_{\mathfrak{m}}\cdot\boldsymbol{n}_{\mathfrak{e}})(\psi-c_p(\psi))\phi_p
 +\int_{\omega_p} r(\hat{\varphi}_{\mathfrak{m}})(\psi-c_p(\psi))\phi_p\Biggr). 
\end{aligned}
\end{equation}
Where the last equality is obtained as usual by
 utilizing Galerkin-orthogonality and element-wise integration by parts.

The relation between the dual norm of the Galerkin functional $\|\mathcal{G}_{\mathfrak{m}}\|_{\ast,\epsilon}$ and the error measure~\eqref{eq:ErrorMeasure} follows from
\begin{eqnarray}\label{eq:LowerBound0}
\|\mathcal{G}_{\mathfrak{m}}\|_{\ast,\epsilon}\lesssim \|\hat{\varphi}-\hat{\varphi}_{\mathfrak{m}}\|_{\epsilon}+\|\hat{\Lambda}-\widetilde{\hat{\Lambda}}_{\mathfrak{m}}\|_{\ast,\epsilon},
\end{eqnarray}
and
\begin{align}\label{eq:UpperBound0}
\|\hat{\varphi}-\hat{\varphi}_{\mathfrak{m}}\|^2_{\epsilon}\leq \|\mathcal{G}_{\mathfrak{m}}\|^2_{\ast,\epsilon} + 2\left< \widetilde{\hat{\Lambda}}_{\mathfrak{m}}-\hat{\Lambda},\hat{\varphi}-\hat{\varphi}_{\mathfrak{m}} \right>_{-1,1},
\end{align}
and
\begin{align}\label{eq:UpperBound1}
\|\hat{\Lambda}-\widetilde{\hat{\Lambda}}_{\mathfrak{m}}\|^2_{\ast,\epsilon}\leq 2\left(\|\mathcal{G}_{\mathfrak{m}}\|^2_{\ast,\epsilon}+ \|\hat{\varphi}-\hat{\varphi}_{\mathfrak{m}}\|^2_{\epsilon}\right),
\end{align}
compare~\cite[Lemma 3.4]{Veeser_2001}.

Based on the combination of~\eqref{eq:UpperBound0} and~\eqref{eq:UpperBound1}
\begin{align}\label{eq:UpperBoundAbstract}
\|\hat{\varphi}-\hat{\varphi}_{\mathfrak{m}}\|^2_{\epsilon}+\|\hat{\Lambda}-\widetilde{\hat{\Lambda}}_{\mathfrak{m}}\|^2_{\ast,\epsilon}\leq 5\|\mathcal{G}_{\mathfrak{m}}\|^2_{\ast,\epsilon} + 6\left< \widetilde{\hat{\Lambda}}_{\mathfrak{m}}-\hat{\Lambda},\hat{\varphi}-\hat{\varphi}_{\mathfrak{m}} \right>_{-1,1}
\end{align}
the reliability of the estimator follows from a computable upper bound of $\|\mathcal{G}_{\mathfrak{m}}\|^2_{\ast,\epsilon}$ and of $\left< \widetilde{\hat{\Lambda}}_{\mathfrak{m}}-\hat{\Lambda},\hat{\varphi}-\hat{\varphi}_{\mathfrak{m}} \right>_{-1,1}$. 
\begin{lemma}[Upper bound of Galerkin functional]\label{UpperBoundGalerkin}
The Galerkin functional defined in~\eqref{eq:DefGalerkinFunctional} satisfies 
\begin{equation*}
\|\mathcal{G}_{\mathfrak{m}}\|_{\ast,\epsilon}\lesssim \left(\sum_{k=1}^3(\eta^{\varphi}_k)^2\right)^{\frac{1}{2}}.
\end{equation*}
\end{lemma}
We will give the proof of Lemma~\ref{UpperBoundGalerkin} with the help of Lemma~\ref{LemmaL2Approx}. We use the same ideas as in~\cite{Walloth_2018b} but due to the different problem, discretization, and error measure some adaptations and comments are required.

Further, we make use of $h_p\approx h_{\mathfrak{e}}\approx h_{\mathfrak{s}}$ with $h_{\mathfrak{s}}=\mathrm{diam}(\omega_{\mathfrak{s}})$; which follows from the assumed shape regularity.
\begin{lemma}[$L^2$-approximation with respect to energy norm~\eqref{eq:EnergyNorm}]\label{LemmaL2Approx}
Let  $c_p(\psi)=\frac{\int_{\tilde{\omega}_p}\psi\phi_p}{\int_{\tilde{\omega}_p}\phi_p}$
with $\tilde{\omega}_p\subset\omega_p$ the patch around $p$ with respect to a three times uniformly red-refined mesh. Then the $L^2$-approximation properties with respect to the energy norm~\eqref{eq:EnergyNorm} hold
\begin{align}\label{eq:L2ApproxEnergyElement}
 \|(\psi-c_p(\psi))\phi_p\|_{\omega_p}&\lesssim\mathrm{min}\{\frac{h_p}{\sqrt{G_c\epsilon}},\alpha_p^{-\frac{1}{2}}\} \|\psi\|_{\epsilon,\omega_p}\\ \label{eq:L2ApproxEnergySide}
 \|(\psi-c_p(\psi))\phi_p\|_{\mathfrak{s}}&\lesssim\mathrm{min}\{\frac{h_p}{\sqrt{G_c\epsilon}},\alpha_p^{-\frac{1}{2}}\}^{\frac{1}{2}}(G_c\epsilon)^{-\frac{1}{4}}\|\psi\|_{\epsilon,\omega_{\mathfrak{s}}}.
 \end{align}
\end{lemma}
\begin{proof}
As in~\cite[Lemma~3]{Walloth_2018b}, we can derive 
\begin{align}\nonumber
\|\psi-c_p(\psi)\|_{\omega_p} \lesssim h_p \|\nabla\psi\|_{\omega_p} = \frac{h_p}{\sqrt{G_c\epsilon}}\sqrt{G_c\epsilon}\|\nabla\psi\|_{\omega_p}.
\end{align}
Using the definition of $\alpha_p$ in \eqref{eq:Def_alpha_p} it also holds
\begin{align*}
\|\psi-c_p(\psi)\|_{\omega_p}\lesssim \|\psi\|_{\omega_p} \lesssim \alpha_p^{-\frac{1}{2}}\|\left(\frac{G_c}{\epsilon}+(1-\kappa)\tensor{\sigma}(\tensor{u}^n_{\mathfrak{m}}):\tensor{E}_{\mathrm{lin}}(\tensor{u}^n_{\mathfrak{m}})\right)^{\frac{1}{2}}\psi\|_{\omega_p}.
\end{align*}
Together, we deduce the $L^2$-approximation property with respect to the energy norm
\begin{align*}
\|\psi-c_p(\psi)\|_{\omega_p} \lesssim \mathrm{min}\{\frac{h_p}{\sqrt{G_c\epsilon}},\alpha_p^{-\frac{1}{2}}\}\|\psi\|_{\epsilon,\omega_p}. 
\end{align*}

It remains to derive the $L^2$- approximation property for sides $\mathfrak{s}$. The result~\cite[Lemma 3.2]{Verfuerth_1998b} can be extended to bilinear finite elements on parallelograms. Thus, we have 
 \begin{align*}
\|(\psi-c_p(\psi))\phi_p\|_{\mathfrak{s}}
\leq\|(\psi-c_p(\psi))\phi_p\|_{0,\omega_{\mathfrak{s}}} ^{\frac{1}{2}}\|\nabla((\psi-c_p(\psi))\phi_p)\|^{\frac{1}{2}}_{0,\omega_{\mathfrak{s}}} 
\end{align*}
We can further proceed as in~\cite[Lemma~3]{Walloth_2018b}. 
We apply the product rule and triangle inequality 
\begin{align*}
\|\nabla((\psi-c_p(\psi))\phi_p)\|_{\mathfrak{e}}&\leq\|\nabla(\psi-c_p(\psi))\phi_p\|_{\mathfrak{e}} + \|(\psi-c_p(\psi))\nabla\phi_p\|_{\mathfrak{e}}\\
&\lesssim \|\nabla(\psi-c_p(\psi))\|_{\mathfrak{e}} + h_{\mathfrak{e}}^{-\frac{1}{2}}\|(\psi-c_p(\psi))\|_{\mathfrak{e}}.
\end{align*} 
Next, we apply the $L^2$-approximation property~\eqref{eq:L2ApproxEnergyElement} on the elements and $\|\nabla\psi\|_{\omega_\mathfrak{s}}\leq \frac{1}{\sqrt{G_c\epsilon}}\|\psi\|_{\epsilon,\omega_{\mathfrak{s}}}$ to get the $L^2$-approximation property on the sides
 \begin{align*}
&\|(\psi-c_p(\psi))\phi_p\|_{\mathfrak{s}}\\
&\leq\|(\psi-c_p(\psi))\phi_p\|_{0,\omega_{\mathfrak{s}}} ^{\frac{1}{2}}\|\nabla((\psi-c_p(\psi))\phi_p)\|^{\frac{1}{2}}_{0,\omega_{\mathfrak{s}}} \\
&\lesssim h_p^{-\frac{1}{2}}\|(\psi-c_p(\psi))\|_{0,\omega_{\mathfrak{s}}} + \|(\psi-c_p(\psi))\phi_p\|^{\frac{1}{2}}_{0,\omega_{\mathfrak{s}}}\|\nabla((\psi-c_p(\psi))) \|^{\frac{1}{2}}_{0,\omega_{\mathfrak{s}}}\\
&\lesssim h_p^{-\frac{1}{2}}\mathrm{min}\{\frac{h_p}{\sqrt{G_c\epsilon}}, \alpha_p^{-\frac{1}{2}}\}\|\psi\|_{\epsilon,\omega_{\mathfrak{s}}} +\mathrm{min}\{\frac{h_p}{\sqrt{G_c \epsilon}}, \alpha_p^{-\frac{1}{2}}\}^{\frac{1}{2}} \|\psi\|^{\frac{1}{2}}_{\epsilon,\omega_{\mathfrak{s}}}(G_c\epsilon)^{-\frac{1}{4}}\|\psi\|^{\frac{1}{2}}_{\epsilon,\omega_{\mathfrak{s}}}\\
&\lesssim \left(\mathrm{min}\{\frac{h_p}{\sqrt{G_c\epsilon}},\alpha_p^{-\frac{1}{2}}\}^{\frac{1}{2}} \mathrm{min}\{\frac{1}{\sqrt{G_c\epsilon}}, \frac{\alpha_p^{-\frac{1}{2}}}{h_p}\}^{\frac{1}{2}} + \mathrm{min}\{\frac{h_p}{\sqrt{G_c\epsilon}}, \alpha_p^{-\frac{1}{2}}\}^{\frac{1}{2}} (G_c\epsilon)^{-\frac{1}{4}}\right)\|\psi\|_{\epsilon,\omega_{\mathfrak{s}}}\\
&\lesssim\mathrm{min}\{\frac{h_p}{\sqrt{G_c\epsilon}}, \alpha_p^{-\frac{1}{2}}\}^{\frac{1}{2}}(G_c\epsilon)^{-\frac{1}{4}}\|\psi\|_{\epsilon,\omega_{\mathfrak{s}}}
\end{align*}
\end{proof}
Together with these preliminary results, we can give the proof of Lemma~\ref{UpperBoundGalerkin}.
\begin{proof}[Proof of Lemma~\ref{UpperBoundGalerkin}] 
In order to derive an upper bound of the dual norm of the Galerkin functional, we use the representation~\eqref{eq:DefGalerkinFunctional} and Cauchy-Schwarz inequality 
\begin{equation}\label{eq:BoundGalerkin}
\begin{aligned}
 \left<\mathcal{G}_{\mathfrak{m}}, \psi\right>
&\leq
\sum_{p\in\mathfrak{N}\backslash \mathfrak{N}^{fC}} \Bigl( \|G_c\epsilon[\nabla \hat{\varphi}_{\mathfrak{m}}]\|_{\gamma_p^I}\|(\psi-c_p(\psi))\phi_p\|_{\gamma_p^I} \\
&\qquad\qquad\qquad +\|G_c\epsilon\nabla \hat{\varphi}_{\mathfrak{m}}\cdot\boldsymbol{n}_{\mathfrak{e}}\|_{\gamma_p^\Gamma}\|(\psi-c_p(\psi))\phi_p\|_{\gamma_p^\Gamma}\\ 
&\qquad\qquad\qquad +\|r(\hat{\varphi}_{\mathfrak{m}})\|_{\omega_p}\|(\psi-c_p(\psi))\phi_p\|_{\omega_p}\Bigr).
\end{aligned}
\end{equation}
Combining~\eqref{eq:BoundGalerkin},~\eqref{eq:L2ApproxEnergyElement}, and~\eqref{eq:L2ApproxEnergySide}, we get 
\begin{align*}
\left<\mathcal{G}_{\mathfrak{m}}, \psi \right>_{-1,1}&\lesssim \Biggl(\sum_{p\in\mathfrak{N}\backslash  \mathfrak{N}^{fC}}\left( \mathrm{min}\{\frac{h_p}{\sqrt{G_c\epsilon}},\alpha_p^{-\frac{1}{2}}\}^{\frac{1}{2}}(G_c\epsilon)^{-\frac{1}{4}}\|\epsilon[\nabla \hat{\varphi}_{\mathfrak{m}}]\|_{\gamma_p^I} \right. \\
\nonumber
&\qquad\qquad\qquad \left.\left.+\mathrm{min}\{\frac{h_p}{\sqrt{G_c\epsilon}},\alpha_p^{-\frac{1}{2}}\}^{\frac{1}{2}}(G_c\epsilon)^{-\frac{1}{4}}\|\epsilon\nabla \hat{\varphi}_{\mathfrak{m}}\cdot\boldsymbol{n}_{\mathfrak{e}}\|_{\gamma_p^\Gamma}\right.\right.\\ 
&\qquad\qquad\qquad\left.+\mathrm{min}\{\frac{h_p}{\sqrt{G_c\epsilon}},\alpha_p^{-\frac{1}{2}}\}\|r(\hat{\varphi}_{\mathfrak{m}})\|_{\omega_p}\right)^2\Biggr)^{\frac{1}{2}}\Biggl(\sum_{p\in\mathfrak{N}}\|\psi\|^2_{\epsilon,\omega_p}\Biggr)^{\frac{1}{2}}
\end{align*}
and thus the bound of the dual norm of the Galerkin functional 
\begin{align}\nonumber
\|\mathcal{G}_{\mathfrak{m}}\|_{\ast,\epsilon}&= \frac{\mathrm{sup}_{\psi\in H^1}\left<\mathcal{G}_{\mathfrak{m}}, \psi \right>_{-1,1}}{\|\psi\|_{\epsilon}}\lesssim \sum_{k=1}^3\eta^{\varphi}_{k}.
\end{align}
\end{proof}

\begin{lemma}[Complementarity residual]\label{LemmaComplementarityResidual}
It holds 
\begin{align}\nonumber
 \left< \tilde{\hat{\Lambda}}_{\mathfrak{m}}-\hat{\Lambda},\hat{\varphi}-\hat{\varphi}_{\mathfrak{m}} \right>_{-1,1}\lesssim  (\eta^{\varphi}_4)^2.
 \end{align}
 \end{lemma}
 Due to the discretization by bilinear finite elements on parallelograms and linear finite elements on triangles $\mathcal{K}^n_{\mathfrak{m}}\subset\mathcal{K}(I_{\mathfrak{m}}^n(\varphi_{\mathfrak{m}}^{n-1}))$ holds. Thus, for the proof of Lemma~\ref{LemmaComplementarityResidual} we refer to~\cite[Lemma~4]{Walloth_2018b}.

Theorem~\ref{Theorem:UpperBound} follows from Lemma~\ref{UpperBoundGalerkin} and Lemma~\ref{LemmaComplementarityResidual}. 

%% file: Efficiency.tex
This Section provides the proofs of Theorem~\ref{Theorem:LowerBound} and~\ref{Theorem:LowerBound2}.

\subsection{Local error bound by $\eta^{\varphi}_{1,p},\eta^{\varphi}_{2,p},\eta^{\varphi}_{3,p}$}
We start with $\eta^{\varphi}_{1,p}$ for which we use the properties of the element bubble functions $\Psi_{\mathfrak{e}}:= c\Pi_{p\in\mathfrak{e}}\phi_p$,  for triangles and parallelograms, see~\cite[Chapter 1.3.4]{Verfuerth_2013}:
\begin{itemize}
\item $0\leq \Psi_{\mathfrak{e}}\leq 1$
\item $\|\nabla (\Psi_{\mathfrak{e}} v)\|_{\mathfrak{e}}\lesssim h_{\mathfrak{e}}^{-1}\|v\|_{\mathfrak{e}}$ for all polynomials $v$
\end{itemize}

Similar to~\eqref{eq:Def_alpha_p}, we define for each element $\mathfrak{e}$
\begin{equation}
\alpha_{\mathfrak{e}}:=\mathrm{max}_{x\in\mathfrak{e}}\left\{\frac{G_c}{\epsilon}+(1-\kappa) (\boldsymbol{\sigma}(\boldsymbol{u}^n_{\mathfrak{m}}):\boldsymbol{E}_{\mathrm{lin}}(\boldsymbol{u}^n_{\mathfrak{m}}))\right\}
\end{equation}
We note that $\frac{G_c}{\epsilon}+(1-\kappa) (\boldsymbol{\sigma}(\boldsymbol{u}^n_{\mathfrak{m}}):\boldsymbol{E}_{\mathrm{lin}}(\boldsymbol{u}^n_{\mathfrak{m}}))|_{\mathfrak{e}}$ is constant if $\mathfrak{e}$ is a triangle.
With respect to the energy norm~\eqref{eq:EnergyNorm} this implies for all polynomials $v$
\begin{align}\label{eq:WeightedInverseInequality}
\|\Psi_{\mathfrak{e}}v\|_{\epsilon,\mathfrak{e}}\lesssim (\sqrt{G_c\epsilon} h_{\mathfrak{e}}^{-1}+\alpha^{\frac{1}{2}}_{\mathfrak{e}})\|v\|_{\mathfrak{e}}\lesssim \mathrm{max}\{\frac{\sqrt{G_c\epsilon}}{h_{\mathfrak{e}}},\alpha^{\frac{1}{2}}_{\mathfrak{e}}\}\|v\|_{\mathfrak{e}}.
\end{align}
We recall that for all $p\in\mathfrak{N}$,  $\tilde{\omega}_p$ is the patch around $p$ with respect to a three times uniformly red-refined mesh $\widetilde{\mathfrak{M}}$ with $\tilde{\mathfrak{e}}\in\widetilde{\mathfrak{M}}$ and $h_{\tilde{\mathfrak{e}}}= ch_{\mathfrak{e}}$. We define a linear combination of element bubble functions $\Psi_j$ with respect to all elements $\tilde{\mathfrak{e}}_j\subset \mathfrak{e}$, i.e., $\theta_{\mathfrak{e}} = \sum_{j=1} a_j\Psi_j$. Taking $a_j=0$ for all elements $\tilde{\mathfrak{e}}_j$ containing a node $p\in\mathfrak{N}^{sC}$, we can assert
\begin{align} \label{eq:MeanValueZero1}
\int_{\tilde{\mathfrak{e}}_j}\phi_{q}\theta_{\mathfrak{e}}\phi_{p} = 0\quad\forall q\in\mathfrak{e}.
\end{align}

The other coefficients of the linear combination are chosen such that the bubble function $\theta_{\mathfrak{e}}$ fulfills the following conditions
\begin{align} \label{eq:CondBubbleFunc}
\int_{\mathfrak{e}}\phi_q\phi_r&=\sum_{p\in\mathfrak{N}\backslash\mathfrak{N}^{fC}}\int_{\mathfrak{e}} \phi_q\phi_r\theta_{\mathfrak{e}}\phi_p\quad \forall q,r\in \mathfrak{e}
\end{align}
 As we have more degrees of freedom (coefficients $a_j$) than conditions 
 \begin{itemize}
\item three for~\eqref{eq:MeanValueZero1} on a triangle and four for~\eqref{eq:MeanValueZero1} on a parallelogram
\item six for~\eqref{eq:CondBubbleFunc} on a triangle and ten for~\eqref{eq:CondBubbleFunc} on a parallelogram
 \end{itemize}
 the construction of a suitable bubble function is possible.

In the following, we make use of the fact that $\bar{r}(\hat{\varphi}_{\mathfrak{m}})$ is a linear finite element function such that~\eqref{eq:MeanValueZero1} implies $c_p(r(\hat{\varphi}_{\mathfrak{m}})\theta_{\mathfrak{e}})=0$. 
Further, we exploit~\eqref{eq:CondBubbleFunc} and that $\theta_{\mathfrak{e}}$ vanishes on the edges.
Thus, exploiting~\eqref{eq:WeightedInverseInequality} for $\theta_{\mathfrak{e}}$ instead of $\Psi_{\mathfrak{e}}$,
\begin{align*}
&\|\bar{r}(\hat{\varphi}_{\mathfrak{m}})\|_{\mathfrak{e}}^2\\ 
&\lesssim \sum_{p\in\mathfrak{N}\backslash\mathfrak{N}^{fC}}\int_{\mathfrak{e}}(\bar{r}(\hat{\varphi}_{\mathfrak{m}}))(r(\hat{\varphi}_{\mathfrak{m}}))\theta_{\mathfrak{e}}\phi_p +\sum_{p\in\mathfrak{N}\backslash\mathfrak{N}^{fC}}\int_{\mathfrak{e}}\left(\bar{r}(\hat{\varphi}_{\mathfrak{m}}) - r(\hat{\varphi}_{\mathfrak{m}})\right) \bar{r}(\hat{\varphi}_{\mathfrak{m}})\theta_{\mathfrak{e}}\phi_p \\ 
& = \left<\mathcal{G}_{\mathfrak{m}},\bar{r}(\hat{\varphi}_{\mathfrak{m}})\theta_{\mathfrak{e}}\right> - \sum_{p\in\mathfrak{N}\backslash\mathfrak{N}^{fC}}\int_{\gamma^I_p}G_c\epsilon[\nabla \hat{\varphi}_{\mathfrak{m}}]\bar{r}(\hat{\varphi}_{\mathfrak{m}})\theta_{\mathfrak{e}}\phi_p   \\ 
&\qquad + \sum_{p\in\mathfrak{N}\backslash\mathfrak{N}^{fC}}\left<\hat{\Lambda}_{\mathfrak{m}} ,\phi_p \right>_{-1,1}c_p(\bar{r}(\hat{\varphi}_{\mathfrak{m}})\theta_{\mathfrak{e}}) + \sum_{p\in\mathfrak{N}\backslash\mathfrak{N}^{fC}}\int_{\mathfrak{e}}\left(\bar{r}(\hat{\varphi}_{\mathfrak{m}}) - r(\hat{\varphi}_{\mathfrak{m}})\right) \bar{r}(\hat{\varphi}_{\mathfrak{m}})\theta_{\mathfrak{e}}\phi_p\\
& \lesssim \|\mathcal{G}_{\mathfrak{m}}\|_{\ast,\epsilon,\omega_p} \|\bar{r}(\hat{\varphi}_{\mathfrak{m}})\theta_{\mathfrak{e}}\|_{\epsilon,\mathfrak{e}}  + \|\bar{r}(\hat{\varphi}_{\mathfrak{m}}) - r(\hat{\varphi}_{\mathfrak{m}} \|_{\mathfrak{e}}\|\bar{r}(\hat{\varphi}_{\mathfrak{m}})\|_{\mathfrak{e}} \\ 
&\lesssim \|\mathcal{G}_{\mathfrak{m}}\|_{\ast,\epsilon,\omega_p}\mathrm{max}\{\frac{\sqrt{G_c\epsilon}}{h_{\mathfrak{e}}},\alpha^{\frac{1}{2}}_{\mathfrak{e}}\}\|\bar{r}(\hat{\varphi}_{\mathfrak{m}})\|_{\mathfrak{e}} + \|\bar{r}(\hat{\varphi}_{\mathfrak{m}}) - r(\hat{\varphi}_{\mathfrak{m}} \|_{\mathfrak{e}}\|\bar{r}(\hat{\varphi}_{\mathfrak{m}})\|_{\mathfrak{e}}.
\end{align*}
Dividing by $\mathrm{max}\{\frac{\sqrt{G_c\epsilon}}{h_{\mathfrak{e}}},\alpha^{\frac{1}{2}}_{\mathfrak{e}}\}\|\bar{r}(\hat{\varphi}_{\mathfrak{m}})\|_{\mathfrak{e}}$ and as  $\mathrm{min}\{\frac{h_p}{\sqrt{G_c\epsilon}},\alpha_{\mathfrak{e}}^{-\frac{1}{2}}\} = \left(\mathrm{max}\{\frac{\sqrt{G_c\epsilon}}{h_p},\alpha_{\mathfrak{e}}^{\frac{1}{2}}\}\right)^{-1} $ and $\frac{\alpha_e}{\alpha_p} = C_{\mathfrak{e},p}\neq 0$ is a computable constant,
we arrive at
\begin{align}\label{eq:LowerBound1}
\eta^{\varphi}_{1,p} =\mathrm{min}\{\frac{h_p}{\sqrt{G_c\epsilon}},\alpha_{p}^{-\frac{1}{2}}\} &\|r(\hat{\varphi}_{\mathfrak{m}})\|_{\omega_p} 
\lesssim \|\mathcal{G}_{\mathfrak{m}}\|_{\ast,\epsilon,\omega_p} +\mathrm{osc}_p(r)\ \\ \label{eq:LowerBoundEta1}
&\lesssim \|\hat{\varphi}-\hat{\varphi}_{\mathfrak{m}}\|_{\epsilon,\omega_p} + \|\hat{\Lambda}-\widetilde{\hat{\Lambda}}_{\mathfrak{m}}\|_{\ast,\epsilon,\omega_p} +\mathrm{osc}_p(r).
\end{align}
We note that $C_{\mathfrak{e},p}\in[1,1+\frac{\mathrm{max}_{x\in\mathfrak{e}}(1-\kappa)\sigma(\boldsymbol{u}_{\mathfrak{m}}^n):E_{\mathrm{lin}}(\boldsymbol{u}_{\mathfrak{m}}^n) - \mathrm{min}_{x\in\omega_p}(1-\kappa)\sigma(\boldsymbol{u}_{\mathfrak{m}}^n):E_{\mathrm{lin}}(\boldsymbol{u}_{\mathfrak{m}}^n)}{G_c}]$.

In order to prove the lower bound in terms of $\eta^{\varphi}_{2,p}$, we use the properties of side bubble functions. Following the ansatz given in~\cite{Verfuerth_1998b}, we define side bubble functions with the help of basis functions belonging to a modified element. On the reference element $\hat{\mathfrak{e}}$ the corresponding transformation $\Phi_{\delta}:\mathbb{R}^2\rightarrow\mathbb{R}^2$ maps the coordinates $x, y$ to $x,\delta y$ with $\delta\in(0,1]$.  The basis functions on the transformed reference element are given by $\hat{\phi}_{\delta,p} := \hat{\phi}_{p}\circ\Phi^{-1}_{\delta}$ on $\Phi_{\delta}(\hat{\mathfrak{e}})$ and $\hat{\phi}_{\delta,p} = 0$ on $\hat{\mathfrak{e}}\backslash \Phi_{\delta}(\hat{\mathfrak{e}})$. 
Let $F_{\mathfrak{s}}:\mathfrak{e}\rightarrow\hat{\mathfrak{e}}$ be the linear transformation which maps $\mathfrak{s}$ on $\hat{\mathfrak{s}}$ which is the side with the nodes $p_0= (0,0)$ and $p_1=(1,0)$. The modified side bubble function is defined by $\Psi_{\delta,\hat{\mathfrak{s}}} := \Pi_{p\in\hat{\mathfrak{s}}}\hat{\phi}_{\delta,p}$.
Then it follows from~\cite[Lemma 3.4]{Verfuerth_1998b} together with the transformation rule 
\begin{equation}\label{eq:PropModBubbFunc_Trans}
\begin{split}
\|\Psi_{\delta,\mathfrak{s}}w\|_{\mathfrak{e}}&\lesssim h_{\mathfrak{e}}^{\frac{1}{2}}\sqrt{\delta}\|w\|_{\mathfrak{s}},\\
\|\frac{\partial}{\partial x_i}(\Psi_{\delta,\mathfrak{s}}w)\|_{\mathfrak{e}}&\lesssim h_{\mathfrak{e}}^{-\frac{1}{2}} \sqrt{\delta}\|w\|_{\mathfrak{s}},\quad\forall 1\leq i\leq n-1\\
\|\frac{\partial}{\partial x_n}(\Psi_{\delta,\mathfrak{s}}w)\|_{\mathfrak{e}}&\lesssim h_{\mathfrak{e}}^{-\frac{1}{2}} \frac{1}{\sqrt{\delta}}\|w\|_{\mathfrak{s}}.
\end{split}
\end{equation}
With respect to the $\|\cdot\|_{\epsilon}$ norm, we get
\begin{equation}\label{eq:WeightedInverseInequalitySide}
\|\Psi_{\delta,\mathfrak{s}}w\|_{\epsilon,\omega_{\mathfrak{s}}}\lesssim \left(\sqrt{G_c\epsilon} h_{\mathfrak{s}}^{-\frac{1}{2}}\delta^{-\frac{1}{2}} + \alpha_{\mathfrak{s}}^{\frac{1}{2}} h_{\mathfrak{s}}^{\frac{1}{2}}\delta^{\frac{1}{2}}\right)\|w \|_{\mathfrak{s}}.
\end{equation} 
where $\alpha_{\mathfrak{s}}:=\mathrm{max}_{\tilde{\mathfrak{e}}\subset\omega_{\mathfrak{s}}} \alpha_{\tilde{\mathfrak{e}}}$.
Similar to the proof of the lower bound in terms of $\eta^{\varphi}_{1,p}$, we consider a partition of $\mathfrak{s}$ by three uniform refinements. 
We construct a linear combination $\theta_{\delta,\mathfrak{s}} = \sum_{j} a_j\Psi_{\delta,\tilde{\mathfrak{s}}_j}$ of modified side bubble functions $\Psi_{\delta,\tilde{\mathfrak{s}}_j}$ with respect to all sides $\tilde{\mathfrak{s}}_j$ of the partition of $\mathfrak{s}$ such that $c_p([\nabla\hat{\varphi}_{\mathfrak{m}}]\theta_{\delta,\mathfrak{s}})=0$. We choose $a_j=0$ for all sides $\tilde{\mathfrak{s}}_j$ containing a node $p\in\mathfrak{N}^{sC}$ such that for  triangles and $p\in\tilde{\mathfrak{s}}_j$
\begin{equation} \label{eq:MeanValueZeroTriangle}
\int_{\tilde{\mathfrak{e}}_j}\theta_{ \delta,\mathfrak{s}}\phi_{p} = 0 
\end{equation}
and for parallelograms and $p\in\tilde{\mathfrak{s}}_j$
\begin{equation} \label{eq:MeanValueZeroRectangle}
\int_{\tilde{\mathfrak{e}}_j}\phi_q\theta_{ \delta,\mathfrak{s}}\phi_{p} = 0\quad\forall q\in\mathfrak{e}. 
\end{equation}
The other coefficients of the linear combination are chosen such that the bubble function $\theta_{\delta,\mathfrak{s}}$ fulfills the following property for triangles
\begin{equation} \label{eq:RelNormIntTriangle}
\int_{\mathfrak{s}}1=\sum_{p\in\mathfrak{N}\backslash\mathfrak{N}^{fC}}\int_{\mathfrak{s}} \theta_{\delta,\mathfrak{s}}\phi_p
\end{equation}
and the following property for parallelograms
\begin{equation} \label{eq:RelNormIntRectangle}
\int_{\mathfrak{s}}\phi_q\phi_r=\sum_{p\in\mathfrak{N}\backslash\mathfrak{N}^{fC}}\int_{\mathfrak{s}} \phi_q\phi_r\theta_{\delta,\mathfrak{s}}\phi_p\quad \forall q,r\in \mathfrak{s}.
\end{equation}
Again, as we have more degrees of freedom (coefficients $a_j$) than conditions 
\begin{itemize}
\item one for~\eqref{eq:MeanValueZeroTriangle} on a triangle and two for~\eqref{eq:MeanValueZeroRectangle} on a parallelogram
\item one for~\eqref{eq:RelNormIntTriangle} on a triangle and three for~\eqref{eq:RelNormIntRectangle} on a parallelogram
 \end{itemize}
 the construction of a suitable bubble function is possible.

We set $w:= G_c\epsilon[\nabla \hat{\varphi}_{\mathfrak{m}}]$. 
Thus, we apply~\eqref{eq:PropModBubbFunc_Trans},~\eqref{eq:WeightedInverseInequalitySide}. Together with~\eqref{eq:MeanValueZeroTriangle},~\eqref{eq:MeanValueZeroRectangle} and~\eqref{eq:RelNormIntTriangle},~\eqref{eq:RelNormIntRectangle}, we get
\begin{equation}\label{eq:LowerBoundJumpStep1}
\begin{aligned}
\|G_c\epsilon[\nabla \hat{\varphi}_{\mathfrak{m}}]\|_{\mathfrak{s}}^2 &= \sum_{p\in\mathfrak{N}\backslash\mathfrak{N}^{fC}}\int_{\mathfrak{s}} G_c^2\epsilon^2[\nabla \hat{\varphi}_{\mathfrak{m}}] [\nabla \hat{\varphi}_{\mathfrak{m}}]\theta_{\delta,\mathfrak{s}}\phi_p\\ 
&\lesssim \left<\mathcal{G}_{\mathfrak{m}}, G_c\epsilon[\nabla \hat{\varphi}_{\mathfrak{m}}]\theta_{\delta,\mathfrak{s}}\right> + \sum_{p\in\mathfrak{N}\backslash\mathfrak{N}^{fC}}\int_{\omega_{\mathfrak{s}}}r(\hat{\varphi}_{\mathfrak{m}})G_c\epsilon[\nabla \hat{\varphi}_{\mathfrak{m}}]\theta_{\delta,\mathfrak{s}}\phi_p\\ 
&\qquad\quad+ \sum_{p\in\mathfrak{N}\backslash\mathfrak{N}^{fC}}\left<\widetilde{\hat{\Lambda}}_{\mathfrak{m}} ,\phi_p \right>_{-1,1}c_p(G_c\epsilon [\nabla \hat{\varphi}_{\mathfrak{m}}] \theta_{\delta,\mathfrak{s}})\\ 
&\lesssim \|\mathcal{G}_{\mathfrak{m}}\|_{\ast,\epsilon,\omega_p}\|G_c\epsilon[\nabla \hat{\varphi}_{\mathfrak{m}}]\theta_{\delta,\mathfrak{s}}\|_{\epsilon, \omega_{\mathfrak{s}}} + \|r(\hat{\varphi}_{\mathfrak{m}})\|_{\omega_{\mathfrak{s}}} \|G_c\epsilon[\nabla \hat{\varphi}_{\mathfrak{m}}]\theta_{\delta,\mathfrak{s}}\|_{\omega_{\mathfrak{s}}}\\ 
&\lesssim \|\mathcal{G}_{\mathfrak{m}}\|_{\ast,\epsilon,\omega_p}  \left(\sqrt{G_c\epsilon} h_{\mathfrak{s}}^{-\frac{1}{2}}\delta^{-\frac{1}{2}} + \alpha_{\mathfrak{s}}^{\frac{1}{2}} h_{\mathfrak{s}}^{\frac{1}{2}}\delta^{\frac{1}{2}}\right)\|G_c\epsilon[\nabla \hat{\varphi}_{\mathfrak{m}}]\|_{\mathfrak{s}}\\ 
&\qquad\quad  + \delta^{\frac{1}{2}}h_{\mathfrak{s}}^{\frac{1}{2}}\|r(\hat{\varphi}_{\mathfrak{m}})\|_{\omega_{\mathfrak{s}}}  \|G_c\epsilon[\nabla \hat{\varphi}_{\mathfrak{m}}]\|_{\mathfrak{s}} 
\end{aligned}
\end{equation}
Choosing $\delta := \alpha_{\mathfrak{s}}^{-\frac{1}{2}}\mathrm{min}\{\frac{\sqrt{G_c\epsilon}}{h_{\mathfrak{s}}},\alpha_{\mathfrak{s}}^{\frac{1}{2}}\}<1$, we get the first factor 
\begin{align*}
&\sqrt{G_c\epsilon} h_{\mathfrak{s}}^{-\frac{1}{2}}\alpha_{\mathfrak{s}}^{\frac{1}{4}}\mathrm{min}\{\frac{\sqrt{G_c\epsilon}}{h_{\mathfrak{s}}},\alpha_{\mathfrak{s}}^{\frac{1}{2}}\}^{-\frac{1}{2}} + \alpha_{\mathfrak{s}}^{\frac{1}{2}}h_{\mathfrak{s}}^{\frac{1}{2}}\alpha_{\mathfrak{s}}^{-\frac{1}{4}}\mathrm{min}\{\frac{\sqrt{G_c\epsilon}}{h_{\mathfrak{s}}},\alpha_{\mathfrak{s}}^{\frac{1}{2}}\}^{\frac{1}{2}}\\
 &= (G_c\epsilon)^{\frac{1}{4}}\mathrm{min}\{\alpha_{\mathfrak{s}}^{-\frac{1}{2}}, \frac{h_{\mathfrak{s}}}{\sqrt{\epsilon G_c}}\}^{-\frac{1}{2}} + \alpha^{\frac{1}{2}}_{\mathfrak{s}} (G_c\epsilon)^{\frac{1}{4}}\mathrm{min}\{\alpha_{\mathfrak{s}}^{-\frac{1}{2}}, \frac{h_{\mathfrak{s}}}{\sqrt{\epsilon G_c}}\}^{\frac{1}{2}}\\
 &\leq(G_c\epsilon)^{\frac{1}{4}}\mathrm{min}\{\alpha_{\mathfrak{s}}^{-\frac{1}{2}}, \frac{h_{\mathfrak{s}}}{\sqrt{\epsilon G_c}}\}^{-\frac{1}{2}} + (G_c\epsilon)^{\frac{1}{4}}\mathrm{min}\{\alpha^{-\frac{1}{2}}, \frac{h_{\mathfrak{s}}}{\sqrt{\epsilon G_c}}\}^{-\frac{1}{2}} \underbrace{\alpha_{\mathfrak{s}}^{\frac{1}{2}}\mathrm{min}\{\alpha^{-\frac{1}{2}}, \frac{h_{\mathfrak{s}}}{\sqrt{\epsilon G_c}}\}}_{\leq 1} \\
 &\lesssim (G_c\epsilon)^{\frac{1}{4}}\mathrm{min}\{\alpha_{\mathfrak{s}}^{-\frac{1}{2}}, \frac{h_{\mathfrak{s}}}{\sqrt{\epsilon G_c}}\}^{-\frac{1}{2}}
\end{align*}
and the second factor
\begin{align*}
\alpha_{\mathfrak{s}}^{-\frac{1}{4}}\mathrm{min}\{\frac{\sqrt{G_c\epsilon}}{h_{\mathfrak{s}}},\alpha_{\mathfrak{s}}^{\frac{1}{2}}\}^{\frac{1}{2}} h_{\mathfrak{s}}^{\frac{1}{2}} 
&=  \mathrm{min}\{\frac{h_{\mathfrak{s}}}{\sqrt{G_c\epsilon}}, \alpha^{-\frac{1}{2}}_{\mathfrak{s}}\}^{\frac{1}{2}} (\epsilon G_c)^{\frac{1}{4}}\\
&=  \mathrm{min}\{\frac{h_{\mathfrak{s}}}{\sqrt{G_c\epsilon}}, \alpha^{-\frac{1}{2}}_{\mathfrak{s}}\}(\epsilon G_c)^{\frac{1}{4}}  \mathrm{min}\{\frac{h_{\mathfrak{s}}}{\sqrt{G_c\epsilon}}, \alpha^{-\frac{1}{2}}_{\mathfrak{s}}\}^{-\frac{1}{2}}
\end{align*}
Thus, dividing~\eqref{eq:LowerBoundJumpStep1} by $(G_c\epsilon)^{\frac{1}{4}}\mathrm{min}\{\alpha_{\mathfrak{s}}^{-\frac{1}{2}}, \frac{h_{\mathfrak{s}}}{\sqrt{\epsilon G_c}}\}^{-\frac{1}{2}}\|G_c\epsilon[\nabla \hat{\varphi}_{\mathfrak{m}}]\|_{\mathfrak{s}}$, we get
\begin{align*}
(G_c\epsilon)^{-\frac{1}{4}}\mathrm{min}\{\alpha_{\mathfrak{s}}^{-\frac{1}{2}}, \frac{h_{\mathfrak{s}}}{\sqrt{\epsilon G_c}}\}^{\frac{1}{2}}\|G_c\epsilon[\nabla \hat{\varphi}_{\mathfrak{m}}]\|_{\mathfrak{s}}
\lesssim \|\mathcal{G}_{\mathfrak{m}}\|_{\ast,\epsilon,\omega_p}  + \mathrm{min}\{\frac{h_{\mathfrak{s}}}{\sqrt{G_c\epsilon}}, \alpha^{-\frac{1}{2}}_{\mathfrak{s}}\} \|r(\hat{\varphi}_{\mathfrak{m}})\|_{\omega_{\mathfrak{s}}}.
\end{align*} 
Similar to the proof of the lower bound in terms of $\eta^{\varphi}_{1,p}$, we exploit that $\frac{\alpha_{\mathfrak{s}}}{\alpha_p} = C_{\mathfrak{s},p}\neq 0$ is a computable constant. Further, we make use of~\eqref{eq:LowerBound1} to get the desired lower bound
\begin{align}\label{eq:LowerBoundEta2}
\eta^{\varphi}_{2,p}\lesssim  \|\mathcal{G}_{\mathfrak{m}}\|_{\ast,\epsilon,\omega_p} +\mathrm{osc}_p(r)
\lesssim \|\hat{\varphi}-\hat{\varphi}_{\mathfrak{m}}\|_{\epsilon,\omega_p} + \|\hat{\Lambda}-\widetilde{\hat{\Lambda}}_{\mathfrak{m}}\|_{\ast,\epsilon,\omega_p} +\mathrm{osc}_p(r). 
\end{align}
To derive a local lower bound in terms of $\eta^{\varphi}_{3,p}$, we can proceed in the same way to get
\begin{align}\label{eq:LowerBoundEta3}
\eta^{\varphi}_{3,p} \lesssim \|\hat{\varphi}-\hat{\varphi}_{\mathfrak{m}}\|_{\epsilon,\omega_p} + \|\hat{\Lambda}-\widetilde{\hat{\Lambda}}_{\mathfrak{m}}\|_{\ast,\epsilon,\omega_p}+\mathrm{osc}_p(r).
\end{align}

Theorem~\ref{Theorem:LowerBound} follows from~\eqref{eq:LowerBoundEta1},~\ref{eq:LowerBoundEta2},~\eqref{eq:LowerBoundEta3}.

\subsection{Local error bound in terms of $\eta^{\varphi}_{4,p}$}

In this subsection, we show that also $\eta^{\varphi}_{4,p}$ constitutes a local lower bound.  We proceed almost as in~\cite{Walloth_2018b}.
As the case $\eta^{\varphi}_{4,p}= 0$ is irrelevant, we can assume $s_p>0$ which implies that $p$ is a contact node, i.e., $(I_{\mathfrak{m}}^n(\varphi^{n-1}_{\mathfrak{m}})-\hat{\varphi}_{\mathfrak{m}})(p)=0$.
Choose a node $\hat{q}$ in $\omega_p$ such that $(I_{\mathfrak{m}}^n(\varphi^{n-1}_{\mathfrak{m}})-\hat{\varphi}_{\mathfrak{m}})(\hat{q})\ge (I_{\mathfrak{m}}^n(\varphi^{n-1}_{\mathfrak{m}})-\hat{\varphi}_{\mathfrak{m}})(q)$ for all $q\in\omega_p$. We denote the unit vector pointing from $p$ to $\hat{q}$ by $\boldsymbol{\tau}$.
We denote the element to which $p$ and $\hat{q}$ belong  by $\mathfrak{e}_1$ and the element in $\omega_p$ which is intersected by $-\boldsymbol{\tau}$, starting in $p$, is denoted by $\mathfrak{e}_N$. The elements between $\mathfrak{e}_1$ and $\mathfrak{e}_N$ are denoted in order by $\mathfrak{e}_i$, $i=2,\ldots, N-1$.
For the ease of presentation, we set $v_{\mathfrak{m}}:= (I_{\mathfrak{m}}^n(\varphi^{n-1}_{\mathfrak{m}})-\hat{\varphi}_{\mathfrak{m}})$ in the following.
We use Taylor expansion around $v_{\mathfrak{m}}(p)=0$ and the mean value form of the remainder, i.e., there exists a $\zeta_{\mathfrak{e}_1}$ such that 
\begin{equation}\label{eq:TaylorExpansion}
\begin{split}
v_{\mathfrak{m}}(\hat{q}) &= \underbrace{v_{\mathfrak{m}}(p)}_{=0} + \nabla|_{\mathfrak{e}_1}(v_{\mathfrak{m}}(\zeta_{\mathfrak{e}_1}))\cdot(\hat{q}-p). 
\end{split}
\end{equation}
As by definition $v_{\mathfrak{m}}(\hat{q})\ge 0$, it follows that  $ \nabla|_{\mathfrak{e}_1}(v_{\mathfrak{m}}(\zeta_{\mathfrak{e}_1}))\cdot\boldsymbol{\tau}\ge 0 $.

Let $\tilde{q}\in\mathfrak{e}_N$ be the point of intersection of $-\tensor{\tau}$ and $\partial \omega_p$. As $v_{\mathfrak{m}}(q)\ge 0$ for all $q\in\omega_p$, we can conclude, as in~\eqref{eq:TaylorExpansion}, that there exists a $\zeta_{\mathfrak{e}_N}$ such that $\nabla|_{\mathfrak{e}_N}(v_{\mathfrak{m}}(\zeta_{\mathfrak{e}_N}))\cdot(-\boldsymbol{\tau})\ge 0$.
Thus, we can add $\nabla|_{\mathfrak{e}_N}(v_{\mathfrak{m}}(\zeta_{\mathfrak{e}_N}))\cdot(-\boldsymbol{\tau})$ to~\eqref{eq:TaylorExpansion}
\begin{equation*}
v_{\mathfrak{m}}(\hat{q})  \lesssim h_p(\nabla|_{\mathfrak{e}_1}(v_{\mathfrak{m}}(\zeta_{\mathfrak{e}_1})) -  \nabla|_{\mathfrak{e}_N}(v_{\mathfrak{m}}(\zeta_{\mathfrak{e}_N})))\cdot \boldsymbol{\tau}
 \lesssim h_p |\nabla|_{\mathfrak{e}_1} v_{\mathfrak{m}}(\zeta_{\mathfrak{e}_1}) -\nabla|_{\mathfrak{e}_{N}}v_{\mathfrak{m}}(\zeta_{\mathfrak{e}_{N}})|
\end{equation*}
Next, we add and subtract $\nabla|_{\mathfrak{e}_i} v_{\mathfrak{m}}(\zeta_{\mathfrak{e}_i}) $ for $i=2,\ldots, N-1$ where the choice of $\zeta_{\mathfrak{e}_i}\in\mathfrak{e}_i$ for $i\neq \{1,N\}$ is arbitrary and can be set to the midpoint of the elements. We define $\overline{\nabla|_{\mathfrak{e}_i} v_{\mathfrak{m}}}:= \nabla|_{\mathfrak{e}_i} v_{\mathfrak{m}}(\zeta_{\mathfrak{e}_i}) $ as a piecewise constant approximation of  $\nabla|_{\mathfrak{e}_i} v_{\mathfrak{m}}$.
Thus, we get from the previous inequality
\begin{align*}
v_{\mathfrak{m}}(\hat{q}) 
& \lesssim h_p h_p^{-\frac{1}{2}}\|[\overline{\nabla (v_{\mathfrak{m}})}]\|_{\gamma^I_p}.
\end{align*}
Further, we exploit  
\begin{align*}
\left<\hat{\Lambda}_{\mathfrak{m}} ,\phi_p \right>:=\int_{\gamma_{p}^I}G_c\epsilon[\nabla (\hat{\varphi}_{\mathfrak{m}})]\phi_p -\int_{\gamma_{p}^{\Gamma}}G_c\epsilon\nabla (\hat{\varphi}_{\mathfrak{m}})\phi_p   + \int_{\omega_p}r(\hat{\varphi}_{\mathfrak{m}})\phi_p. 
\end{align*}
Putting together and assuming $\frac{h_p}{\sqrt{G_c\epsilon}} \leq\alpha_p^{-\frac{1}{2}}$
\begin{align*}
&(\eta^{\varphi}_{4,p})^2= \left<\hat{\Lambda}_{\mathfrak{m}} ,\phi_p \right>c_p(I_{\mathfrak{m}}^n(\varphi^{n-1}_{\mathfrak{m}})-\hat{\varphi}_{\mathfrak{m}})\\
 &\leq  h_p^{2}(h_p h_p^{-\frac{1}{2}}\|[\overline{\nabla v_{\mathfrak{m}}}]\|_{\gamma^I_p})h_p^{-2}\left(\|G_c\epsilon[\nabla \hat{\varphi}_{\mathfrak{m}}]\|_{\gamma^I_p} \|\phi_p\|_{\gamma^I_p}+ \|G_c\epsilon\nabla \hat{\varphi}_{\mathfrak{m}}\|_{\gamma^{\Gamma}_p} \|\phi_p\|_{\gamma^{\Gamma}_p}\right.\\
  &\left.\qquad+\|r(\hat{\varphi}_{\mathfrak{m}})\|_{\omega_p}\|\phi_p\|_{\omega_p} \right)\\
&\leq (h_p^{\frac{1}{2}}G_c^{-1}\epsilon^{-1}\|G_c\epsilon[\overline{\nabla v_{\mathfrak{m}}}]\|_{\gamma^I_p})\left(\|G_c\epsilon[\nabla \hat{\varphi}_{\mathfrak{m}}]\|_{\gamma^I_p} h_p^{\frac{1}{2}}+\|G_c\epsilon\nabla \hat{\varphi}_{\mathfrak{m}}\|_{\gamma^\Gamma_p} h_p^{\frac{1}{2}}+  \|r(\hat{\varphi}_{\mathfrak{m}})\|_{\omega_p}h_p^{1} \right)\\
&\leq \biggl(\frac{h_p^{\frac{1}{2}}}{\sqrt{G_c\epsilon}}\|G_c\epsilon[\overline{\nabla v_{\mathfrak{m}}}]\|_{\gamma^I_p}\biggr)\\
&\qquad \biggl(\|G_c\epsilon[\nabla \hat{\varphi}_{\mathfrak{m}}]\|_{\gamma^I_p} \frac{h_p^{\frac{1}{2}}}{\sqrt{G_c\epsilon}}+\|G_c\epsilon\nabla \hat{\varphi}_{\mathfrak{m}}\|_{\gamma^\Gamma_p} \frac{h_p^{\frac{1}{2}}}{\sqrt{G_c\epsilon}}+  \|r(\hat{\varphi}_{\mathfrak{m}})\|_{\omega_p}\frac{h_p}{\sqrt{G_c\epsilon}} \biggr)\\
&\lesssim\frac{h_p}{G_c\epsilon}\|G_c\epsilon[\nabla \hat{\varphi}_{\mathfrak{m}}]\|^2_{\gamma^I_p}  +\frac{h_p}{G_c\epsilon}\|G_c\epsilon\nabla \hat{\varphi}_{\mathfrak{m}}\|^2_{\gamma^\Gamma_p}  + \frac{h_p^2}{G_c\epsilon}\|r(\hat{\varphi}_{\mathfrak{m}})\|^2_{\omega_p}  + \frac{h_p}{G_c\epsilon}\|G_c\epsilon[\overline{\nabla v_{\mathfrak{m}}}]\|^2_{\gamma^I_p} \\
&\lesssim(\eta^{\varphi}_{1,p})^2 + (\eta^{\varphi}_{2,p})^2 + (\eta^{\varphi}_{3,p})^2 + \frac{h_p}{\sqrt{G_c\epsilon}}\frac{1}{\sqrt{G_c\epsilon}}\|G_c\epsilon[\overline{\nabla v_{\mathfrak{m}}}]\|^2_{\gamma^I_p}.
\end{align*}
Thus, together with~\eqref{eq:LowerBoundEta1},~\eqref{eq:LowerBoundEta2},~\eqref{eq:LowerBoundEta3}
\begin{equation}\label{eq:ProofEta4Robust}
(\eta^{\varphi}_{4,p})^2 \lesssim  \|\hat{\varphi}-\hat{\varphi}_{\mathfrak{m}}\|_{\epsilon,\omega_p} + \|\hat{\Lambda}-\widetilde{\hat{\Lambda}}_{\mathfrak{m}}\|_{\ast,\epsilon,\omega_p} + \mathrm{osc}_p(r)+ \left(\frac{h_p}{\sqrt{G_c\epsilon}}\frac{1}{\sqrt{G_c\epsilon}}\|G_c\epsilon[\overline{\nabla v_{\mathfrak{m}}}]\|^2_{\gamma^I_p}\right)^{\frac{1}{2}}.
\end{equation}
In the remaining case $
\alpha_p^{-\frac{1}{2}}<\frac{h_p}{\sqrt{G_c\epsilon}}$, i.e, in
$\eta^{\varphi}_{2,p},\eta^{\varphi}_{3,p}$ it is $ \mathrm{min}\{\frac{h_p}{\sqrt{G_c\epsilon}},\alpha_p^{-\frac{1}{2}}\}^{\frac{1}{2}}(G_c\epsilon)^{-\frac{1}{4}} = \alpha_p^{-\frac{1}{4}}(G_c\epsilon)^{-\frac{1}{4}} $ and  in $\eta^{\varphi}_{1,p}$ it is $\mathrm{min}\{\frac{h_p}{\sqrt{G_c\epsilon}},\alpha_p^{-\frac{1}{2}}\} =\alpha_p^{-\frac{1}{2}} $. We exploit $h_p\lesssim 1$ and proceed as before 
\begin{align*}
(\eta^{\varphi}_{4,p})^2 &= \left<\hat{\Lambda}_{\mathfrak{m}} ,\phi_p \right>c_p(I_{\mathfrak{m}}^n(\varphi^{n-1}_{\mathfrak{m}})-\hat{\varphi}_{\mathfrak{m}})\\
&\leq \left(\frac{h_p^{\frac{1}{2}}}{\sqrt{G_c\epsilon}}\|G_c\epsilon[\overline{\nabla v_{\mathfrak{m}}}]\|_{\gamma^I_p} \right)\\
&\qquad \left(\|G_c\epsilon[\nabla \hat{\varphi}_{\mathfrak{m}}]\|_{\gamma^I_p} \frac{h_p^{\frac{1}{2}}}{\sqrt{G_c\epsilon}}+\|G_c\epsilon\nabla \hat{\varphi}_{\mathfrak{m}}\|_{\gamma^\Gamma_p} \frac{h_p^{\frac{1}{2}}}{\sqrt{G_c\epsilon}}+  \|r(\hat{\varphi}_{\mathfrak{m}})\|_{\omega_p}\frac{h_p}{\sqrt{G_c\epsilon}} \right)\\
&\leq \alpha_p^{-\frac{1}{4}}(G_c\epsilon)^{-\frac{1}{4}}\|G_c\epsilon[\nabla \hat{\varphi}_{\mathfrak{m}}]\|_{\gamma^I_p}(\alpha_p^{\frac{1}{4}}G_c^{-\frac{3}{4}}\epsilon^{-\frac{3}{4}}\|G_c\epsilon[\overline{\nabla v_{\mathfrak{m}}}]\|_{\gamma^I_p} ) \\
&\qquad + \alpha_p^{-\frac{1}{4}}(G_c\epsilon)^{-\frac{1}{4}}\|G_c\epsilon\nabla \hat{\varphi}_{\mathfrak{m}}\|_{\gamma^\Gamma_p}(\alpha_p^{\frac{1}{4}}G_c^{-\frac{3}{4}}\epsilon^{-\frac{3}{4}}\|G_c\epsilon[\overline{\nabla v_{\mathfrak{m}}}]\|_{\gamma^I_p})\\
&\qquad + \alpha_p^{-\frac{1}{2}}\|r(\hat{\varphi}_{\mathfrak{m}})\|_{\omega_p}(\alpha_p^{\frac{1}{2}}G_c^{-1}\epsilon^{-1}\|G_c\epsilon[\overline{\nabla v_{\mathfrak{m}}}]\|_{\gamma^I_p})\\
&\lesssim (\eta^{\varphi}_{1,p})^2 + (\eta^{\varphi}_{2,p})^2 + (\eta^{\varphi}_{3,p})^2 +\mathrm{max} \{\alpha_p(G_c\epsilon)^{-2}, \alpha_p^{\frac{1}{2}} (G_c\epsilon)^{-\frac{3}{2}}\}\|G_c\epsilon[\overline{\nabla v_{\mathfrak{m}}}]\|^2_{\gamma^I_p}.
\end{align*}
Thus, together
with~\eqref{eq:LowerBoundEta1},~\eqref{eq:LowerBoundEta2},~\eqref{eq:LowerBoundEta3},
we get
\begin{equation*}
\eta^{\varphi}_{4,p} \lesssim  \|\hat{\varphi}-\hat{\varphi}_{\mathfrak{m}}\|_{\epsilon,\omega_p} + \|\hat{\Lambda}-\widetilde{\hat{\Lambda}}_{\mathfrak{m}}\|_{\ast,\epsilon,\omega_p} +\mathrm{osc}_p(r)+ \mathrm{max} \{\alpha_p(G_c\epsilon)^{-2}, \alpha_p^{\frac{1}{2}} (G_c\epsilon)^{-\frac{3}{2}}\}\|G_c\epsilon[\overline{\nabla v_{\mathfrak{m}}}]\|^2_{\gamma^I_p}. 
\end{equation*}
This together with~\eqref{eq:LowerBound0}
and~\eqref{eq:ProofEta4Robust} yields Theorem~\ref{Theorem:LowerBound2}.

%% file: EstimatorEqu.tex
The residual a posteriori estimator, we give in this section is derived for the solution of the following equation
\begin{problem}\label{AuxProbEquCont}
Let $\varphi_{m}^{n-1}$ be given, then find $\hat{\boldsymbol{u}}\in \boldsymbol{\mathcal{H}}^n_D$ such that 
\begin{equation}
\left<g(\varphi_{\mathfrak{m}}^{n-1})\boldsymbol{\sigma}(\hat{\boldsymbol{u}}),\boldsymbol{E}_{\mathrm{lin}}(\boldsymbol{w}) \right> = 0 \qquad\forall \boldsymbol{w}\in\boldsymbol{\mathcal{H}}_0.
\end{equation}
\end{problem}
As discrete approximation of Problem~\ref{AuxProbEquCont}, we consider
\begin{problem}\label{AuxProbEquDisc}
Let $\varphi_{m}^{n-1}$ be given, then find $\hat{\boldsymbol{u}}_{\mathfrak{m}}\in \boldsymbol{\mathcal{H}}^n_{\mathfrak{m},D}$ such that 
\begin{equation}
\left<g(\varphi_{\mathfrak{m}}^{n-1})\boldsymbol{\sigma}(\hat{\boldsymbol{u}}_{\mathfrak{m}}),\boldsymbol{E}_{\mathrm{lin}}(\boldsymbol{w}_{\mathfrak{m}}) \right> = 0 \qquad\forall \boldsymbol{w}_{\mathfrak{m}}\in\boldsymbol{\mathcal{H}}_{\mathfrak{m},0}.
\end{equation}
\end{problem}
We note that the discrete solution $\hat{\boldsymbol{u}}_{\mathfrak{m}}$ of Problem~\ref{AuxProbEquDisc} equals the discrete solution $\boldsymbol{u}_{\mathfrak{m}}^n$ of Problem~\ref{DiscreteFormulation} in time step $n$. Further as $\varphi_{m}^{n-1}$ is an approximation of $\varphi^{n-1}$, the solution $\hat{\boldsymbol{u}}$ of Problem~\ref{AuxProbEquCont} is an approximation of $\boldsymbol{u}^n$ in Problem~\ref{WeakFormulation}.

Following~\cite{Verfuerth_1998a}, we derive the residual a posteriori estimator.
We define the residual
\begin{align}\nonumber
\left<\boldsymbol{\mathcal{R}}_{\mathfrak{m}}(\hat{\boldsymbol{u}}_{\mathfrak{m}}), \boldsymbol{w}\right>_{-1,1}:= &\;0-\left<g(\varphi_{\mathfrak{m}}^{n-1})\boldsymbol{\sigma}(\hat{\boldsymbol{u}}_{\mathfrak{m}}),\boldsymbol{E}_{\mathrm{lin}}(\boldsymbol{w}) \right>\\ \label{eq:ResLinElast}
=&\left<g(\varphi_{\mathfrak{m}}^{n-1})\boldsymbol{\sigma}(\hat{\boldsymbol{u}}_{\mathfrak{m}}),\boldsymbol{E}_{\mathrm{lin}}(\boldsymbol{w}) \right>-\left<g(\varphi_{\mathfrak{m}}^{n-1})\boldsymbol{\sigma}(\hat{\boldsymbol{u}}_{\mathfrak{m}}),\boldsymbol{E}_{\mathrm{lin}}(\boldsymbol{w}) \right>.
\end{align}
Let $c^{\ast}$ be a constant depending on the largest eigenvalue of Hooke's tensor $C$ and $c_{\ast}$ be a constant depending on the smallest eigenvalues of Hooke's tensor. Further, we note that $\kappa\leq g(\varphi^{n-1}_{\mathfrak{m}})\leq 1$. 
We conclude from~\eqref{eq:ResLinElast}
\begin{align}\label{AbstractLowerBoundU}
\|\boldsymbol{\mathcal{R}}_{\mathfrak{m}}(\hat{\boldsymbol{u}}_{\mathfrak{m}})\|_{-1}= 
\mathrm{sup}_{w\in H^1}\frac{\left<g(\varphi_{\mathfrak{m}}^{n-1})\boldsymbol{\sigma}(\hat{\boldsymbol{u}}- \hat{\boldsymbol{u}}_{\mathfrak{m}}),\boldsymbol{E}_{\mathrm{lin}}(\boldsymbol{w}) \right>}{\|\boldsymbol{w}\|_1}\leq c^{\ast}\|\hat{\boldsymbol{u}}-\hat{\boldsymbol{u}}_{\mathfrak{m}}\|_{1}
\end{align}
and
\begin{align*}%\label{AbstractUpperBoundU}
\|\boldsymbol{\mathcal{R}}_{\mathfrak{m}}(\hat{\boldsymbol{u}}_{\mathfrak{m}})\|_{-1}\ge \frac{\left<g(\varphi_{\mathfrak{m}}^{n-1})\boldsymbol{\sigma}(\hat{\boldsymbol{u}}- \hat{\boldsymbol{u}}_{\mathfrak{m}}),\boldsymbol{E}_{\mathrm{lin}}(\hat{\boldsymbol{u}} -\hat{\boldsymbol{u}}_{\mathfrak{m}}) \right>}{\|\hat{\boldsymbol{u}} -\hat{\boldsymbol{u}}_{\mathfrak{m}}\|_1}\ge \kappa c_{\ast}\|\hat{\boldsymbol{u}} -\hat{\boldsymbol{u}}_{\mathfrak{m}}\|_1
\end{align*}

In order to derive the upper bound, we reformulate the residual by means of piecewise integration by parts
\begin{equation}\label{eq:LinResU_part}
\begin{aligned}
\bigl<&\boldsymbol{\mathcal{R}}_{\mathfrak{m}}(\hat{\boldsymbol{u}}_{\mathfrak{m}}), \boldsymbol{w}\bigr>_{-1,1} = -\int_{\Omega} g(\varphi_{\mathfrak{m}}^{n-1})\boldsymbol{\sigma}(\hat{\boldsymbol{u}}_{\mathfrak{m}}):\boldsymbol{E}_{\mathrm{lin}}(\boldsymbol{w})\\
&= \sum_{\mathfrak{e}\in\mathfrak{M}} \int_{\mathfrak{e}}\mathrm{div}(g(\varphi^{n-1}_{\mathfrak{m}})\boldsymbol{\sigma}(\hat{\boldsymbol{u}}_{\mathfrak{m}}))\cdot\boldsymbol{w} - \int_{\partial\mathfrak{e}}\boldsymbol{n}_{\mathfrak{e}} g(\varphi_{\mathfrak{m}}^{n-1})\boldsymbol{\sigma}(\hat{\boldsymbol{u}}_{\mathfrak{m}})\cdot\boldsymbol{w}\\
&= \sum_{\mathfrak{e}\in\mathfrak{M}} \int_{\mathfrak{e}}\left(\nabla g(\varphi^{n-1}_{\mathfrak{m}})\cdot\boldsymbol{\sigma}(\hat{\boldsymbol{u}}_{\mathfrak{m}}) + \mathrm{div}\boldsymbol{\sigma}(\hat{\boldsymbol{u}}_{\mathfrak{m}}) g(\varphi^{n-1}_{\mathfrak{m}})  \right)\cdot\boldsymbol{w} - \int_{\partial\mathfrak{e}}\boldsymbol{n}_{\mathfrak{e}} g(\varphi_{\mathfrak{m}}^{n-1})\boldsymbol{\sigma}(\hat{\boldsymbol{u}}_{\mathfrak{m}})\cdot\boldsymbol{w}.
\end{aligned}
\end{equation}
We abbreviate the interior residual by $\boldsymbol{r}(\hat{\boldsymbol{u}}_{\mathfrak{m}}):=\nabla g(\varphi^{n-1}_{\mathfrak{m}})\cdot\boldsymbol{\sigma}(\hat{\boldsymbol{u}}_{\mathfrak{m}}) + \mathrm{div}\boldsymbol{\sigma}(\hat{\boldsymbol{u}}_{\mathfrak{m}}) g(\varphi^{n-1}_{\mathfrak{m}})$, the jump terms between two neighboring elements $\mathfrak{e},\tilde{\mathfrak{e}}$ 
\begin{align*} 
\boldsymbol{J}(\hat{\boldsymbol{u}}_{\mathfrak{m}}):= \left((g(\varphi^{n-1}_{\mathfrak{m}})\boldsymbol{\sigma}(\hat{\boldsymbol{u}}_{\mathfrak{m}}))|_{\mathfrak{e}}\boldsymbol{n}_{\mathfrak{e}}-(g(\varphi^{n-1}_{\mathfrak{m}})\boldsymbol{\sigma}(\hat{\boldsymbol{u}}_{\mathfrak{m}}))|_{\tilde{\mathfrak{e}}}\boldsymbol{n}_{\tilde{\mathfrak{e}}}\right) 
\end{align*} 
and the jump terms at the Neumann boundary by  $\boldsymbol{J}^N(\hat{\boldsymbol{u}}_{\mathfrak{m}}):= \left((g(\varphi^{n-1}_{\mathfrak{m}})\boldsymbol{\sigma}(\hat{\boldsymbol{u}}_{\mathfrak{m}}))|_{\mathfrak{e}}\boldsymbol{n}_{\mathfrak{e}}\right) $.
In order to define the quasi-interpolation operator, we need to subclassify the boundary nodes in Dirichlet boundary nodes $\mathfrak{N}_{\mathfrak{m}}^D$ and Neumann boundary nodes $\mathfrak{N}_{\mathfrak{m}}^N$ with respect to the displacements. Thereby the quasi-interpolation operator is $\mathcal{I}_{\mathfrak{m}}(\boldsymbol{v}):= \sum_{p\in\mathfrak{N}_{\mathfrak{m}}} c_p(\boldsymbol{v})\phi_p\in\boldsymbol{\mathcal{H}}_{\mathfrak{m},0}$ for all $\boldsymbol{v}\in\boldsymbol{\mathcal{H}}_0$ with $c_p(v_i):=\frac{\int_{\omega_p}v_i\phi_p}{\int_{\omega_p}\phi_p}$  for $p\in \mathfrak{N}_{\mathfrak{m}}^I$ and $c_p(v_i):=0$ for $p\in\mathfrak{N}_{\mathfrak{m}}^D$, see, e.g.,~\cite[Section 3.5.]{Verfuerth_2013}.  Further, we denote $\gamma_p^N:= \Gamma^N\cap \omega_p$. We add $\left<\boldsymbol{\mathcal{R}}_{\mathfrak{m}}(\hat{\boldsymbol{u}}_{\mathfrak{m}}), \mathcal{I}_{\mathfrak{m}}(\boldsymbol{w})\right>_{-1,1}=0$ to~\eqref{eq:LinResU_part}, apply Cauchy-Schwarz and the $L^2$-approximation property of the quasi-interpolation operator 
\begin{align*}
\Bigl<&\boldsymbol{\mathcal{R}}_{\mathfrak{m}}(\hat{\boldsymbol{u}}_{\mathfrak{m}}), \boldsymbol{w}\Bigr>_{-1,1} =\Bigl< \boldsymbol{\mathcal{R}}_{\mathfrak{m}}(\hat{\boldsymbol{u}}_{\mathfrak{m}}),\sum_{p\in\mathfrak{N}_{\mathfrak{m}}} (\boldsymbol{w}- c_p(\boldsymbol{w}))\phi_p \Bigr>_{-1,1}\\
&= \sum_{p\in\mathfrak{N}_{\mathfrak{m}}} \int_{\omega_p}\boldsymbol{r}(\hat{\boldsymbol{u}}_{\mathfrak{m}})\cdot\left(\boldsymbol{w}- c_p(\boldsymbol{w})\right)\phi_p\\
&\quad- \int_{\gamma^N_p} \boldsymbol{J}^I(\hat{\boldsymbol{u}}_{\mathfrak{m}})\cdot(\boldsymbol{w} - c_p(\boldsymbol{w}))\phi_p
- \int_{\gamma^I_p} \boldsymbol{J}^N(\hat{\boldsymbol{u}}_{\mathfrak{m}})\cdot(\boldsymbol{w} - c_p(\boldsymbol{w}))\phi_p\\
&\leq \sum_{p\in\mathfrak{N}_{\mathfrak{m}}}\|\boldsymbol{r}(\hat{\boldsymbol{u}}_{\mathfrak{m}}) \|_{\omega_p}h_p\|\boldsymbol{w}\|_{1,\omega_p} + \|\boldsymbol{J}^N(\hat{\boldsymbol{u}}_{\mathfrak{m}})\|_{\gamma^N_p} h_p^{\frac{1}{2}}\|\boldsymbol{w}\|_{1,\omega_p} + \|\boldsymbol{J}^I(\hat{\boldsymbol{u}}_{\mathfrak{m}})\|_{\gamma^I_p} h_p^{\frac{1}{2}}\|\boldsymbol{w}\|_{1,\omega_p}\\
&\lesssim \left(\sum_{p\in\mathfrak{N}_{\mathfrak{m}}}\|\boldsymbol{r}(\hat{\boldsymbol{u}}_{\mathfrak{m}}) \|^2_{\omega_p}h^2_p + \|\boldsymbol{J}^I(\hat{\boldsymbol{u}}_{\mathfrak{m}})\|_{\gamma^I_p} h_p+ \|\boldsymbol{J}^N(\hat{\boldsymbol{u}}_{\mathfrak{m}})\|_{\gamma^N_p} h_p\right)^{\frac{1}{2}}\left( \sum_{p\in\mathfrak{N}_{\mathfrak{m}}}\|\boldsymbol{w}\|^2_{1,\omega_p} \right)^{\frac{1}{2}}
\end{align*}

Thus, we get the upper bound 
\begin{equation*}
\|\hat{\boldsymbol{u}} - \hat{\boldsymbol{u}}_{\mathfrak{m}}\|_1\lesssim \|\boldsymbol{\mathcal{R}}_{\mathfrak{m}}(\hat{\boldsymbol{u}}_{\mathfrak{m}})\|_{-1}\lesssim \left(\sum_{p\in\mathfrak{N}_{\mathfrak{m}}} (\eta^u_{1,p})^2 + (\eta^u_{2,p})^2 + (\eta^u_{3,p})^2\right)^{\frac{1}{2}} 
\end{equation*}
for the error estimator contributions 
\begin{align*}
\eta^{u}_{1,p}:= &h_p\|\boldsymbol{r}(\hat{\boldsymbol{u}}_{\mathfrak{m}}) \|_{\omega_p}\qquad
\eta^u_{2,p}:= h_p^{\frac{1}{2}} \|\boldsymbol{J}^I(\hat{\boldsymbol{u}}_{\mathfrak{m}})\|_{\gamma^I_p}\qquad
\eta^u_{3,p}:= h_p^{\frac{1}{2}} \|\boldsymbol{J}^N(\hat{\boldsymbol{u}}_{\mathfrak{m}})\|_{\gamma^N_p}. 
\end{align*}

In order to prove the lower bound, we use the bubble functions on elements $\Psi_{\mathfrak{e}}:=\Pi_{p\in\mathfrak{e}}\phi_p$ and on sides $\Psi_{\mathfrak{s}}:=\Pi_{p\in\mathfrak{s}}\phi_p$ with the properties
\begin{equation*}
\begin{aligned}
\|\rho\|^2_{\mathfrak{e}}&\lesssim\int_{\mathfrak{e}}\Psi_{\mathfrak{e}}\rho^2\lesssim
\|\rho\|^2_{\mathfrak{e}}, &\qquad
\|\Psi_{\mathfrak{e}}\rho\|_{1,\mathfrak{e}}&\lesssim h_{\mathfrak{e}}^{-1}\|\rho\|_{\mathfrak{e}}\\
\|\rho\|_{\mathfrak{s}}^2&\lesssim\int_{\mathfrak{s}}\Psi_{\mathfrak{s}}\rho^2\lesssim\|\rho\|^2_{\mathfrak{s}},&
\|\Psi_{\mathfrak{s}}\rho\|_{1,\omega_{\mathfrak{s}}}&\lesssim h_{\mathfrak{s}}^{-\frac{1}{2}}\|\rho\|_{\mathfrak{s}}\\
\|\Psi_{\mathfrak{s}}\rho\|_{\omega_{\mathfrak{s}}}&\lesssim h_{\mathfrak{s}}^{\frac{1}{2}}\|\rho\|_{\mathfrak{s}}.
\end{aligned}
\end{equation*}
for all polynomials $\rho$ defined on $\mathfrak{e}$ and $\mathfrak{s}$.
Thus, we get
\begin{align*}
\|\boldsymbol{r}(\hat{\boldsymbol{u}}_{\mathfrak{m}})\|^2_{\mathfrak{e}}&\lesssim \int_{\mathfrak{e}} \boldsymbol{r}(\hat{\boldsymbol{u}}_{\mathfrak{m}})\boldsymbol{r}(\hat{\boldsymbol{u}}_{\mathfrak{m}})\boldsymbol{\Psi}_{\mathfrak{e}}\\
&=\left<\boldsymbol{\mathcal{R}}_{\mathfrak{m}}(\hat{\boldsymbol{u}}_{\mathfrak{m}}), \boldsymbol{r}(\hat{\boldsymbol{u}}_{\mathfrak{m}})\boldsymbol{\Psi}_{\mathfrak{e}}\right>_{-1,1}\\
&\leq \|\boldsymbol{\mathcal{R}}_{\mathfrak{m}}(\hat{\boldsymbol{u}}_{\mathfrak{m}})\|_{-1,\mathfrak{e}}\|\boldsymbol{r}(\hat{\boldsymbol{u}}_{\mathfrak{m}})\boldsymbol{\Psi}_{\mathfrak{e}}\|_{1}\\
&\lesssim \|\boldsymbol{\mathcal{R}}_{\mathfrak{m}}(\hat{\boldsymbol{u}}_{\mathfrak{m}})\|_{-1,\mathfrak{e}}h^{-1}_{\mathfrak{e}}\|\boldsymbol{r}(\hat{\boldsymbol{u}}_{\mathfrak{m}})\|_{\mathfrak{e}}.% \\
%&\leq c^{\ast}\|\hat{\boldsymbol{u}} - \hat{\boldsymbol{u}}_{\mathfrak{m}}\|_1
\end{align*}
Dividing by $h^{-1}_{\mathfrak{e}}\|\boldsymbol{r}(\hat{\boldsymbol{u}}_{\mathfrak{m}})\|_{\mathfrak{e}}$ and exploiting \eqref{AbstractLowerBoundU} we arrive at
\begin{equation*}
\|\boldsymbol{r}(\hat{\boldsymbol{u}}_{\mathfrak{m}})\|\lesssim c^{\ast}\|\hat{\boldsymbol{u}} - \hat{\boldsymbol{u}}_{\mathfrak{m}}\|_{1,\mathfrak{e}}.
\end{equation*}
Further, we have
\begin{align*}
\|\boldsymbol{J}^I(\hat{\boldsymbol{u}}_{\mathfrak{m}})\|^2_{\mathfrak{s}}&\lesssim \int_{\mathfrak{s}} \boldsymbol{J}^I(\hat{\boldsymbol{u}}_{\mathfrak{m}})\boldsymbol{J}^I(\hat{\boldsymbol{u}}_{\mathfrak{m}})\boldsymbol{\Psi}_{\mathfrak{s}}\\
&=\left<\boldsymbol{\mathcal{R}}_{\mathfrak{m}}(\hat{\boldsymbol{u}}_{\mathfrak{m}}), \boldsymbol{J}^I(\hat{\boldsymbol{u}}_{\mathfrak{m}})\boldsymbol{\Psi}_{\mathfrak{s}}\right>_{-1,1} - \int_{\omega_{\mathfrak{s}}}\boldsymbol{r}(\hat{\boldsymbol{u}}_{\mathfrak{m}})\boldsymbol{J}^I(\hat{\boldsymbol{u}}_{\mathfrak{m}})\boldsymbol{\Psi}_{\mathfrak{s}} \\
&\leq \|\boldsymbol{\mathcal{R}}_{\mathfrak{m}}(\hat{\boldsymbol{u}}_{\mathfrak{m}})\|_{-1,\omega_{\mathfrak{s}}}\|\boldsymbol{J}^I(\hat{\boldsymbol{u}}_{\mathfrak{m}})\boldsymbol{\Psi}_{\mathfrak{s}}\|_{1}
 + \|\boldsymbol{r}(\hat{\boldsymbol{u}}_{\mathfrak{m}})\|_{\omega_{\mathfrak{s}}}\|\boldsymbol{J}^I(\hat{\boldsymbol{u}}_{\mathfrak{m}})\boldsymbol{\Psi}_{\mathfrak{s}}\|_{\omega_{\mathfrak{s}}} \\
 &\lesssim \|\boldsymbol{\mathcal{R}}_{\mathfrak{m}}(\hat{\boldsymbol{u}}_{\mathfrak{m}})\|_{-1,\omega_{\mathfrak{s}}}h_{\mathfrak{s}}^{-\frac{1}{2}}\|\boldsymbol{J}^I(\hat{\boldsymbol{u}}_{\mathfrak{m}})\|_{\mathfrak{s}}
 + \|\boldsymbol{r}(\hat{\boldsymbol{u}}_{\mathfrak{m}})\|_{\omega_{\mathfrak{s}}}h_{\mathfrak{s}}^{\frac{1}{2}} \|\boldsymbol{J}^I(\hat{\boldsymbol{u}}_{\mathfrak{m}})\boldsymbol{\Psi}_{\mathfrak{s}}\|_{\mathfrak{s}}.% \\
%&\leq c^{\ast}\|\hat{\boldsymbol{u}} - \hat{\boldsymbol{u}}_{\mathfrak{m}}\|_1.
\end{align*}
Dividing by $h_{\mathfrak{s}}^{\frac{1}{2}} \|\boldsymbol{J}^I(\hat{\boldsymbol{u}}_{\mathfrak{m}})\boldsymbol{\Psi}_{\mathfrak{s}}\|_{\mathfrak{s}}$ and exploiting \eqref{AbstractLowerBoundU}  we arrive at 
\begin{equation*}
\|\boldsymbol{J}^I(\hat{\boldsymbol{u}}_{\mathfrak{m}})\|_{\mathfrak{s}}\lesssim c^{\ast}\|\hat{\boldsymbol{u}} - \hat{\boldsymbol{u}}_{\mathfrak{m}}\|_{1,\omega_{\mathfrak{s}}}.
\end{equation*}
The proof for the jump terms at the Neumann boundary follows in the same way. 
Thus, we get the local lower bounds
\begin{align*}
\eta^u_{1,p}&\lesssim \|\hat{\boldsymbol{u}} - \hat{\boldsymbol{u}}_{\mathfrak{m}}\|_{1,\omega_p},&
\eta^u_{2,p}&\lesssim \|\hat{\boldsymbol{u}} - \hat{\boldsymbol{u}}_{\mathfrak{m}}\|_{1,\omega_p},&
\eta^u_{3,p}&\lesssim \|\hat{\boldsymbol{u}} - \hat{\boldsymbol{u}}_{\mathfrak{m}}\|_{1,\omega_p}.
\end{align*}
We note that the constants in the relation of error and upper and lower bounds depend on Hooke's tensor as well as on $\kappa$.

%% file: Numerik.tex
In this section, we demonstrate the properties of the estimators. We show the adaptively refined grids as well as the convergence and the efficiency index. Therefor, we consider different examples for which we first describe the configurations of the tests. \\
{\bf A single edge notched tension test}\\
We adapt the data from~\cite{Miehe_Welschinger_Hofacker_2010a}.
The domain is a unit square of length $\SI{1}{\milli\meter}$ with a slit on the line $y=\SI{0.5}{\milli\meter}$ and $x=[\SI{0.25}{\milli\meter}; \SI{1.0}{\milli\meter}]$.  The uniform starting mesh consists of squares with a diameter of  $h\approx \SI{0.044}{\milli\meter}$. 
The time step size is $\tau=\SI{e-5}{s}$. The Lam{\'e} coefficients are  $\mu= \SI{80.77}{\kilo\newton\per\square\milli\meter}$ and $\lambda= \SI{121.15}{\kilo\newton\per\square\milli\meter}$. Further, we have the parameters $G_c=\SI{2.7d-3}{\kilo\newton\per\milli\meter}$ and  $\kappa=10^{-8}$.
We choose $\epsilon\in [0.022; 0.325]$.

At the boundary of the unit square we impose Dirichlet and Neumann boundary conditions. On the upper boundary we pull with Dirichlet values $(u_D)_y = 2\cdot\tau$ in $y$-direction, while it is fixed in $x$-direction, i.e., $(u_D)_x=0$. At the lower boundary we fix the body in $x$- and $y$-direction, i.e., $\boldsymbol{u}_D=\boldsymbol{0}$. On the remaining boundaries we have Neumann boundaries with zero values. 

{\bf A single edge notched shear test}\\
We adapt the data from~\cite{Miehe_Welschinger_Hofacker_2010a}.
The domain is a unit square of length $\SI{1}{\milli\meter}$  with a slit on the line $y=\SI{0.5}{\milli\meter}$ and $x=[\SI{0.5}{\milli\meter}; \SI{1.0}{\milli\meter}]$. The uniform starting mesh consists of squares with a diameter of  $h\approx \SI{0.044}{\milli\meter}$. 
The time step size is $\tau=\SI{e-4}{s}$. The Lam{\'e} coefficients are  $\mu= \SI{80.77}{\kilo\newton\per\square\milli\meter}$ and $\lambda= \SI{121.15}{\kilo\newton\per\square\milli\meter}$. Further, we have the parameters $G_c=\SI{2.7d-3}{\kilo\newton\per\milli\meter}$ and  $\kappa=10^{-8}$. We choose $\epsilon\in [0.022; 0.325]$.

At the boundary of the unit square we impose Dirichlet and Neumann boundary conditions. In $x$-direction we pull to the left on the upper boundary, i.e., $(u_D)_x= -\tau$ and on the lower boundary we fix the body with $(u_D)_x=0$. On the remaining boundaries we have zero Neumann values in $x$-direction. In $y$-direction we impose zero Dirichlet values everywhere. 

{\bf An L-shape panel test}\\
We adapt the data from~\cite{Winkler2001}.
The L-shaped domain is given by
$(\SI{0}{\milli\meter},\SI{250}{\milli\meter})\times(\SI{0}{\milli\meter},\SI{500}{\milli\meter})
\cup [\SI{250}{\milli\meter},\SI{500}{\milli\meter}) \times
(\SI{250}{\milli\meter},\SI{500}{\milli\meter})$.
The uniform starting mesh consists of squares with a diameter of $h\approx \SI{17.67}{\milli\meter}$. The time step size is $\tau=\SI{e-3}{s}$. The Lam{\'e} coefficients are  $\mu= \SI{10.95}{\kilo\newton\per\square\milli\meter}$ and $\lambda= \SI{6.16}{\kilo\newton\per\square\milli\meter}$. Further, we have the parameters $G_c=\SI{8.9d-5}{\kilo\newton\per\milli\meter}$ and  $\kappa=10^{-8}$.  

At the bottom boundary we fix the body with Dirichlet boundary conditions $\boldsymbol{u}_D=\boldsymbol{0}$. Further, at the small horizontal boundary line at the right where $y=\SI{250}{\milli\meter}$ and $x=[\SI{470}{\milli\meter}; \SI{500}{\milli\meter}]$ we push with Dirichlet boundary conditions in $y$-direction, i.e., $(u_D)_y=\tau$. At all other boundaries and directions we have zero Neumann values. 

\begin{remark}
The physics of the single edge notched shear test and the L-shape panel test  demand to use the stress splitting as in Problem~\ref{WeakFormulationWithSplit} and~\ref{DiscreteFormulationWithSplit}. The stress splitting enters in the derivation of the estimator $\eta^{\varphi}$ but not in the estimator $\eta^u$ which has been derived as standard residual estimator for an elliptic problem. 
\end{remark}

For the solution a complementarity formulation of
the variational inequality as described
in~\cite[Section~4]{Mang_Walloth_Wick_Wollner_2019} is used. The
calculations are performed using~\cite{Goll_Wick_Wollner_2017} based
on the finite element library~\cite{dealii2019design}.

\subsection{Adaptive refinement using the estimator $\eta^{\varphi}$}

\subsubsection{Adaptively refined grids}

We show the adaptively refined grids steered by the estimator
$\eta^{\varphi}$ and the corresponding phase field. For the tension
and the shear test, we show grids and the phase field at two different
time steps. One time step is chosen in the middle of the simulation
when the crack already started to grow and the second time step is the
moment when the crack reaches the boundary. We see for the tension
test in Figure~\ref{fig:tension_preRef4_Mesh6_Grid_etaPhi}, for the
shear test in Figure~\ref{fig:shear_preRef4_Mesh5_Grid_etaPhi}, and
for the L-shape panel test in
Figure~\ref{fig:Lshape_Mesh5_Grid_etaPhi} that the crack path is well
resolved.
\vspace*{-5mm}
\begin{figure}[H]
\centering
\subfloat[mesh at $n=310$]{\includegraphics[width=0.22\textwidth,clip,trim=20cm 4cm 20cm 4cm]{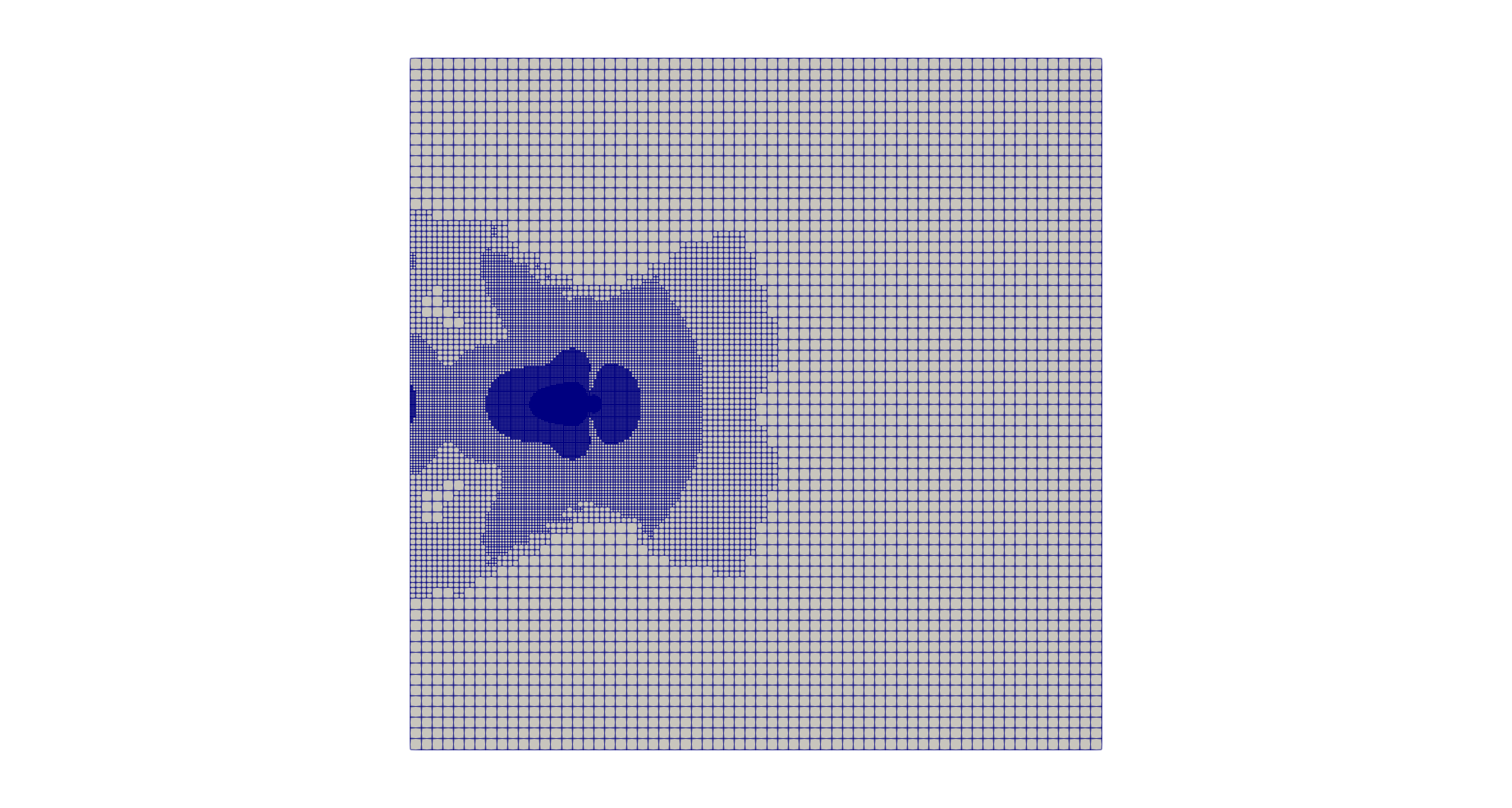}}\hfill
\subfloat[$\varphi$ at $n=310$]{\includegraphics[width=0.22\textwidth,clip,trim=20cm 4cm 20cm 4cm]{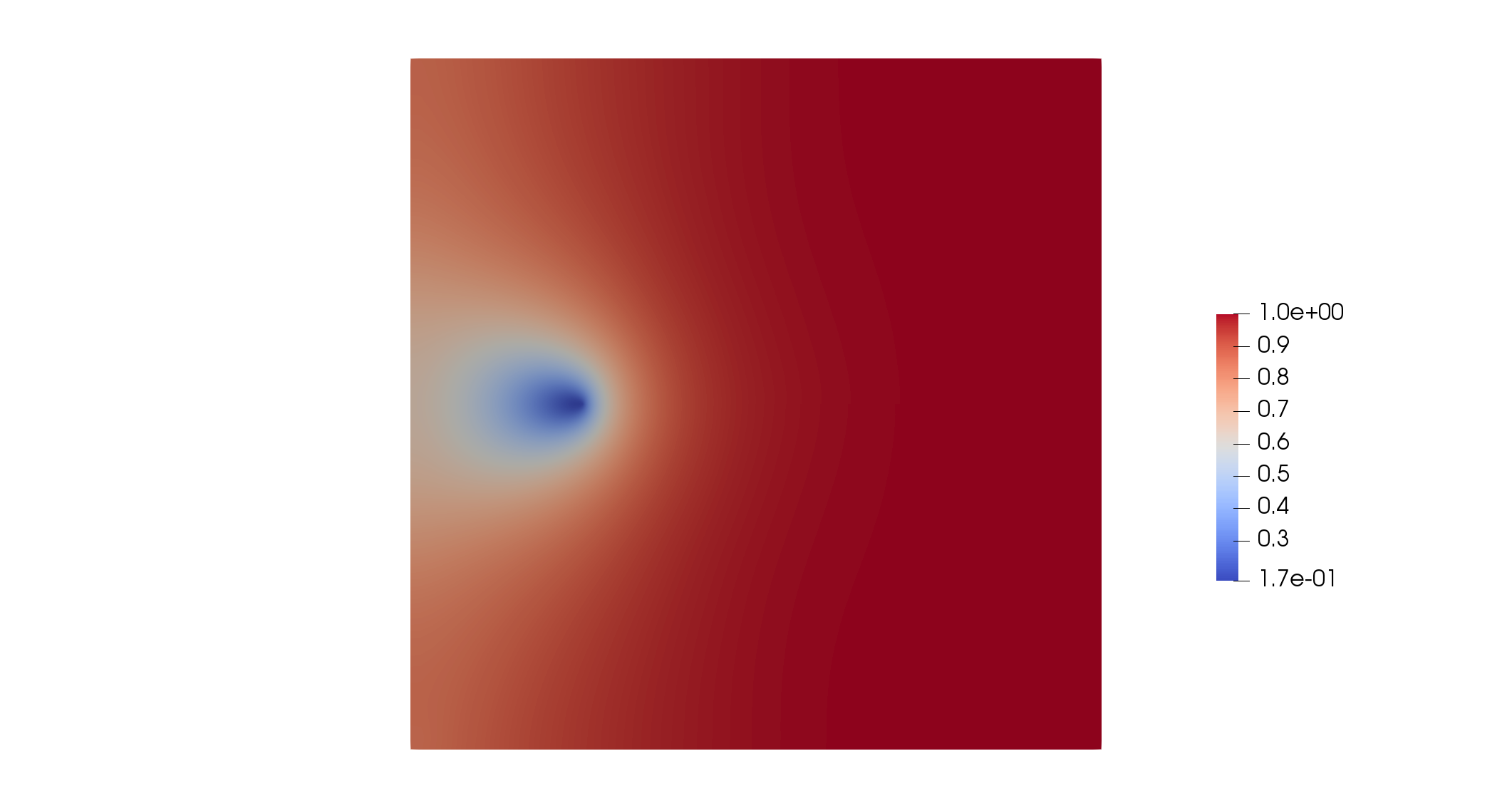}}\hfill
\subfloat[mesh at $n=324$]{\includegraphics[width=0.22\textwidth,clip,trim=20cm 4cm 20cm 4cm]{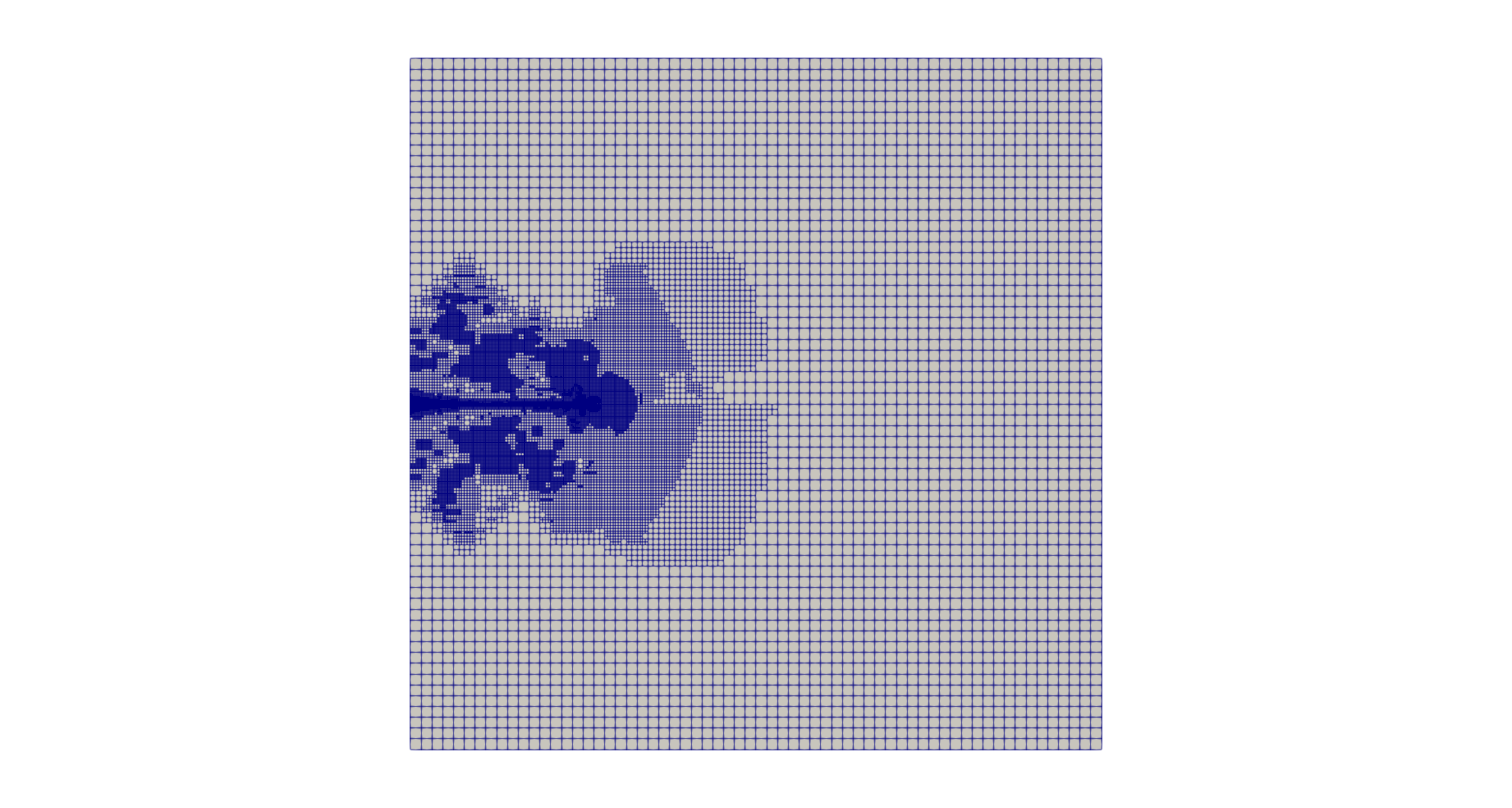}}\hfill
\subfloat[$\varphi$ at $n=324$]{\includegraphics[width=0.22\textwidth,clip,trim=20cm 4cm 20cm 4cm]{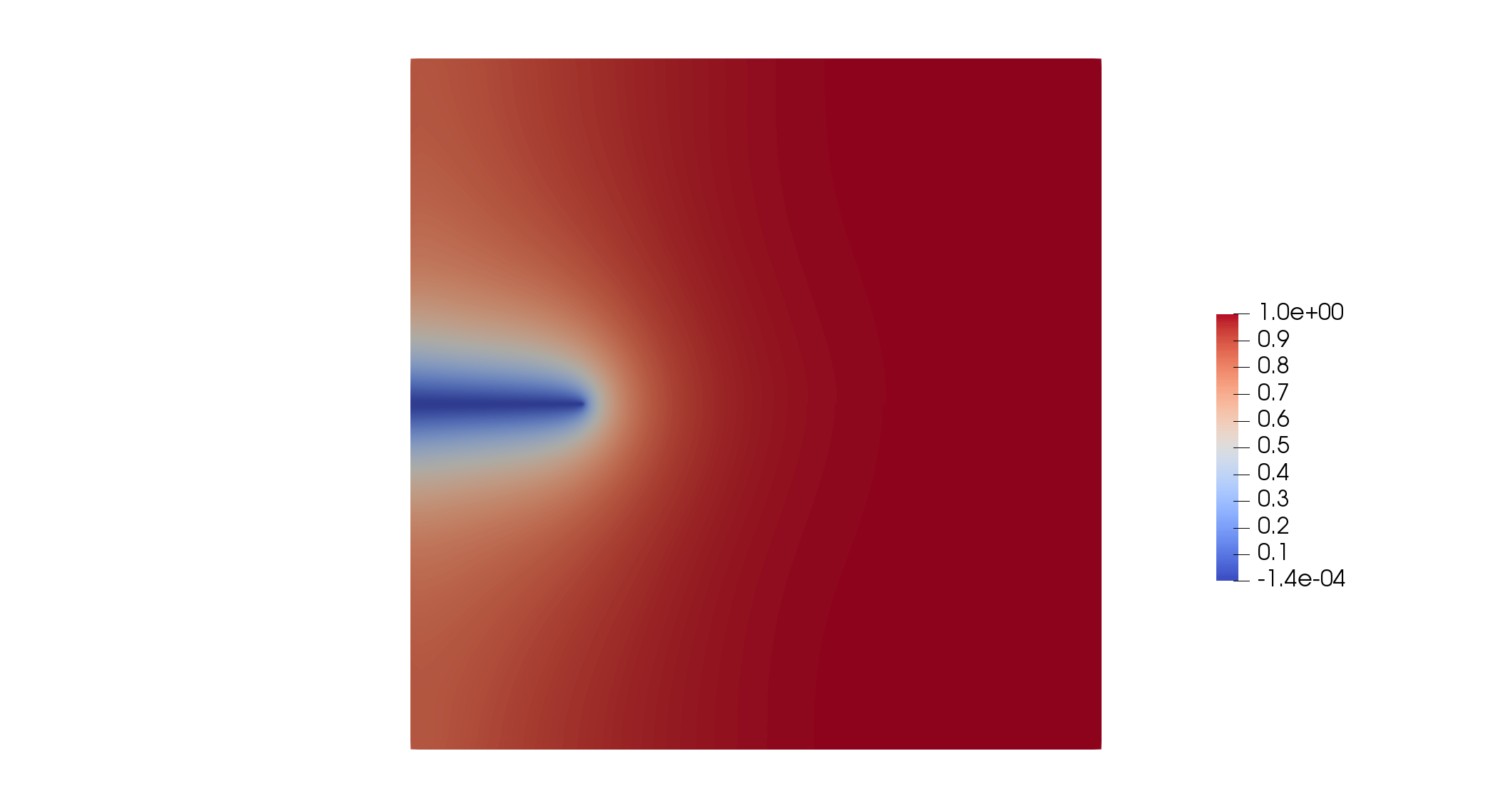}}
\caption{Tension test with $\epsilon=0.088$ at different time points after six adaptive refinement steps based on the estimator $\eta^{\varphi}$. Values of $\varphi \approx 1$ are colored in red, values $\varphi \approx 0$ are blue.}
\label{fig:tension_preRef4_Mesh6_Grid_etaPhi}
\end{figure}
\vspace*{-10mm}
\begin{figure}[H]
\centering
\subfloat[mesh at $n=120$]{\includegraphics[width=0.22\textwidth,clip,trim=20cm 4cm 20cm 4cm]{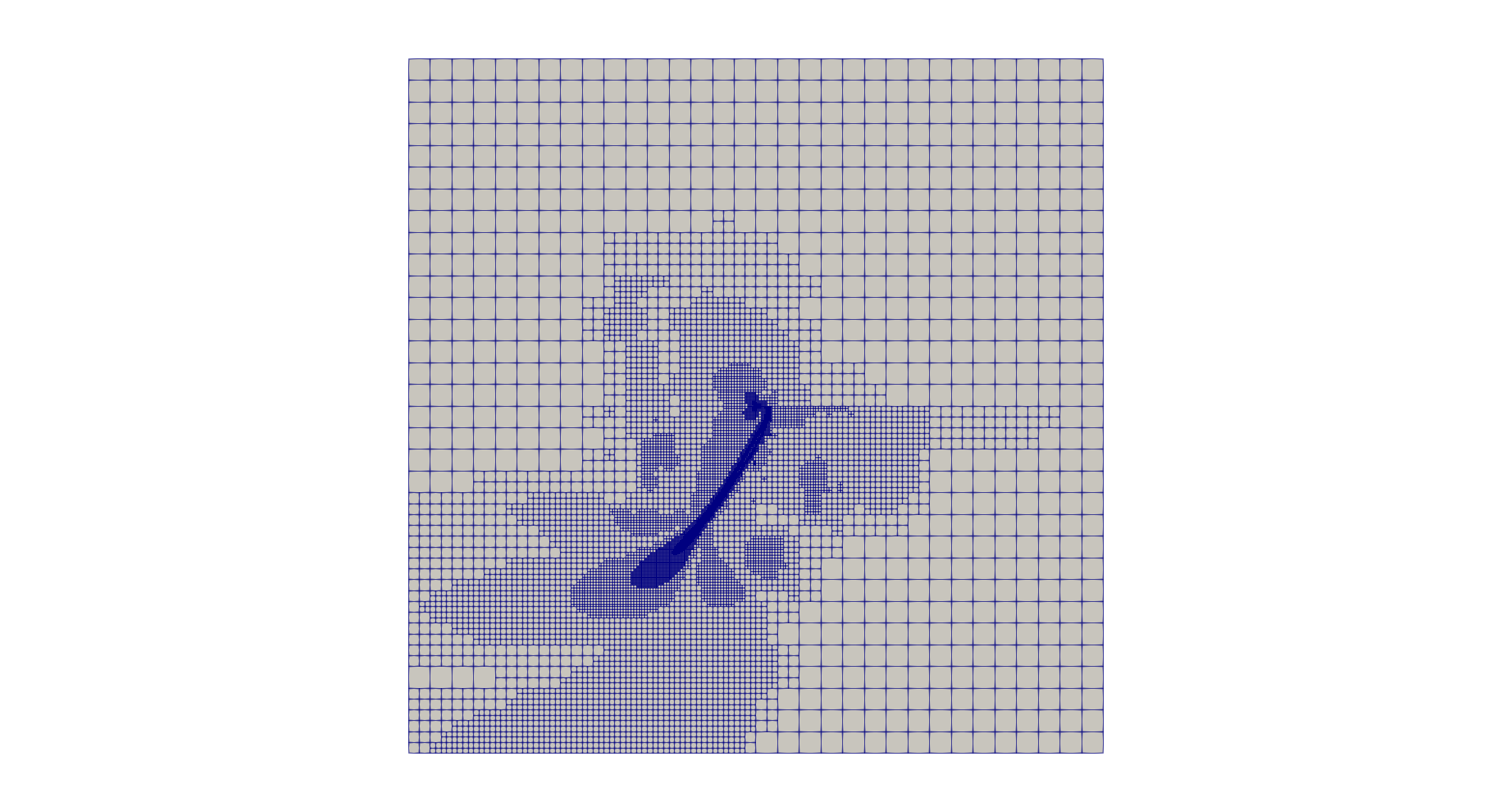}}\hfill
\subfloat[ $\varphi$ at $n=120$]{\includegraphics[width=0.22\textwidth,clip,trim=20cm 4cm 20cm 4cm]{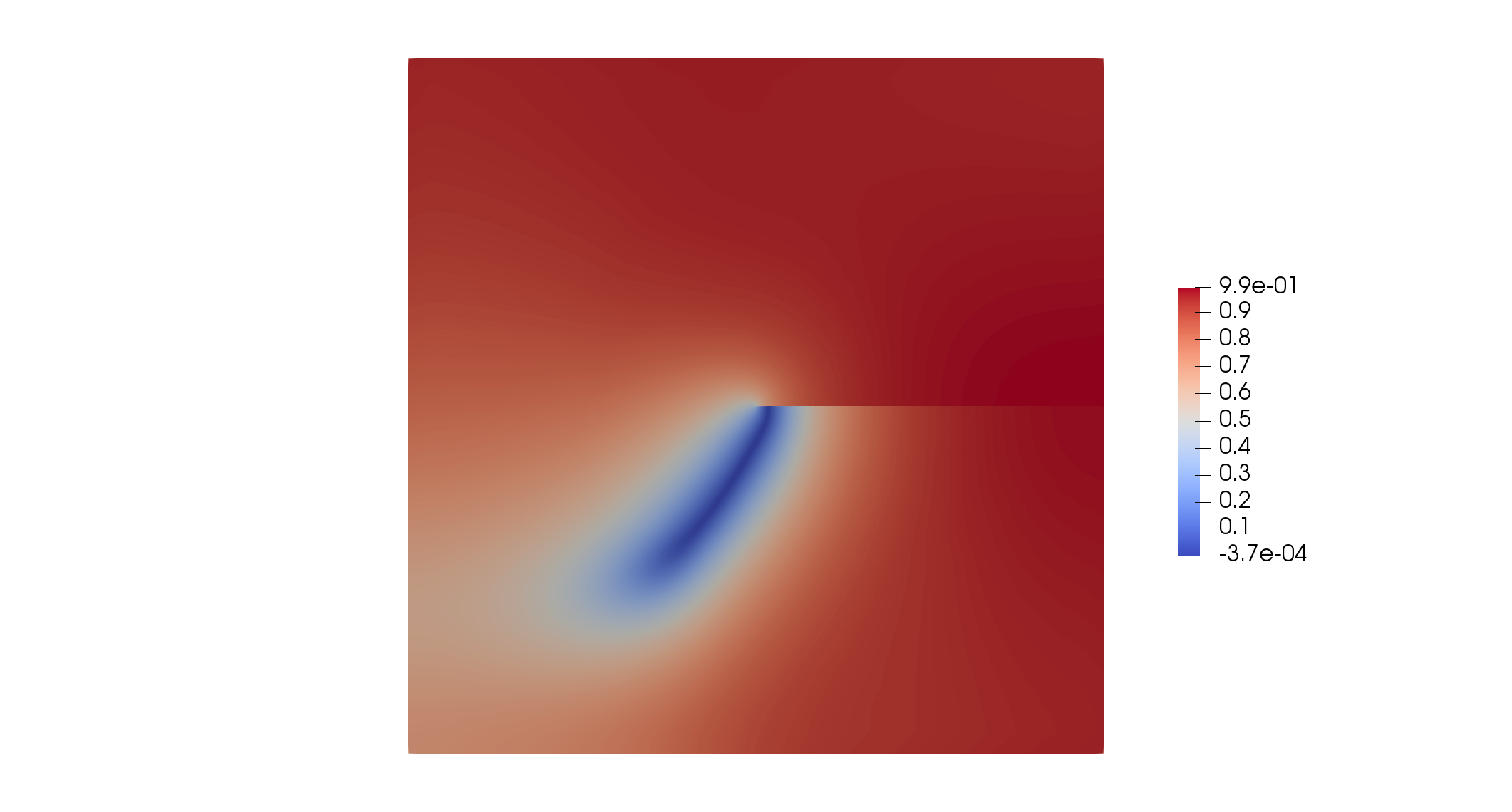}}\hfill
\subfloat[mesh at $n=132$]{\includegraphics[width=0.22\textwidth,clip,trim=20cm 4cm 20cm 4cm]{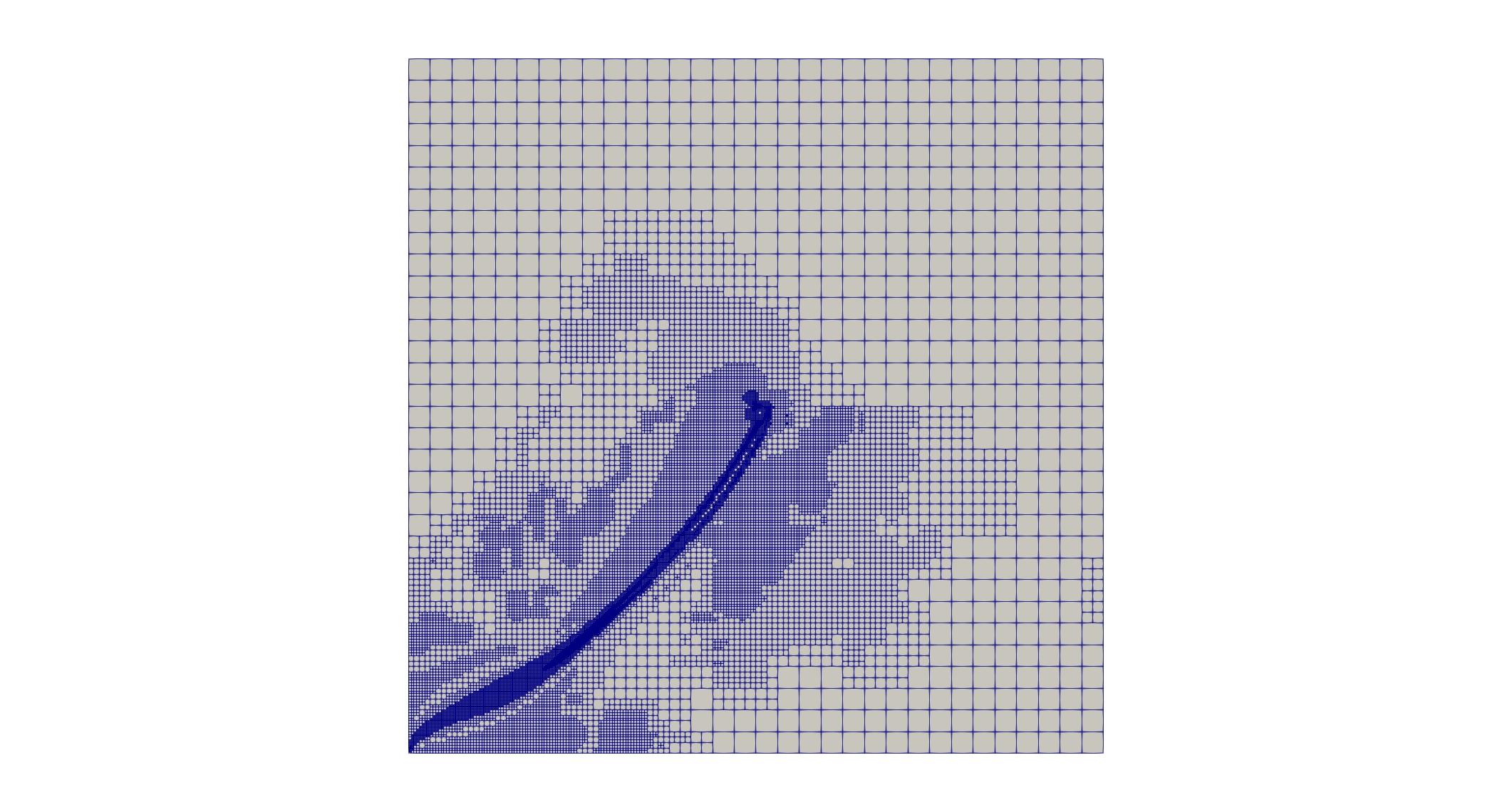}}\hfill
\subfloat[ $\varphi$ at $n=132$]{\includegraphics[width=0.22\textwidth,clip,trim=20cm 4cm 20cm 4cm]{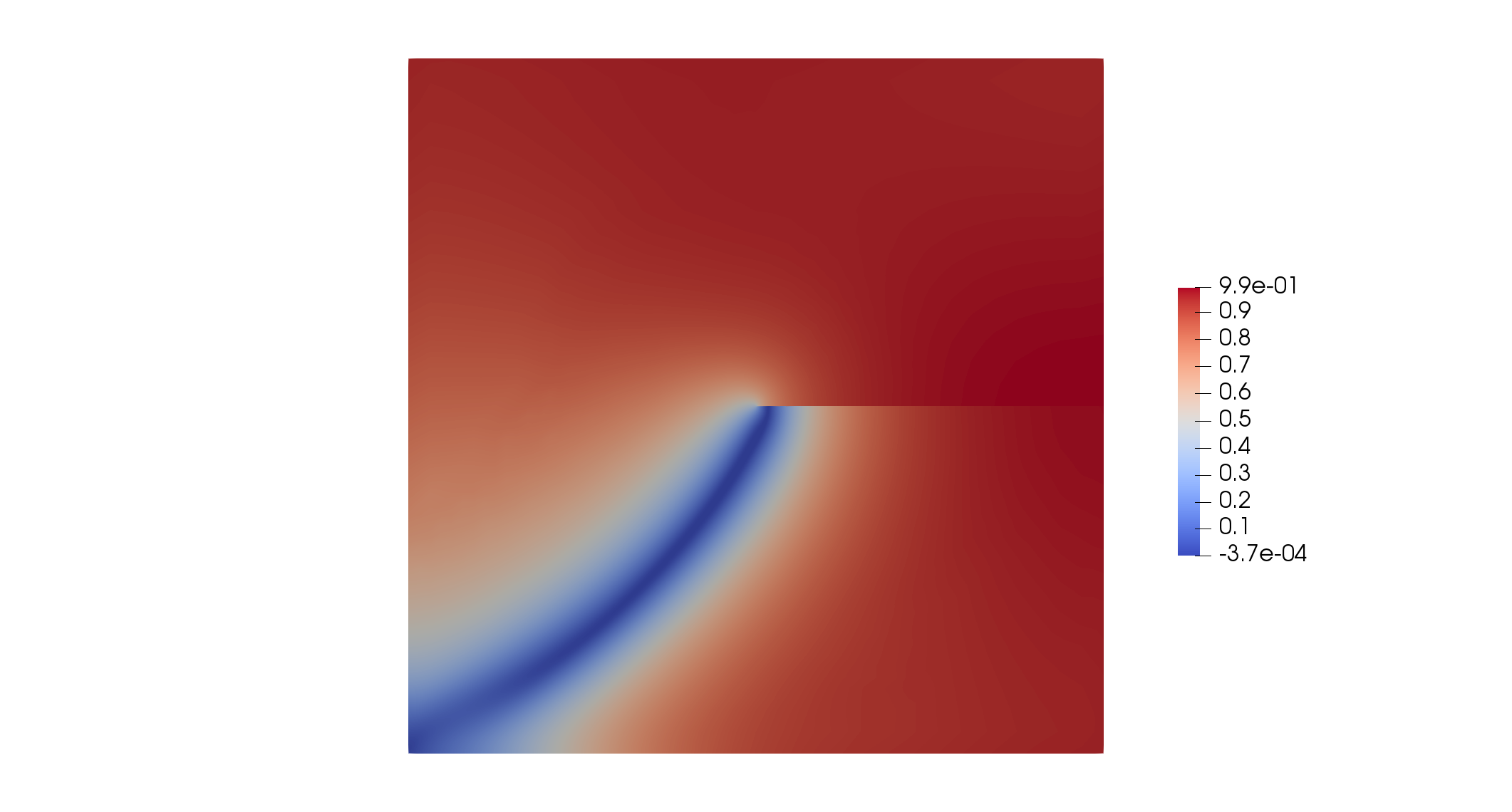}}
\caption{Shear test with $\epsilon=0.088$ at different time points after five adaptive refinement steps based on the estimator $\eta^{\varphi}$. Values of $\varphi \approx 1$ are colored in red, values $\varphi \approx 0$ are blue.}
\label{fig:shear_preRef4_Mesh5_Grid_etaPhi}
\end{figure}
\vspace*{-10mm}
\begin{figure}[H]
\centering
\hfill
\subfloat[mesh at $n=300$]{\includegraphics[width=0.2\textwidth,clip,trim=20cm 4cm 20cm 4cm]{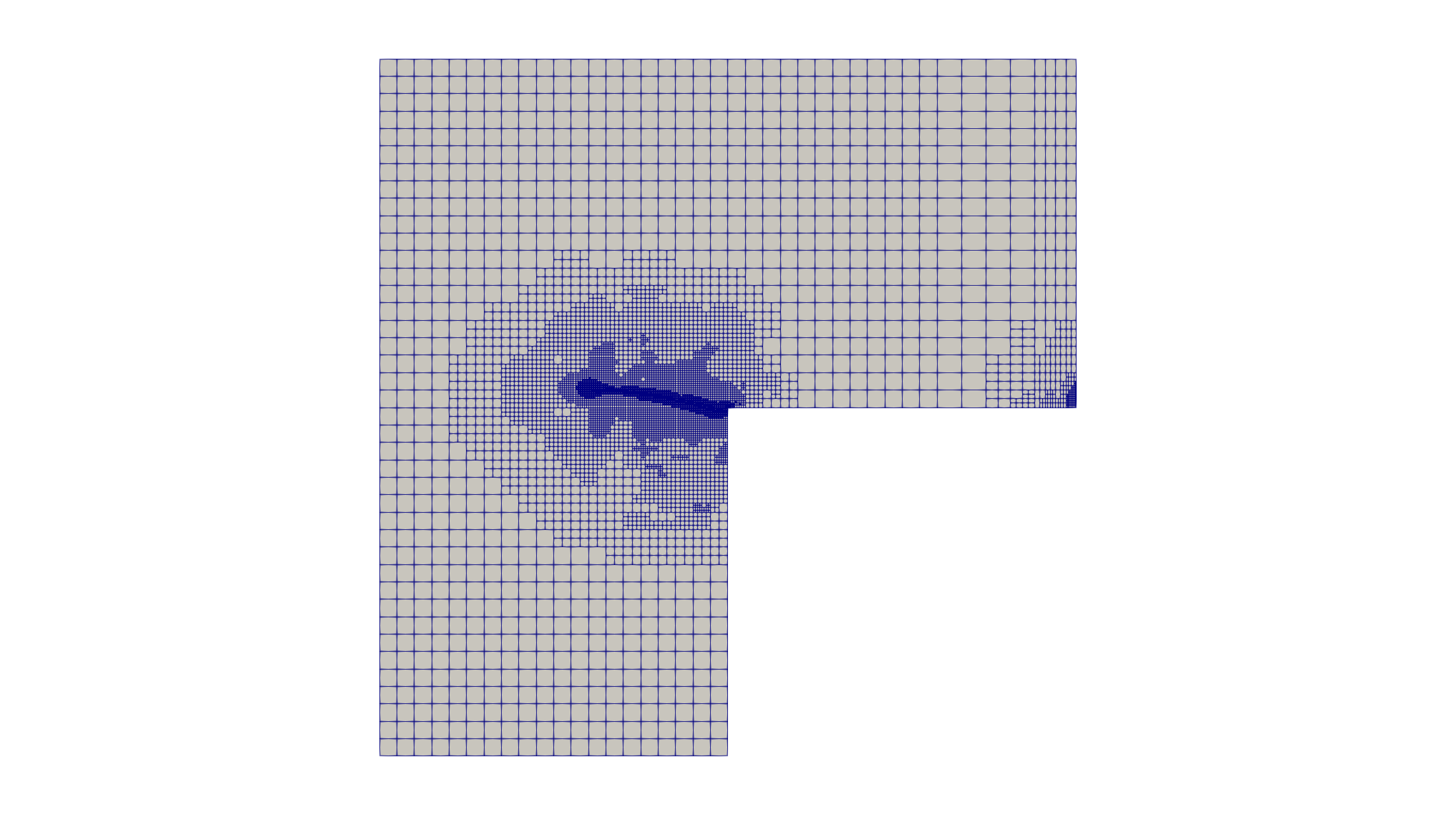}}\hfill\hfill
\subfloat[ $\varphi$ at $n=300$]{\includegraphics[width=0.2\textwidth,clip,trim=20cm 4cm 20cm 4cm]{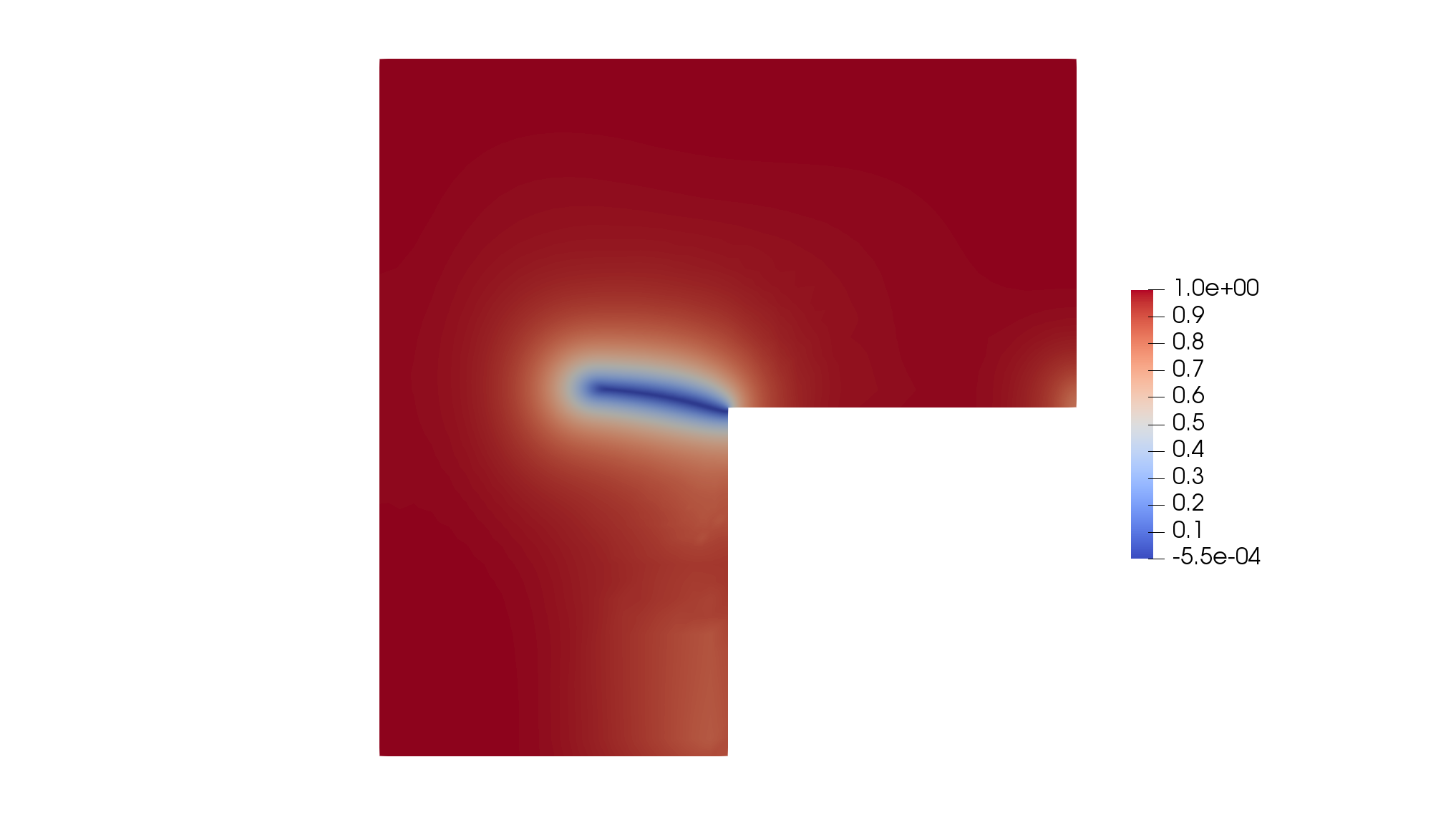}}\hfill \ 
\caption{L-shape test with $\epsilon=20$ at different time points after five adaptive refinement steps based on the estimator $\eta^{\varphi}$. Values of $\varphi \approx 1$ are colored in red, values $\varphi \approx 0$ are blue.}
\label{fig:Lshape_Mesh5_Grid_etaPhi}
\end{figure}
\vspace*{-10mm}

\subsubsection{Crack and bulk energy and load-displacement curves}

Further, we show the crack and the bulk energy as well as plots for the load-displacement curves for all three tests in Figures~\ref{fig:Tension_Quantities},~\ref{fig:Shear_Quantities} and~\ref{fig:Lshape_Quantities}. The curves converge with the adaptive refinement.
\vspace*{-5mm}
\input{Numerik-Fig4}
\vspace*{-10mm}
\input{Numerik-Fig5}
\vspace*{-10mm}
\input{Numerik-Fig6}
\vspace*{-5mm}

\subsubsection{Convergence in different error norms}

To demonstrate the convergence behavior of the errors, we compute reference solutions $\bu{u}^n$, $\ul{\varphi}^n$ on a finer mesh which has been at least three times more uniformly refinement than the adaptive meshes on which the solutions $\boldsymbol{u}^n_{\mathfrak{m}}$ and 
$\varphi_{\mathfrak{m}}^n$ have been computed. 

To measure the errors in $\varphi$  and $\boldsymbol{u}$, we use the
energy norm in $\varphi$
given by 
\begin{equation*}
\|\ul{\varphi}^n- \varphi_{\mathfrak{m}}^n\|_{\epsilon}^2
 = G_c\epsilon\|\nabla(\ul{\varphi}^n-\varphi_{\mathfrak{m}}^n)\|^2 + \|\left(\frac{G_c}{\epsilon} + (1-\kappa)\boldsymbol{\sigma}(\bu{u}^n):\boldsymbol{E}_{\mathrm{lin}}(\bu{u}^n)\right)^{\frac{1}{2}}(\ul{\varphi}^n- \varphi_{\mathfrak{m}}^n)\|^2 
\end{equation*}
and the energy norm in $\boldsymbol{u}$ given by
\begin{equation*}
\|\bu{u}^n-\boldsymbol{u}^n_{\mathfrak{m}}\|_{Eu}^2 :=\int_{\Omega}g(\varphi^n_{\mathfrak{m}})\boldsymbol{\sigma}(\bu{u}^n-\boldsymbol{u}^n_{\mathfrak{m}}):\boldsymbol{E}_{\mathrm{lin}}(\bu{u}^n-\boldsymbol{u}^n_{\mathfrak{m}})
\end{equation*}
As expected the adaptive refinement gives rise to a stronger error reduction for the error in $\varphi$ than the uniform refinement. But this does not only hold for the error in $\varphi$  but also for the error in $\boldsymbol{u}$ although the adaptive refinement has been steered by the estimator $\eta^{\varphi}$. 
\vspace*{-5mm}
\input{Numerik-Fig7}
\vspace*{-15mm}
\input{Numerik-Fig8}
\vspace*{-15mm}
\input{Numerik-Fig9}
\vspace*{-10mm}

\subsubsection{Efficiency index}
In this subsection, we visualize the efficiency index, i.e., the
quotient of $\eta^{\varphi}$ and the energy norm $\|\ul{\varphi}^n-
\varphi_{\mathfrak{m}}^n\|_{\epsilon}$.
We  compare it to the efficiency index for a non-robust residual
estimator which can be easily derived without taking care of the
aspect of robustness.
For this we derived a residual-type a posteriori estimator for Problem \ref{AuxiliaryProblemCont} with respect to the $H^1$-norm of the error, not 
paying attention to the $\epsilon$-dependency. The derivation
basically follows along the lines of Section \ref{Sec:Reliability} and
\ref{Sec:Efficiency}. The proofs would be simplified as the standard versions of the $L^2$-approximation and the bubble functions can be used. Thus, instead of the energy norm the $H^1$-norm is taken whenever calculating efficiencies for the non-robust estimator.
\vspace*{-5mm}
\input{Numerik-Fig10}
\vspace*{-10mm}
\input{Numerik-Fig11}
\vspace*{-5mm}
As is clearly visible in Figures~\ref{fig:Efficiency_Tension}
and~\ref{fig:Efficiency_shear},
the efficiency indices for the new estimator are robust with respect
to the variation of $\epsilon$ while the efficiency tends to zero for
the standard estimator.

\subsection{Adaptive refinement using both estimators $\eta^u$ and $\eta^{\varphi}$}

In this subsection, we investigate the adaptive refinement which is steered by both estimators, the estimator $\eta^{\varphi}$ from Section~\ref{Sec:EstimatorVI} and $\eta^u$ from Section~\ref{Sec:EstimatorEqu}. In the implementation, we normalize both estimators and add them before the marking strategy is called.  
In the following we solely show numerical results for the tension and
the shear tests because the L-shape test is modeled by inhomogeneous
Dirichlet boundary conditions on a small portion of the boundary. The
estimator $\eta^u$  correctly identifies the singularity induced by
this boundary condition and resolves the resulting singularity.
This is reasonable for the discretization error but contains a model error as the Dirichlet conditions imitate that the area is vertically clamped, see~\cite{Winkler2001}.

\subsubsection{Adaptively refined grids}

Comparing Figures~\ref{fig:tension_preRef4_Mesh6_Grid_etaUPhi} and~\ref{fig:shear_preRef4_Mesh5_Grid_etaUPhi}  with the Figures~\ref{fig:tension_preRef4_Mesh6_Grid_etaPhi} and~\ref{fig:shear_preRef4_Mesh5_Grid_etaPhi}, we see that the adaptive refinement is different. 
While for the tension test in
Figure~\ref{fig:tension_preRef4_Mesh6_Grid_etaUPhi} the crack path is
still well resolved, the influence of the estimator $\eta^u$ is
stronger for the shear test and leads to a strong refinement of the
origin of the crack and thus to less refinement of the crack path.
\vspace*{-5mm}
\begin{figure}[H]
\centering
\hfill
\subfloat[mesh at $n=310$]{\includegraphics[width=0.22\textwidth,clip,trim=20cm 4cm 20cm 4cm]{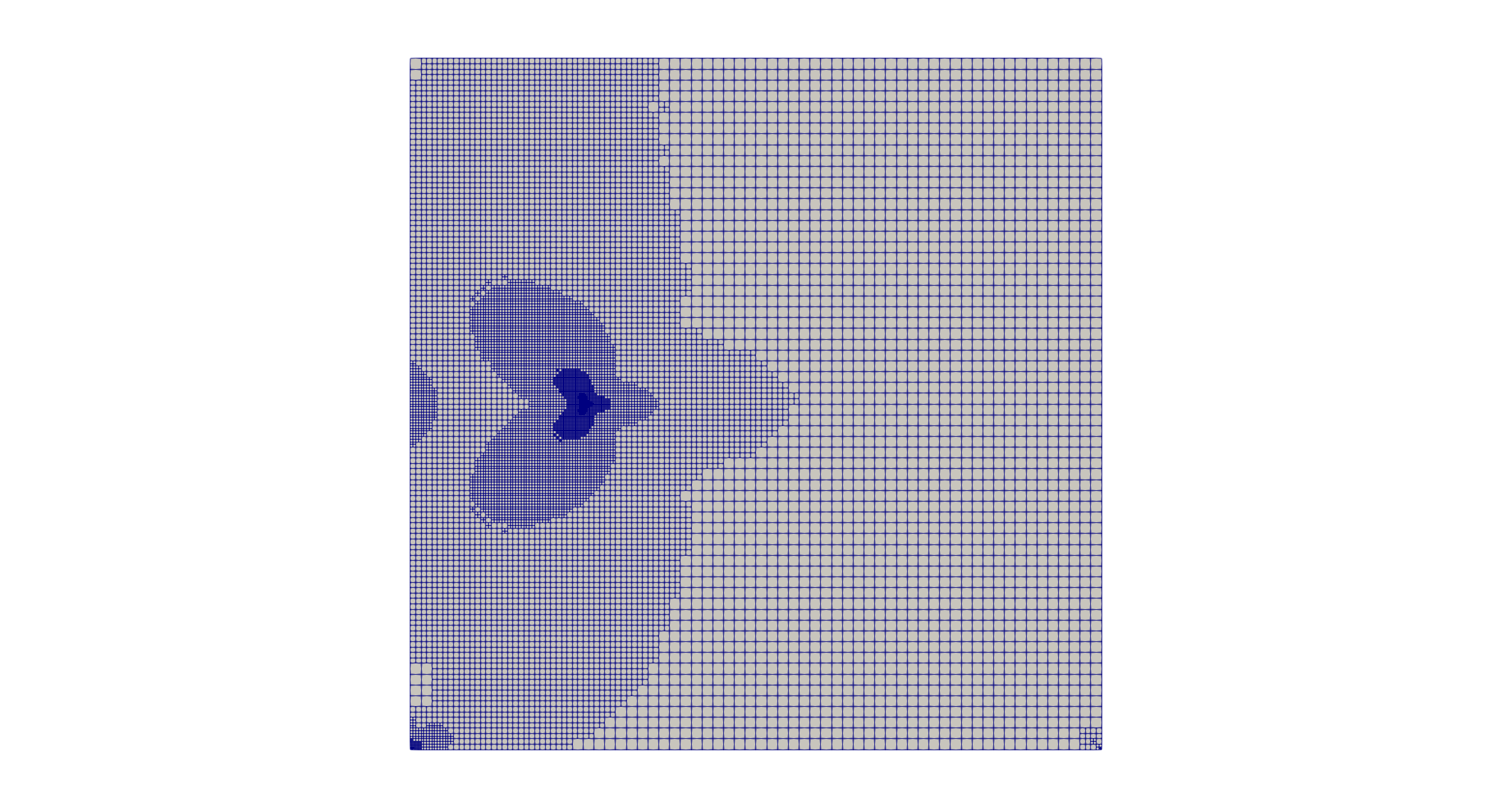}}\hfill
\subfloat[$\varphi$ at $n=310$]{\includegraphics[width=0.22\textwidth,clip,trim=20cm 4cm 20cm 4cm]{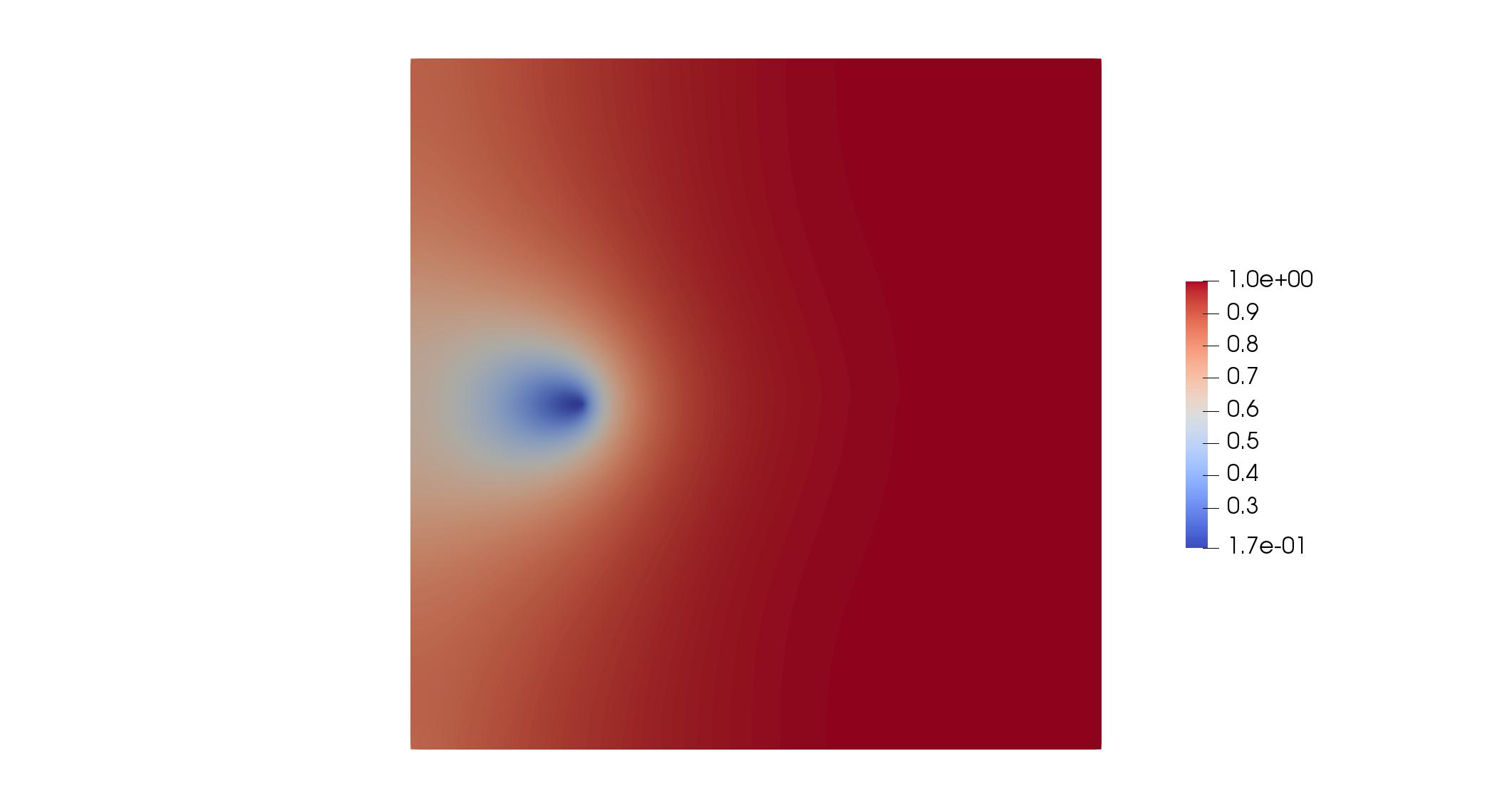}}\hfill
\subfloat[$u_1$ at $n=310$]{\includegraphics[width=0.22\textwidth,clip,trim=20cm 4cm 20cm 4cm]{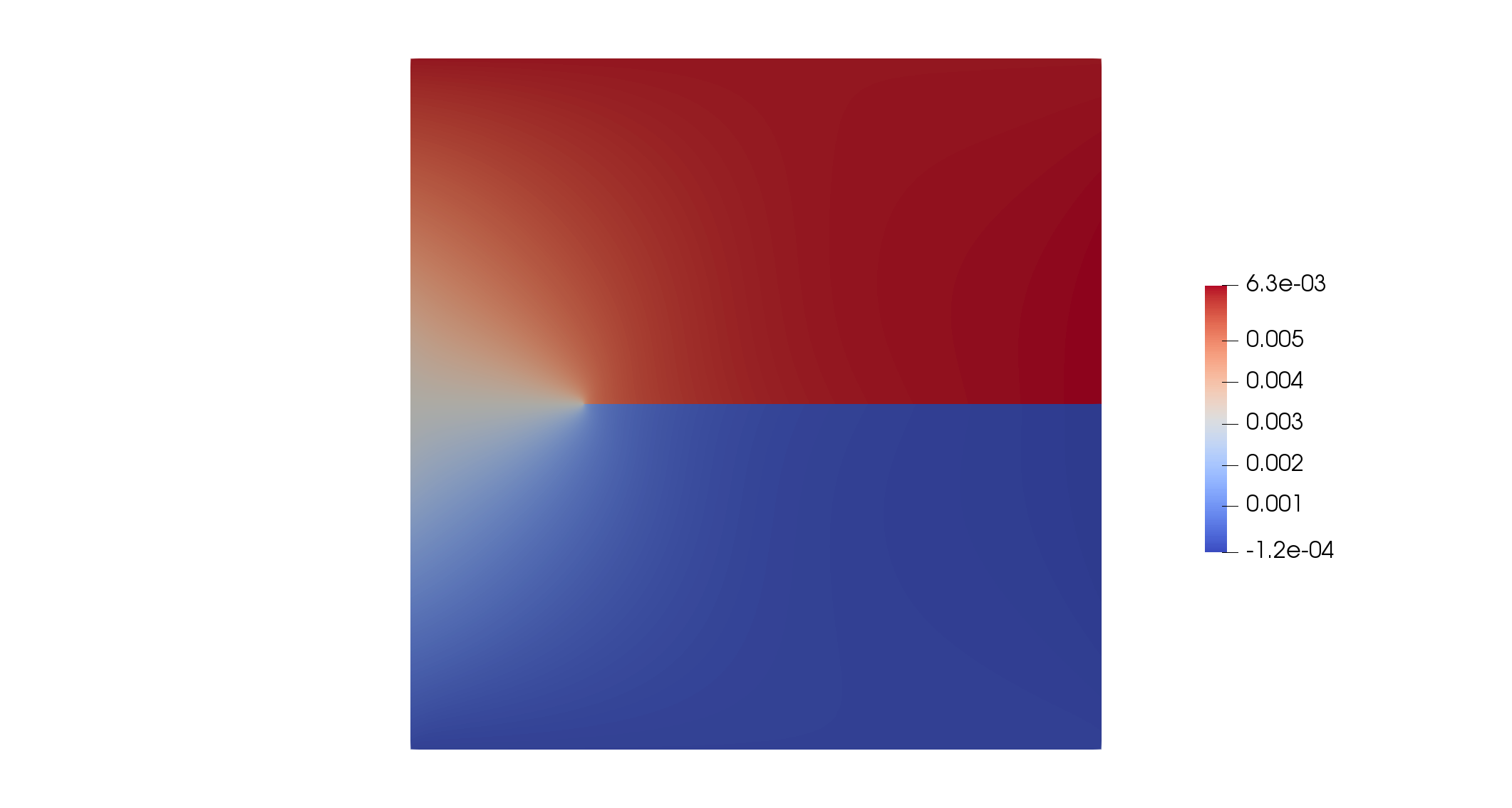}}\hfill \ \ \\
\hfill
\subfloat[mesh at $n=324$]{\includegraphics[width=0.22\textwidth,clip,trim=20cm 4cm 20cm 4cm]{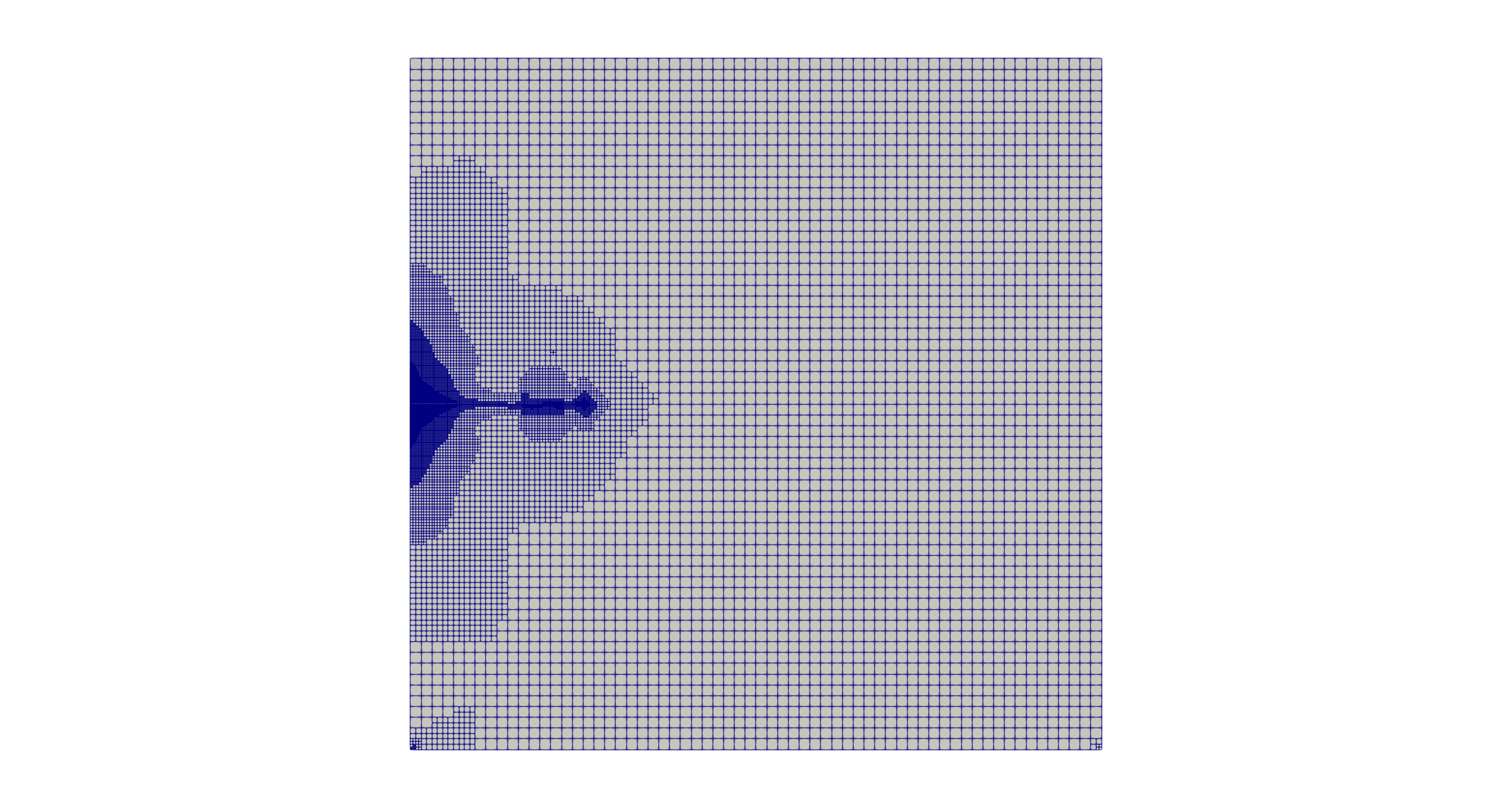}}\hfill
\subfloat[$\varphi$ at $n=324$]{\includegraphics[width=0.22\textwidth,clip,trim=20cm 4cm 20cm 4cm]{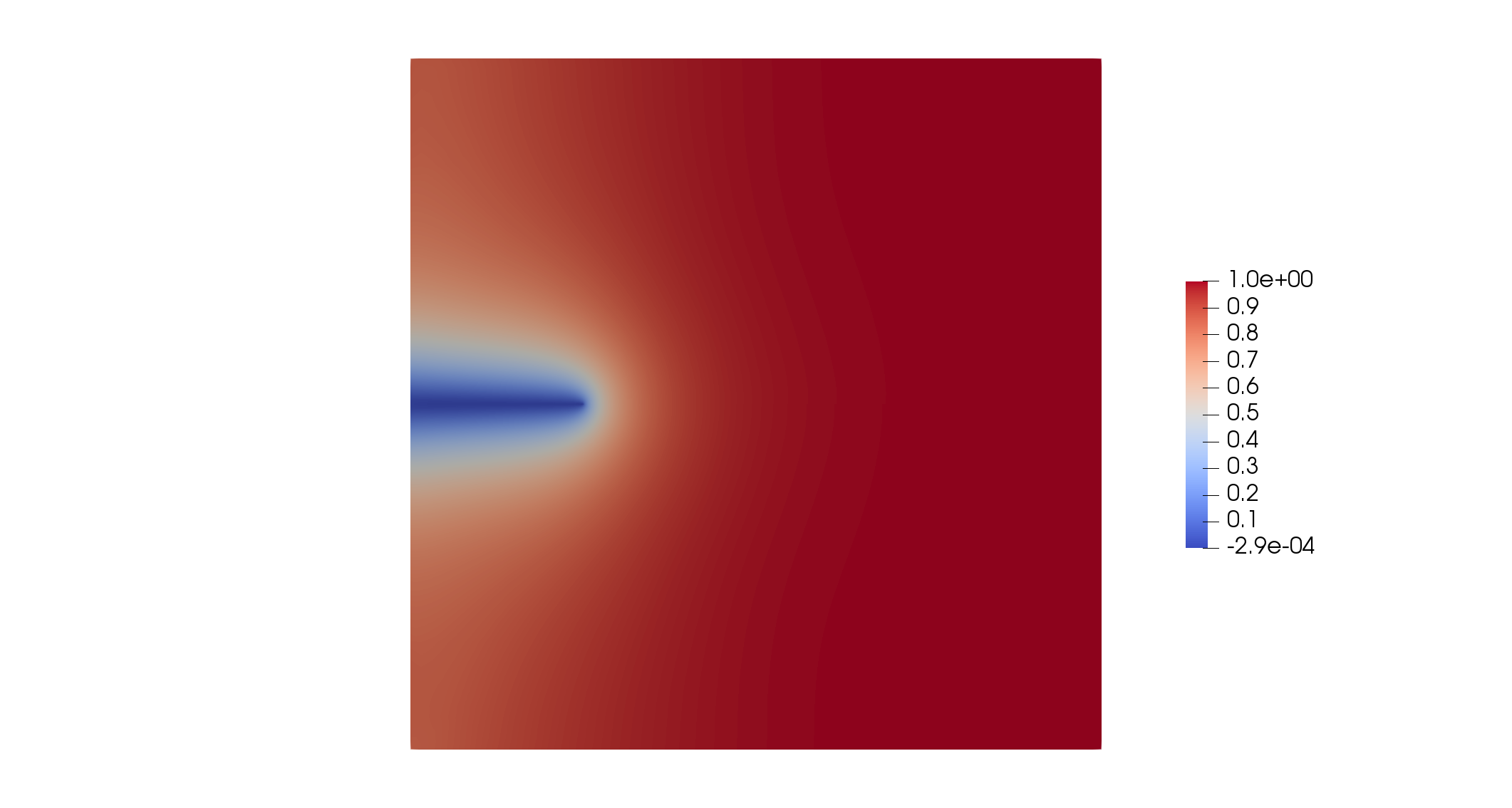}}\hfill
\subfloat[$u_1$ at $n=324$]{\includegraphics[width=0.22\textwidth,clip,trim=20cm 4cm 20cm 4cm]{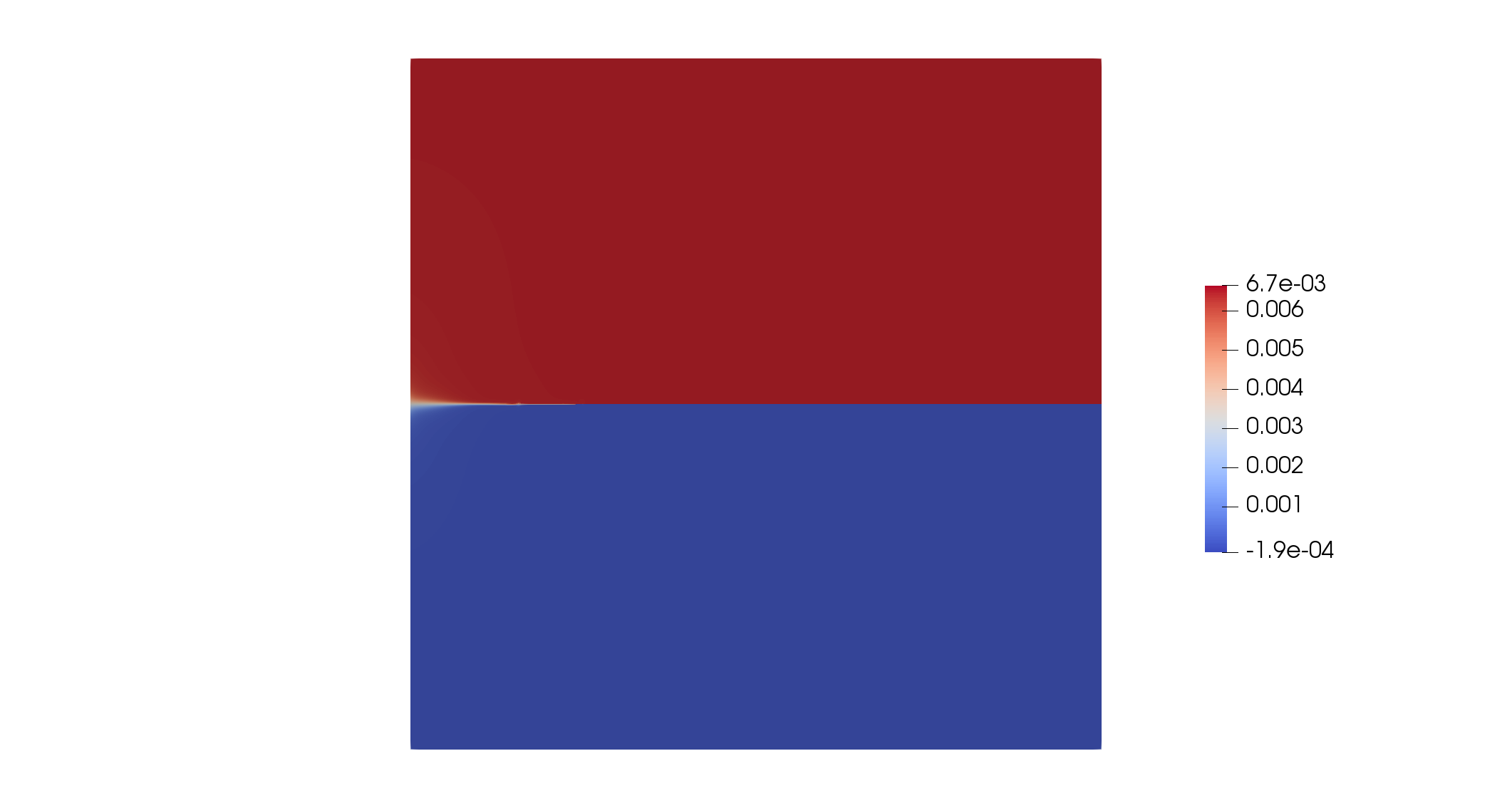}}\hfill \ \ \\
\caption{Tension test with $\epsilon=0.088$ at different time points after six adaptive refinement steps based on the estimators $\eta^{\varphi}$ and $\eta^u$.  Values of $\varphi \approx 1$ and $u_1 \gg 0$ are colored in red, values $\varphi \approx 0$ and $u_1 \ll 0$ are blue.}
\label{fig:tension_preRef4_Mesh6_Grid_etaUPhi}
\end{figure}
\vspace*{-15mm}
\begin{figure}[H]
\centering
\hfill
\subfloat[mesh at $n=120$]{\includegraphics[width=0.22\textwidth,clip,trim=20cm 4cm 20cm 4cm]{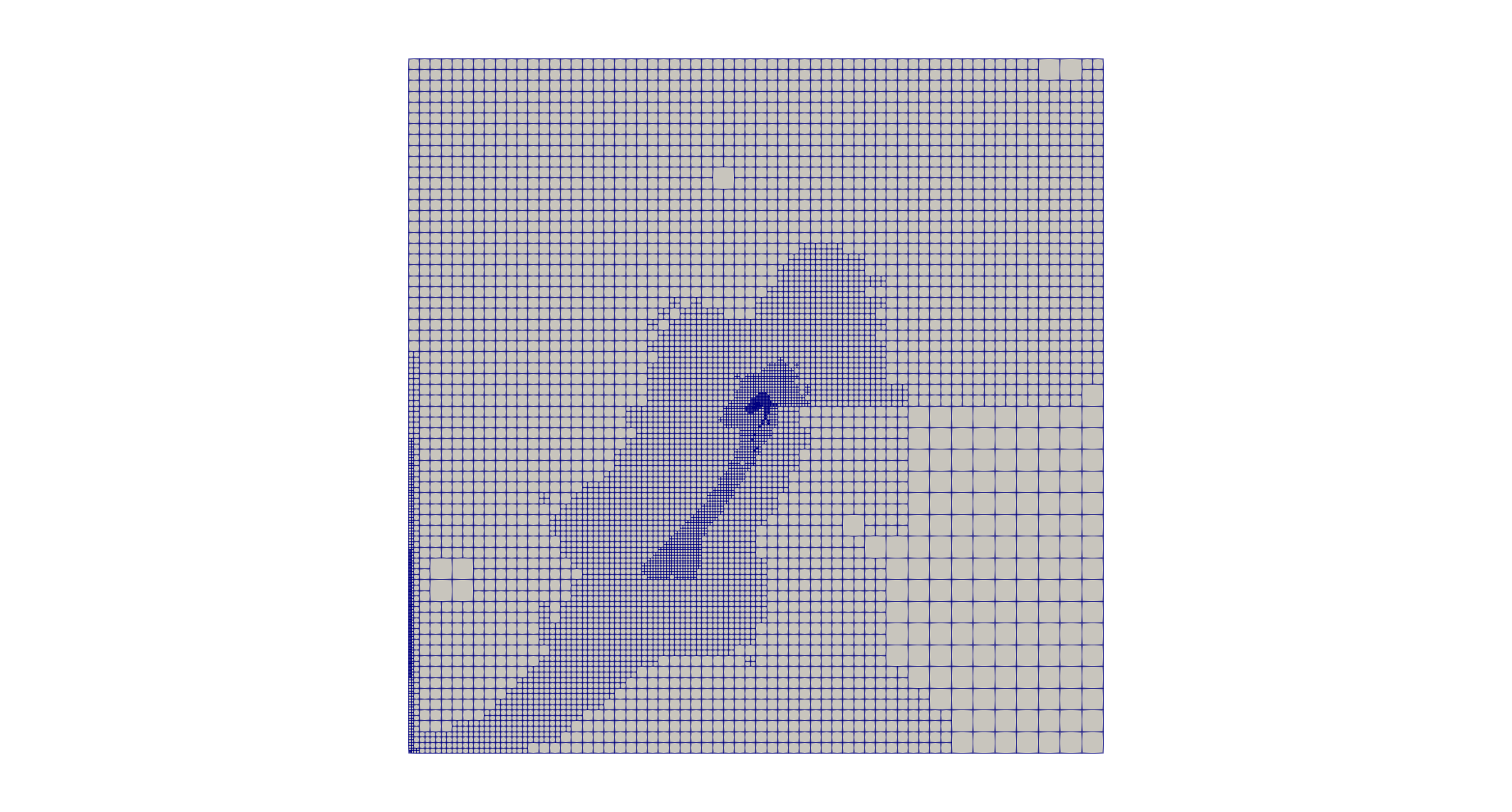}}\hfill
\subfloat[$\varphi$ at  $n=120$]{\includegraphics[width=0.22\textwidth,clip,trim=20cm 4cm 20cm 4cm]{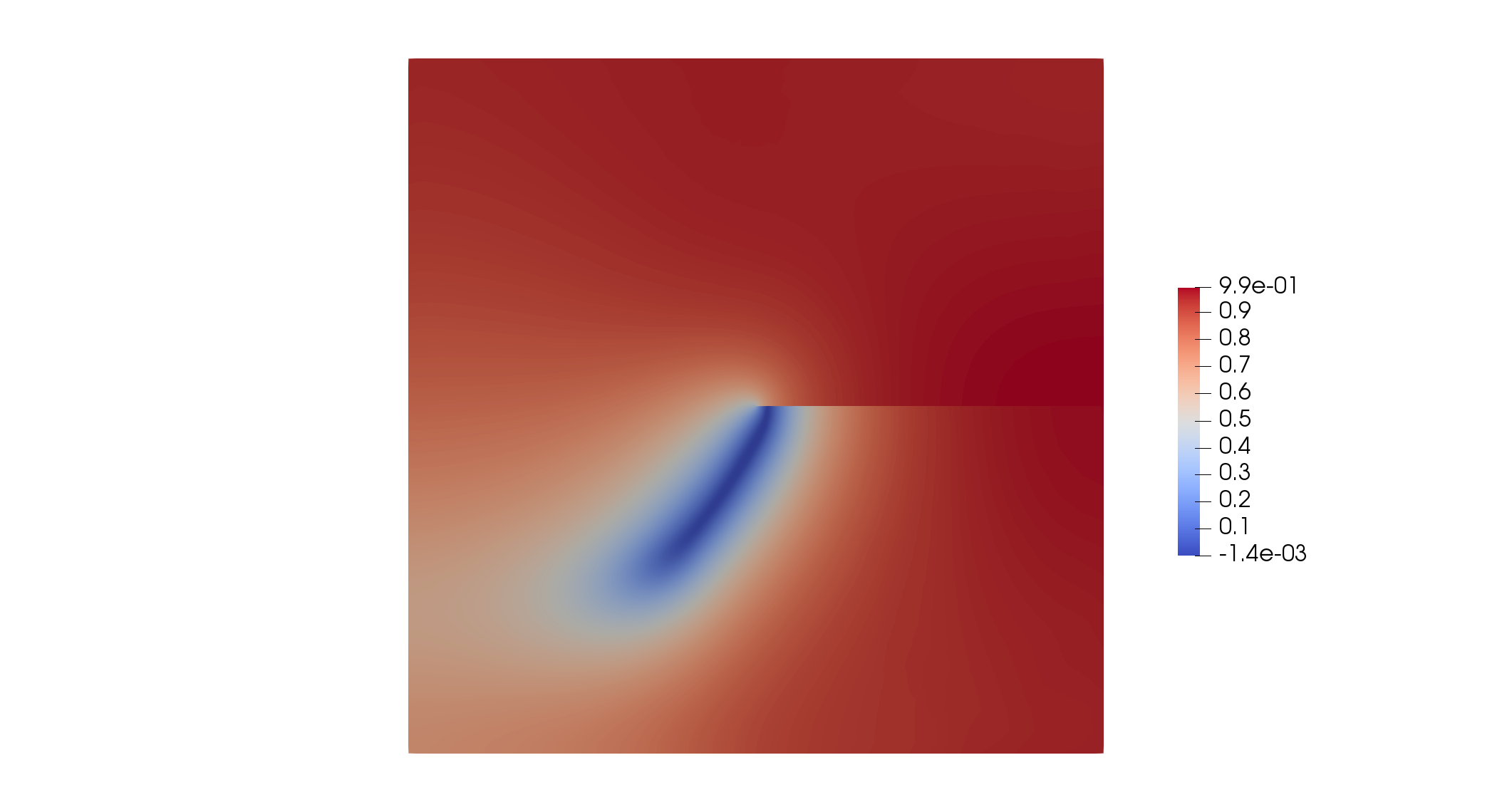}}\hfill
\subfloat[$u_1$ at $n=120$]{\includegraphics[width=0.22\textwidth,clip,trim=20cm 4cm 20cm 4cm]{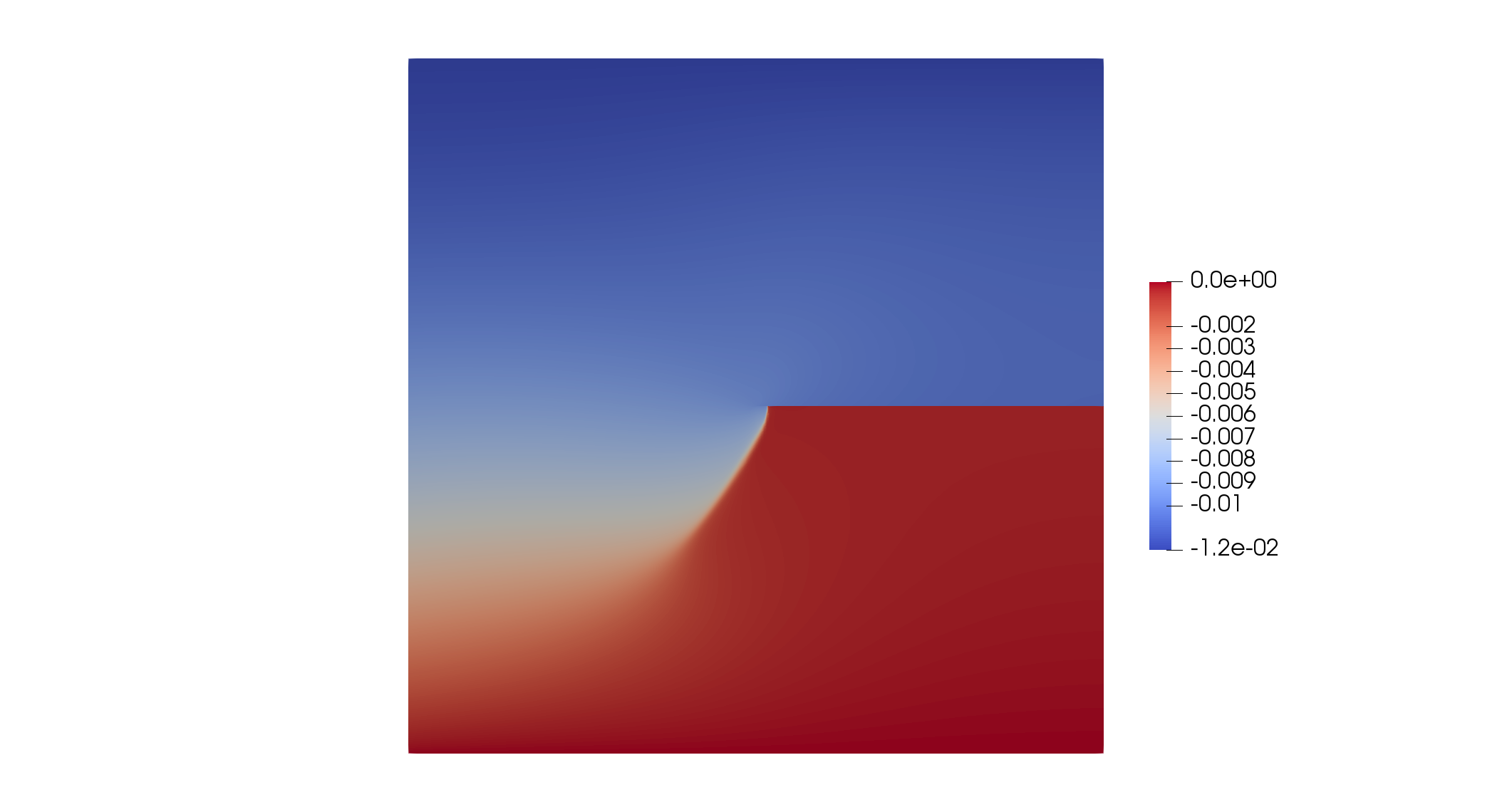}}\hfill \ \ \\
\hfill
\subfloat[mesh at $n=132$]{\includegraphics[width=0.22\textwidth,clip,trim=20cm 4cm 20cm 4cm]{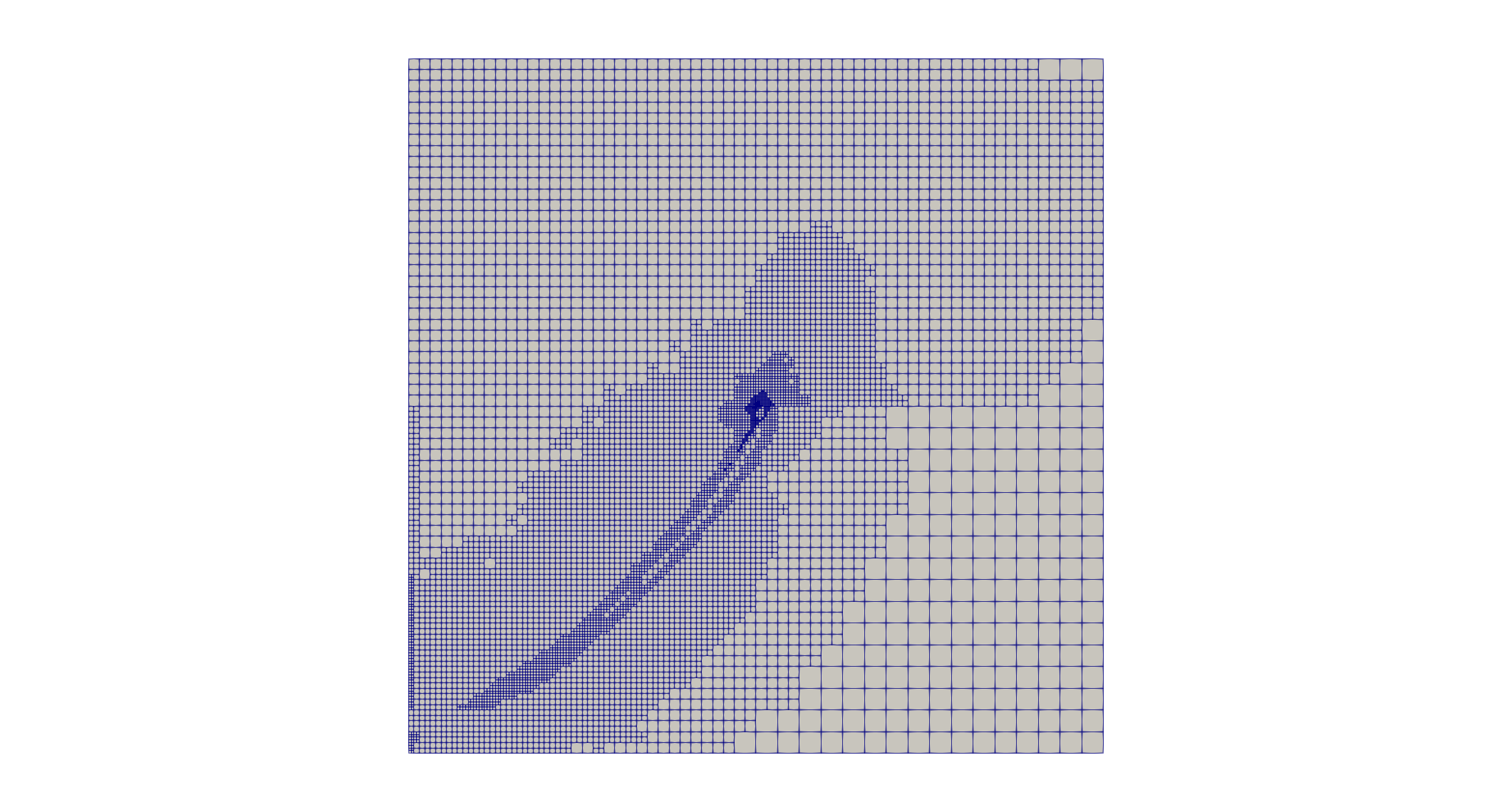}}\hfill
\subfloat[$\varphi$ at $n=132$]{\includegraphics[width=0.22\textwidth,clip,trim=20cm 4cm 20cm 4cm]{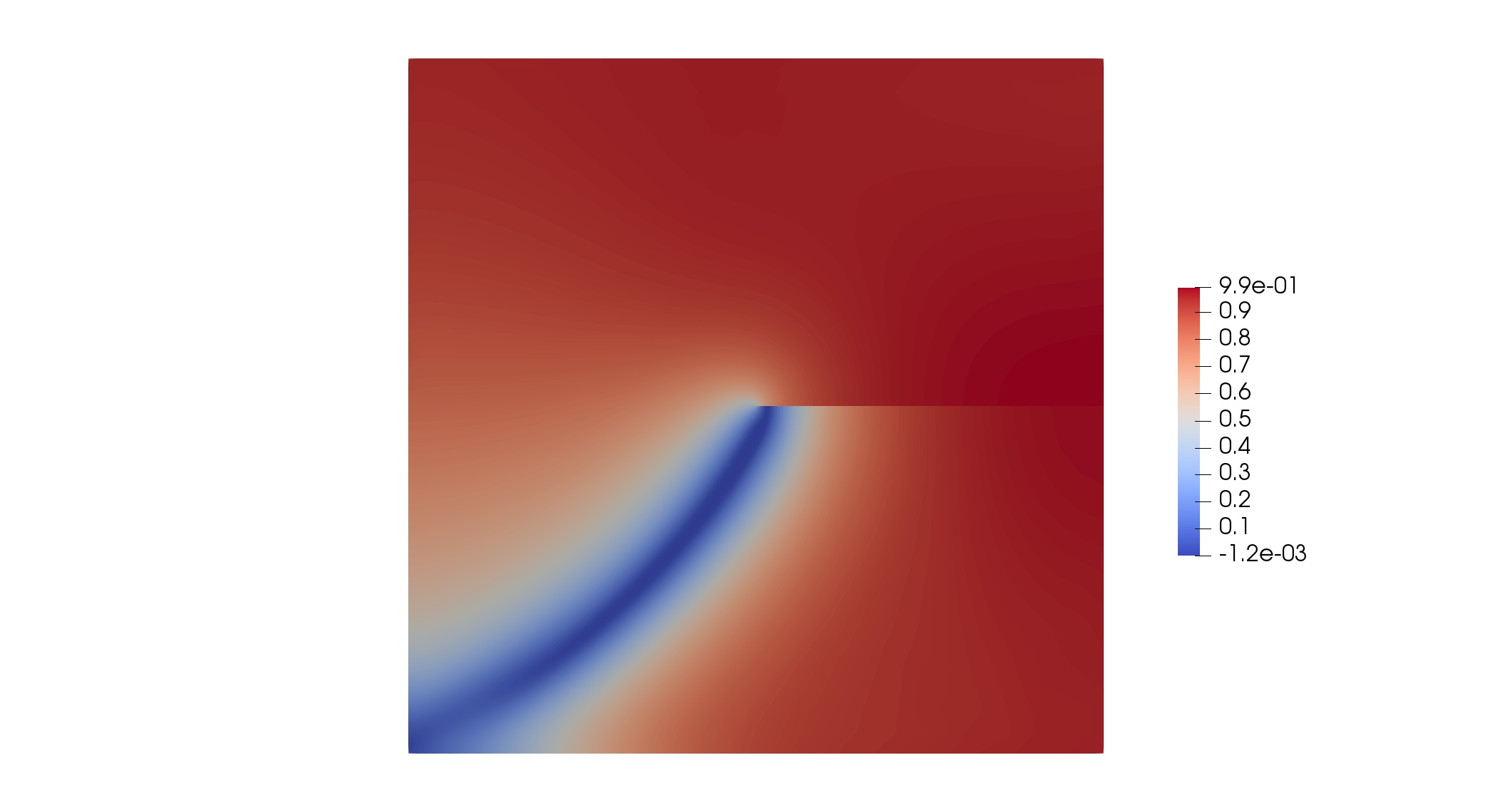}}\hfill
\subfloat[$u_1$ at $n=132$]{\includegraphics[width=0.22\textwidth,clip,trim=20cm 4cm 20cm 4cm]{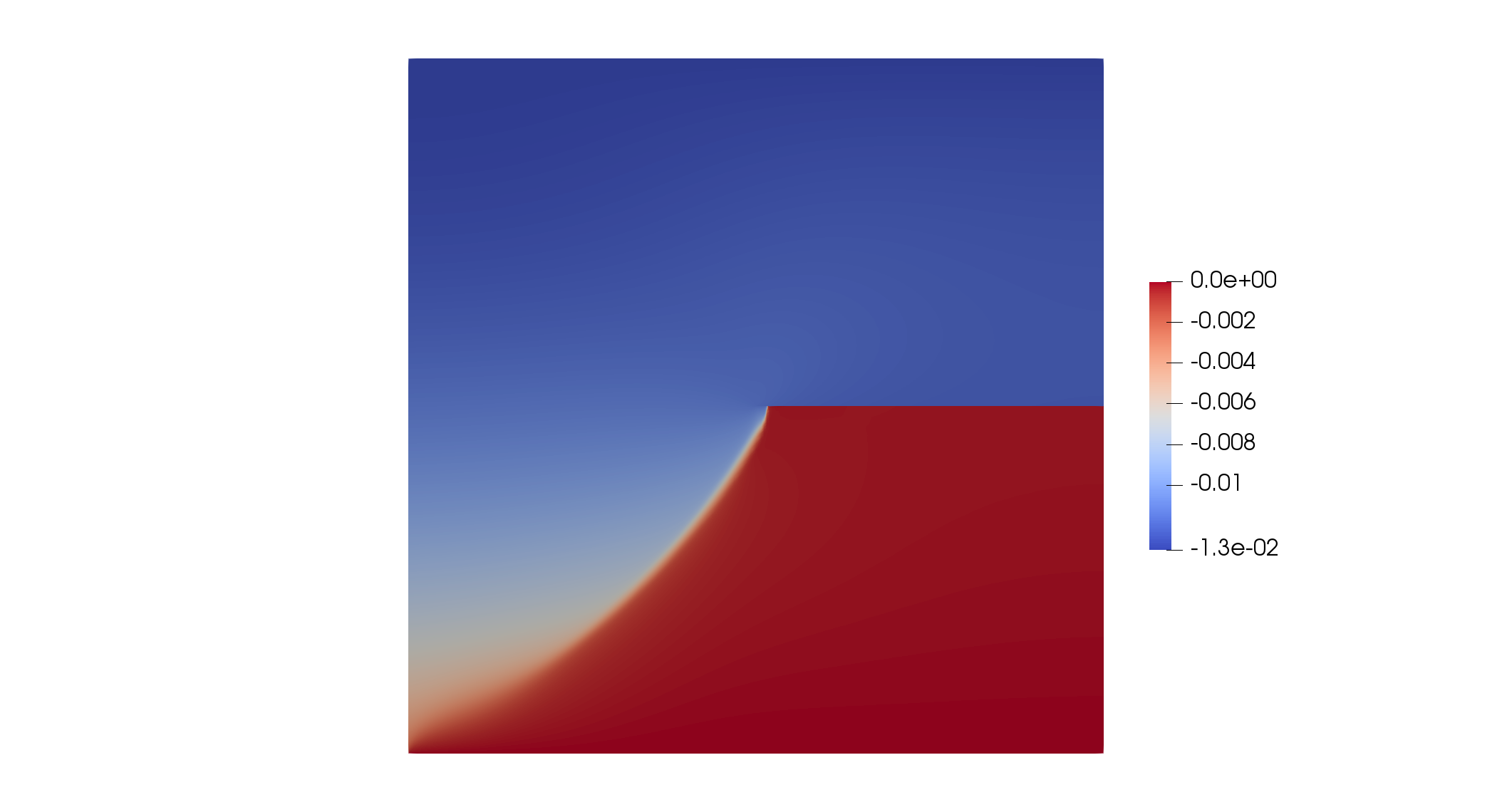}}\hfill \ \ \\
\caption{Shear test with $\epsilon=0.088$  at different time points after six adaptive refinement steps based on the estimators $\eta^{\varphi}$ and $\eta^u$.  Values of $\varphi \approx 1$ and $u_1 \gg 0$ are colored in red, values $\varphi \approx 0$ and $u_1 \ll 0$ are blue.}
\label{fig:shear_preRef4_Mesh5_Grid_etaUPhi}
\end{figure}
\vspace*{-15mm}

\subsubsection{Convergence in different error norms}

Finally, we show the convergence behavior for the tension and shear test using adaptive refinement steered by both estimators $\eta^{\varphi}$ and $\eta^u$ compared to uniform refinement. Especially for the shear test it is obvious that the influence of the estimator $\eta^u$ improve the convergence order for the error in $\boldsymbol{u}$.
\vspace*{-5mm}
\input{Numerik-Fig14}
\vspace*{-10mm}
\input{Numerik-Fig15}
\vspace*{-5mm}

%% file: Numerik-Fig4.tex
\begin{figure}[H]
    \definecolor{col0}{rgb}{0,0,0}
    \definecolor{col1}{rgb}{0.2,0.3,1}
    \definecolor{col2}{rgb}{0.5,1,1}
    \definecolor{col3}{rgb}{0.2,0.8,0.5}
    \definecolor{col4}{rgb}{0.9,0.8,0.3}
    \definecolor{col5}{rgb}{0.9,0.5,0.1}
    \definecolor{col6}{rgb}{0.9,0.1,0}
    \definecolor{col7}{rgb}{0.4,0.2,0.9}
  \centering
  %Crack Energy
\subfloat[crack energy]{\begin{tikzpicture}
    \begin{axis}[clip marker paths=true,
      ylabel near ticks, xlabel near ticks,
      ylabel={crack energy},
      xtick={0,0.0005,0.001,0.0015,0.002,0.0025,0.003,0.0035},
      xtick scale label code/.code={\pgfmathparse{int(-#1)} time $\cdot 10^{-\pgfmathresult}$},
      every x tick scale label/.style={at={(0.5,-0.1)},anchor=north},
      ytick={0,0.2,0.4,0.6,0.8,1,1.2},
      width=0.32\textwidth,font=\scriptsize,legend cell
      align=left, xmin=0, xmax=0.0037, ymin=0, ymax=1.2]
      \pgfplotsset{legend style={at={(0,1)}, anchor=north
          west},legend columns=6,row sep=-4pt,legend to name=TensionModified_phiEst}
%Step0
      \addplot[very thin,color=col0, solid, mark=o, mark size=1pt] table[x index=0, y index=1, col sep=tab]
      {Data/TensionModified_phiEst_eps088_grid/CrackEnergy_Step0.dat};
      \addlegendentry{start mesh}
%Step1
      \addplot[very thin,color=col1, dashed] table[x index=0, y index=1, col sep=tab]
      {Data/TensionModified_phiEst_eps088_grid/CrackEnergy_Step1.dat};
      \addlegendentry{adaptive 1}
%Step2
      \addplot[very thin,color=col2, solid, mark=o, mark size=1pt] table[x index=0, y index=1, col sep=tab]
      {Data/TensionModified_phiEst_eps088_grid/CrackEnergy_Step2.dat};
      \addlegendentry{adaptive 2}
%Step3
      \addplot[very thin,color=col3, solid, mark=star, mark size=2pt] table[x index=0, y index=1, col sep=tab]
      {Data/TensionModified_phiEst_eps088_grid/CrackEnergy_Step3.dat};
      \addlegendentry{adaptive 3}
%Step4
      \addplot[very thin,color=col4, solid, mark=diamond, mark size=1pt] table[x index=0, y index=1, col sep=tab]
      {Data/TensionModified_phiEst_eps088_grid/CrackEnergy_Step4.dat};
      \addlegendentry{adaptive 4}
%Step5
      \addplot[very thin,color=col5, solid] table[x index=0, y index=1, col sep=tab]
      {Data/TensionModified_phiEst_eps088_grid/CrackEnergy_Step5.dat};
      \addlegendentry{adaptive 5}
%Step6
%      \addplot[very thin,color=col6, densely dashed] table[x index=0, y index=1, col sep=tab]
%      {Data/TensionModified_phiEst_eps088_grid/CrackEnergy_Step6.dat};
%      \addlegendentry{adaptive 6}
    \end{axis}
  \end{tikzpicture}}\hfill
  %BulkEnergy
  \subfloat[bulk energy]{\begin{tikzpicture}
    \begin{axis}[clip marker paths=true,
      ylabel near ticks, xlabel near ticks,
      ylabel={bulk energy},
      xtick={0,0.0005,0.001,0.0015,0.002,0.0025,0.003,0.0035},
      xtick scale label code/.code={\pgfmathparse{int(-#1)} time $\cdot 10^{-\pgfmathresult}$},
      every x tick scale label/.style={at={(0.5,-0.1)},anchor=north},
      ytick={0,0.2,0.4,0.6,0.8,1,1.2,1.4},
      width=0.32\textwidth,font=\scriptsize,legend cell
      align=left, xmin=0, xmax=0.0037, ymin=0, ymax=1.4]
      % Step0
      \addplot[very thin,color=col0, solid, mark=o, mark size=1pt] table[x index=0, y index=1, col sep=tab]
      {Data/TensionModified_phiEst_eps088_grid/BulkEnergy_Step0.dat};
      % Step1
      \addplot[very thin,color=col1, dashed] table[x index=0, y index=1, col sep=tab]
      {Data/TensionModified_phiEst_eps088_grid/BulkEnergy_Step1.dat};
%Step2
      \addplot[very thin,color=col2, solid, mark=o, mark size=1pt] table[x index=0, y index=1, col sep=tab]
      {Data/TensionModified_phiEst_eps088_grid/BulkEnergy_Step2.dat};
%Step3
      \addplot[very thin,color=col3, solid, mark=star, mark size=2pt] table[x index=0, y index=1, col sep=tab]
      {Data/TensionModified_phiEst_eps088_grid/BulkEnergy_Step3.dat};
%Step4
      \addplot[very thin,color=col4, solid, mark=diamond, mark size=1pt] table[x index=0, y index=1, col sep=tab]
      {Data/TensionModified_phiEst_eps088_grid/BulkEnergy_Step4.dat};
%Step5
      \addplot[very thin,color=col5, solid] table[x index=0, y index=1, col sep=tab]
      {Data/TensionModified_phiEst_eps088_grid/BulkEnergy_Step5.dat};
%Step6
%      \addplot[very thin,color=col6, densely dashed] table[x index=0, y index=1, col sep=tab]
%      {Data/TensionModified_phiEst_eps088_grid/BulkEnergy_Step6.dat};
    \end{axis}
  \end{tikzpicture}}\hfill
  %LoadDispleacement
  \subfloat[load displacement curves]{\begin{tikzpicture}
    \begin{axis}[clip marker paths=true,
        ylabel near ticks, xlabel near ticks,
        ylabel={load},
      xtick={0,0.001,0.002,0.003,0.004,0.005,0.006,0.007},
      xtick scale label code/.code={\pgfmathparse{int(-#1)} displacement $\cdot 10^{-\pgfmathresult}$},
      every x tick scale label/.style={at={(0.5,-0.1)},anchor=north},
      ytick={0,50,100,150,200,250,300,350,400},
      width=0.32\textwidth,font=\scriptsize,legend cell
      align=left, xmin=0, xmax=0.0074, ymin=0, ymax=400]
      % Step0
      \addplot[very thin,color=col0, solid, mark=o, mark size=1pt] table[x expr=\thisrowno{0}*2, y index=1, col sep=tab]
      {Data/TensionModified_phiEst_eps088_grid/StressY_Step0.dat};
      % Step1
      \addplot[very thin,color=col1, dashed] table[x expr=\thisrowno{0}*2, y index=1, col sep=tab]
      {Data/TensionModified_phiEst_eps088_grid/StressY_Step1.dat};
%Step2
      \addplot[very thin,color=col2, solid, mark=o, mark size=1pt] table[x expr=\thisrowno{0}*2, y index=1, col sep=tab]
      {Data/TensionModified_phiEst_eps088_grid/StressY_Step2.dat};
%Step3
      \addplot[very thin,color=col3, solid, mark=star, mark size=2pt] table[x expr=\thisrowno{0}*2, y index=1, col sep=tab]
      {Data/TensionModified_phiEst_eps088_grid/StressY_Step3.dat};
%Step4
      \addplot[very thin,color=col4, solid, mark=diamond, mark size=1pt] table[x expr=\thisrowno{0}*2, y index=1, col sep=tab]
      {Data/TensionModified_phiEst_eps088_grid/StressY_Step4.dat};
%Step5
      \addplot[very thin,color=col5, solid] table[x expr=\thisrowno{0}*2, y index=1, col sep=tab]
      {Data/TensionModified_phiEst_eps088_grid/StressY_Step5.dat};
%Step6
%      \addplot[very thin,color=col6, densely dashed] table[x expr=\thisrowno{0}*2, y index=1, col sep=tab]
%      {Data/TensionModified_phiEst_eps088_grid/StressY_Step6.dat};
    \end{axis}
  \end{tikzpicture}}\\
\vspace{-3mm}
\hfill\subfloat{\begin{tikzpicture}[font=\scriptsize]
    \ref{TensionModified_phiEst}
    \end{tikzpicture}
  }
  \caption{Tension test with $\epsilon=0.088$.}
\label{fig:Tension_Quantities}
\end{figure}

%% file: Numerik-Fig5.tex
\begin{figure}[H]
    \definecolor{col0}{rgb}{0,0,0}
    \definecolor{col1}{rgb}{0.2,0.3,1}
    \definecolor{col2}{rgb}{0.5,1,1}
    \definecolor{col3}{rgb}{0.2,0.8,0.5}
    \definecolor{col4}{rgb}{0.9,0.8,0.3}
    \definecolor{col5}{rgb}{0.9,0.5,0.1}
    \definecolor{col6}{rgb}{0.9,0.1,0}
    \definecolor{col7}{rgb}{0.4,0.2,0.9}
  \centering
  %Crack Energy
\subfloat[crack energy]{\begin{tikzpicture}
    \begin{axis}[clip marker paths=true,
      ylabel near ticks, xlabel near ticks,
      ylabel={crack energy},
      xtick={0,0.004,0.008,0.012,0.016},
      ytick={0,0.5,1,1.5,2,2.5,3},
      xtick scale label code/.code={\pgfmathparse{int(-#1)} time $\cdot 10^{-\pgfmathresult}$},
      every x tick scale label/.style={at={(0.5,-0.1)},anchor=north},
      width=0.32\textwidth,font=\scriptsize,legend cell
      align=left, xmin=0, xmax=0.018, ymin=0, ymax=3]
      \pgfplotsset{legend style={at={(0,1)}, anchor=north
          west},legend columns=6,row sep=-4pt,legend to name=Shear_phiEst}
%Step0
      \addplot[very thin,color=col0, solid, mark=o, mark size=1pt] table[x index=0, y index=1, col sep=tab]
      {Data/Shear_phiEstBd2_eps088_grid/CrackEnergy_Step0.gpl};
      \addlegendentry{start mesh}
%Step1
      \addplot[very thin,color=col1, dashed] table[x index=0, y index=1, col sep=tab]
      {Data/Shear_phiEstBd2_eps088_grid/CrackEnergy_Step1.gpl};
      \addlegendentry{adaptive 1}
%Step2
      \addplot[very thin,color=col2, solid, mark=o, mark size=1pt] table[x index=0, y index=1, col sep=tab]
      {Data/Shear_phiEstBd2_eps088_grid/CrackEnergy_Step2.gpl};
      \addlegendentry{adaptive 2}
%Step3
      \addplot[very thin,color=col3, solid, mark=star, mark size=2pt] table[x index=0, y index=1, col sep=tab]
      {Data/Shear_phiEstBd2_eps088_grid/CrackEnergy_Step3.gpl};
      \addlegendentry{adaptive 3}
%Step4
      \addplot[very thin,color=col4, solid, mark=diamond, mark size=1pt] table[x index=0, y index=1, col sep=tab]
      {Data/Shear_phiEstBd2_eps088_grid/CrackEnergy_Step4.gpl};
      \addlegendentry{adaptive 4}
%Step5
      \addplot[very thin,color=col5, solid] table[x index=0, y index=1, col sep=tab]
      {Data/Shear_phiEstBd2_eps088_grid/CrackEnergy_Step5.gpl};
      \addlegendentry{adaptive 5}
    \end{axis}
  \end{tikzpicture}}\hfill
  %BulkEnergy
  \subfloat[bulk energy]{\begin{tikzpicture}
    \begin{axis}[clip marker paths=true,
        ylabel near ticks, xlabel near ticks,
      ylabel={bulk energy},
      xtick={0,0.004,0.008,0.012,0.016},
      xtick scale label code/.code={\pgfmathparse{int(-#1)} time $\cdot 10^{-\pgfmathresult}$},
      every x tick scale label/.style={at={(0.5,-0.1)},anchor=north},
      ytick={0,0.5,1,1.5,2,2.5,3},
      width=0.32\textwidth,font=\scriptsize,legend cell
      align=left, xmin=0, xmax=0.018, ymin=0, ymax=3]
      % Step0
      \addplot[very thin,thin,color=col0, solid, mark=o, mark size=1pt] table[x index=0, y index=1, col sep=tab]
      {Data/Shear_phiEstBd2_eps088_grid/BulkEnergy_Step0.gpl};
      % Step1
      \addplot[very thin,thin,color=col1, dashed] table[x index=0, y index=1, col sep=tab]
      {Data/Shear_phiEstBd2_eps088_grid/BulkEnergy_Step1.gpl};
%Step2
      \addplot[very thin,thin,color=col2, solid, mark=o, mark size=1pt] table[x index=0, y index=1, col sep=tab]
      {Data/Shear_phiEstBd2_eps088_grid/BulkEnergy_Step2.gpl};
%Step3
      \addplot[very thin,thin,color=col3, solid, mark=star, mark size=2pt] table[x index=0, y index=1, col sep=tab]
      {Data/Shear_phiEstBd2_eps088_grid/BulkEnergy_Step3.gpl};
%Step4
      \addplot[very thin,thin,color=col4, solid, mark=diamond, mark size=1pt] table[x index=0, y index=1, col sep=tab]
      {Data/Shear_phiEstBd2_eps088_grid/BulkEnergy_Step4.gpl};
%Step5
      \addplot[very thin,thin,color=col5, solid] table[x index=0, y index=1, col sep=tab]
      {Data/Shear_phiEstBd2_eps088_grid/BulkEnergy_Step5.gpl};
    \end{axis}
  \end{tikzpicture}}\hfill
  %LoadDispleacement
  \subfloat[load displacement curves]{\begin{tikzpicture}
      \begin{axis}[clip marker paths=true,
        ylabel near ticks, xlabel near ticks,
        ylabel={load},
        xtick={0,0.004,0.008,0.012,0.016},
        xtick scale label code/.code={\pgfmathparse{int(-#1)} displacement $\cdot 10^{-\pgfmathresult}$},
        every x tick scale label/.style={at={(0.5,-0.1)},anchor=north},
        ytick={0,100,200,300,400,500,600,700},
        width=0.32\textwidth,font=\scriptsize,legend cell
        align=left, xmin=0, xmax=0.018, ymin=0, ymax=700]
        % Step0
        \addplot[very thin,color=col0, solid, mark=o, mark size=1pt] table[x index=0, y index=1, col sep=tab]
      {Data/Shear_phiEstBd2_eps088_grid/StressX_Step0.gpl};
      % Step1
      \addplot[very thin,color=col1, dashed] table[x index=0, y index=1, col sep=tab]
      {Data/Shear_phiEstBd2_eps088_grid/StressX_Step1.gpl};
%Step2
      \addplot[very thin,color=col2, solid, mark=o, mark size=1pt] table[x index=0, y index=1, col sep=tab]
      {Data/Shear_phiEstBd2_eps088_grid/StressX_Step2.gpl};
%Step3
      \addplot[very thin,color=col3, solid, mark=star, mark size=2pt] table[x index=0, y index=1, col sep=tab]
      {Data/Shear_phiEstBd2_eps088_grid/StressX_Step3.gpl};
%Step4
      \addplot[very thin,color=col4, solid, mark=diamond, mark size=1pt] table[x index=0, y index=1, col sep=tab]
      {Data/Shear_phiEstBd2_eps088_grid/StressX_Step4.gpl};
%Step5
      \addplot[very thin,color=col5, solid] table[x index=0, y index=1, col sep=tab]
      {Data/Shear_phiEstBd2_eps088_grid/StressX_Step5.gpl};
    \end{axis}
  \end{tikzpicture}}\\
\vspace{-3mm}
\hfill\subfloat{\begin{tikzpicture}[font=\scriptsize]
    \ref{Shear_phiEst}
    \end{tikzpicture}
  }
  \caption{Shear test with $\epsilon=0.088$.}
\label{fig:Shear_Quantities}
\end{figure}

%% file: Numerik-Fig6.tex
\begin{figure}[H]
    \definecolor{col0}{rgb}{0,0,0}
    \definecolor{col1}{rgb}{0.2,0.3,1}
    \definecolor{col2}{rgb}{0.5,1,1}
    \definecolor{col3}{rgb}{0.2,0.8,0.5}
    \definecolor{col4}{rgb}{0.9,0.8,0.3}
    \definecolor{col5}{rgb}{0.9,0.5,0.1}
    \definecolor{col6}{rgb}{0.9,0.1,0}
    \definecolor{col7}{rgb}{0.4,0.2,0.9}
  \centering
  %Crack Energy
  \subfloat[crack energy]{\begin{tikzpicture}
      \begin{axis}[clip marker paths=true,
      ylabel near ticks, xlabel near ticks,
      xlabel={time},ylabel={crack energy},
      xtick={0,0.05,0.1,0.15,0.2,0.25,0.3},
      ytick={0,2,4,6,8,10,12,14},
      xticklabel style={/pgf/number format/fixed},
      width=0.32\textwidth,font=\scriptsize,legend cell
      align=left, xmin=0, xmax=0.3, ymin=0, ymax=14]
      \pgfplotsset{legend style={at={(0,1)}, anchor=north
          west},legend columns=6,row sep=-4pt,legend to name=LShape_phiEst}
%Step0
      \addplot[very thin,color=col0, solid, mark=o, mark size=1pt] table[x index=0, y index=1, col sep=tab]
      {Data/LShape_phiEst_eps20_gridNEU/CrackEnergy_Step0.dat};
      \addlegendentry{start mesh}
%Step1
      \addplot[very thin,color=col1, dashed] table[x index=0, y index=1, col sep=tab]
      {Data/LShape_phiEst_eps20_gridNEU/CrackEnergy_Step1.dat};
      \addlegendentry{adaptive 1}
%Step2
      \addplot[very thin,color=col2, solid, mark=o, mark size=1pt] table[x index=0, y index=1, col sep=tab]
      {Data/LShape_phiEst_eps20_gridNEU/CrackEnergy_Step2.dat};
      \addlegendentry{adaptive 2}
%Step3
      \addplot[very thin,color=col3, solid, mark=star, mark size=2pt] table[x index=0, y index=1, col sep=tab]
      {Data/LShape_phiEst_eps20_gridNEU/CrackEnergy_Step3.dat};
      \addlegendentry{adaptive 3}
%Step4
      \addplot[very thin,color=col4, solid, mark=diamond, mark size=1pt] table[x index=0, y index=1, col sep=tab]
      {Data/LShape_phiEst_eps20_gridNEU/CrackEnergy_Step4.dat};
      \addlegendentry{adaptive 4}
%Step5
      \addplot[very thin,color=col5, solid] table[x index=0, y index=1, col sep=tab]
      {Data/LShape_phiEst_eps20_gridNEU/CrackEnergy_Step5.dat};
      \addlegendentry{adaptive 5}
    \end{axis}
  \end{tikzpicture}}\hfill
  %BulkEnergy
  \subfloat[bulk energy]{\begin{tikzpicture}
    \begin{axis}[clip marker paths=true,
        ylabel near ticks, xlabel near ticks,
      xlabel={time},ylabel={bulk energy},
      xtick={0,0.05,0.1,0.15,0.2,0.25,0.3},
      ytick={0,5,10,15,20,25},
      xticklabel style={/pgf/number format/fixed},
      width=0.32\textwidth,font=\scriptsize,legend cell
      align=left, xmin=0, xmax=0.3, ymin=0, ymax=25]
      % Step0
      \addplot[very thin,color=col0, solid, mark=o, mark size=1pt] table[x index=0, y index=1, col sep=tab]
      {Data/LShape_phiEst_eps20_gridNEU/BulkEnergy_Step0.dat};
      % Step1
      \addplot[very thin,color=col1, dashed] table[x index=0, y index=1, col sep=tab]
      {Data/LShape_phiEst_eps20_gridNEU/BulkEnergy_Step1.dat};
%Step2
      \addplot[very thin,color=col2, solid, mark=o, mark size=1pt] table[x index=0, y index=1, col sep=tab]
      {Data/LShape_phiEst_eps20_gridNEU/BulkEnergy_Step2.dat};
%Step3
      \addplot[very thin,color=col3, solid, mark=star, mark size=2pt] table[x index=0, y index=1, col sep=tab]
      {Data/LShape_phiEst_eps20_gridNEU/BulkEnergy_Step3.dat};
%Step4
      \addplot[very thin,color=col4, solid, mark=diamond, mark size=1pt] table[x index=0, y index=1, col sep=tab]
      {Data/LShape_phiEst_eps20_gridNEU/BulkEnergy_Step4.dat};
%Step5
      \addplot[very thin,color=col5, solid] table[x index=0, y index=1, col sep=tab]
      {Data/LShape_phiEst_eps20_gridNEU/BulkEnergy_Step5.dat};
    \end{axis}
  \end{tikzpicture}}\hfill
  %LoadDispleacement
  \subfloat[load displacement curves]{\begin{tikzpicture}
    \begin{axis}[clip marker paths=true,
        ylabel near ticks, xlabel near ticks,
      xlabel={displacement},ylabel={load},
      xtick={0,0.05,0.1,0.15,0.2,0.25,0.3},
      ytick={0,50,100,150,200,250},
      xticklabel style={/pgf/number format/fixed},
      width=0.32\textwidth,font=\scriptsize,legend cell
      align=left, xmin=0, xmax=0.3, ymin=0, ymax=250]
      % Step0
      \addplot[very thin,color=col0, solid, mark=o, mark size=1pt] table[x index=0, y index=1, col sep=tab]
      {Data/LShape_phiEst_eps20_gridNEU/StressY_Step0.dat};
      % Step1
      \addplot[very thin,color=col1, dashed] table[x index=0, y index=1, col sep=tab]
      {Data/LShape_phiEst_eps20_gridNEU/StressY_Step1.dat};
%Step2
      \addplot[very thin,color=col2, solid, mark=o, mark size=1pt] table[x index=0, y index=1, col sep=tab]
      {Data/LShape_phiEst_eps20_gridNEU/StressY_Step2.dat};
%Step3
      \addplot[very thin,color=col3, solid, mark=star, mark size=2pt] table[x index=0, y index=1, col sep=tab]
      {Data/LShape_phiEst_eps20_gridNEU/StressY_Step3.dat};
%Step4
      \addplot[very thin,color=col4, solid, mark=diamond, mark size=1pt] table[x index=0, y index=1, col sep=tab]
      {Data/LShape_phiEst_eps20_gridNEU/StressY_Step4.dat};
%Step5
      \addplot[very thin,color=col5, solid] table[x index=0, y index=1, col sep=tab]
      {Data/LShape_phiEst_eps20_gridNEU/StressY_Step5.dat};
    \end{axis}
  \end{tikzpicture}}\\
\vspace{-3mm}
\hfill\subfloat{\begin{tikzpicture}[font=\scriptsize]
    \ref{LShape_phiEst}
    \end{tikzpicture}
  }
  \caption{L-shape test with $\epsilon=20$.}
\label{fig:Lshape_Quantities}
\end{figure}

%% file: Numerik-Fig7.tex
\begin{figure}[H]
\centering
\subfloat{
  \begin{tikzpicture}
    \begin{loglogaxis}[width=0.35\textwidth,
        ylabel near ticks, xlabel near ticks,
        ylabel={error},
        xlabel={number of nodes},
        legend cell align=left,font=\scriptsize,row sep=-4pt,
      ]
      \addplot[color=red,mark=*,mark size=1pt]
              table[x expr=\thisrow{DOF}/2, y expr=sqrt(\thisrow{ERROR}), col sep=tab]
              {Data/TensionModified_Conv_eps088_phiEst/ConvergenceT280/Data_phi.dat};
      \addlegendentry{adaptive}
      \addplot[color=blue,mark=*,mark size=1pt]
              table[x expr=\thisrow{DOF}/2, y expr=sqrt(\thisrow{ERROR}), col sep=tab]
              {Data/TensionModified_Conv_eps088_phiEst/ConvergenceT280/Data_phi_uni.dat};
      \addlegendentry{uniform}        
    \end{loglogaxis}
  \end{tikzpicture}
}\hfill
\subfloat{
  \begin{tikzpicture}
    \begin{loglogaxis}[width=0.35\textwidth,
        ylabel near ticks, xlabel near ticks,
        ylabel={error},
        xlabel={number of nodes},
        ytick={0.04,0.06,0.1,0.14,0.18,0.22},
        yticklabel={
        \pgfkeys{/pgf/fpu=true}
        \pgfmathparse{exp(\tick)}
        \pgfmathprintnumber[fixed relative, precision=3]{\pgfmathresult}
        \pgfkeys{/pgf/fpu=false}
        },
        legend cell align=left,font=\scriptsize,row sep=-4pt,
        ymin=0.03, ymax=0.25
      ]
      \addplot[color=red,mark=*,mark size=1pt]
              table[x expr=\thisrow{DOF}/2, y expr=sqrt(\thisrow{ERROR}), col sep=tab]
              {Data/TensionModified_Conv_eps088_phiEst/ConvergenceT280/Data_u.dat};
      \addlegendentry{adaptive}
      \addplot[color=blue,mark=*,mark size=1pt]
              table[x expr=\thisrow{DOF}/2, y expr=sqrt(\thisrow{ERROR}), col sep=tab]
              {Data/TensionModified_Conv_eps088_phiEst/ConvergenceT280/Data_u_uni.dat};
      \addlegendentry{uniform}        
    \end{loglogaxis}
  \end{tikzpicture} 
}
\caption{Convergence in different error measures for tension test at
  $n=280$ with $\epsilon=0.088$ and adaptive refinement based on the
  estimator $\eta^{\varphi}$. Left: energy norm in $\varphi$, Right:
  energy norm in $u$.}
\label{fig:Convergence_tension}
\end{figure}

%% file: Numerik-Fig8.tex
\begin{figure}[H]
\centering
\subfloat{
    \begin{tikzpicture}
    \begin{loglogaxis}[width=0.35\textwidth,
        ylabel near ticks, xlabel near ticks,
        ylabel={error},
        xlabel={number of nodes},
        ytick={0.1,1,10},
        ymax=10,
        legend cell align=left,font=\scriptsize,row sep=-4pt,
      ]
      \addplot[color=red,mark=*,mark size=1pt]
              table[x expr=\thisrow{DOF}/2, y expr=sqrt(\thisrow{ERROR}), col sep=tab]
              {Data/Shear_Conv_eps088_phiEst/Data_phi.dat};
      \addlegendentry{adaptive}
      \addplot[color=blue,mark=*,mark size=1pt]
              table[x expr=\thisrow{DOF}/2, y expr=sqrt(\thisrow{ERROR}), col sep=tab]
              {Data/Shear_Conv_eps088_phiEst/Data_phi_uni.dat};
      \addlegendentry{uniform}        
    \end{loglogaxis}
  \end{tikzpicture}
}\hfill
\subfloat{
   \begin{tikzpicture}
    \begin{loglogaxis}[width=0.35\textwidth,
        ylabel near ticks, xlabel near ticks,
        ylabel={error},
        xlabel={number of nodes},
        ytick={0.1,1},
        ymax=1,
        legend cell align=left,font=\scriptsize,row sep=-4pt
      ]
      \addplot[color=red,mark=*,mark size=1pt]
              table[x expr=\thisrow{DOF}/2, y expr=sqrt(\thisrow{ERROR}), col sep=tab]
              {Data/Shear_Conv_eps088_phiEst/Data_u.dat};
      \addlegendentry{adaptive}
      \addplot[color=blue,mark=*,mark size=1pt]
              table[x expr=\thisrow{DOF}/2, y expr=sqrt(\thisrow{ERROR}), col sep=tab]
              {Data/Shear_Conv_eps088_phiEst/Data_u_uni.dat};
      \addlegendentry{uniform}        
    \end{loglogaxis}
  \end{tikzpicture} 
}
\caption{Convergence in different error measures for shear test at
  $n=107$ with $\epsilon=0.088$ and adaptive refinement based on the
  estimator $\eta^{\varphi}$. Left: energy norm in $\varphi$, Right:
  energy norm in $u$.}
\label{fig:Convergence_shear}
\end{figure}

%% file: Numerik-Fig9.tex
\begin{figure}[H]
\centering
\subfloat{
    \begin{tikzpicture}
    \begin{loglogaxis}[width=0.35\textwidth,
        ylabel near ticks, xlabel near ticks,
        ylabel={error},
        ytick={0.01,0.05,0.1,0.2},
        ymin=0.01,ymax=0.5,
        yticklabel={
        \pgfkeys{/pgf/fpu=true}
        \pgfmathparse{exp(\tick)}
        \pgfmathprintnumber[fixed relative, precision=3]{\pgfmathresult}
        \pgfkeys{/pgf/fpu=false}
        },
        xlabel={number of nodes},
        legend cell align=left,font=\scriptsize,row sep=-4pt,
      ]
      \addplot[color=red,mark=*,mark size=1pt]
              table[x expr=\thisrow{DOF}/2, y expr=sqrt(\thisrow{ERROR}), col sep=tab]
              {Data/LShape_Conv_eps20_phiEst_NEU/Data_phi.dat};
      \addlegendentry{adaptive}
      \addplot[color=blue,mark=*,mark size=1pt]
              table[x expr=\thisrow{DOF}/2, y expr=sqrt(\thisrow{ERROR}), col sep=tab]
              {Data/LShape_Conv_eps20_phiEst_NEU/Data_phi_uni.dat};
      \addlegendentry{uniform}        
    \end{loglogaxis}
  \end{tikzpicture}
}\hfill
\subfloat{
  \begin{tikzpicture}
    \begin{loglogaxis}[width=0.35\textwidth,
        ylabel near ticks, xlabel near ticks,
        ylabel={error},
        xlabel={number of nodes},
        ymin=0.25,
        ytick={0.25,0.5,0.75,1},
        yticklabel={
        \pgfkeys{/pgf/fpu=true}
        \pgfmathparse{exp(\tick)}
        \pgfmathprintnumber[fixed relative, precision=3]{\pgfmathresult}
        \pgfkeys{/pgf/fpu=false}
        },
        legend cell align=left,font=\scriptsize,row sep=-4pt,
      ]
      \addplot[color=red,mark=*,mark size=1pt]
              table[x expr=\thisrow{DOF}/2, y expr=sqrt(\thisrow{ERROR}), col sep=tab]
              {Data/LShape_Conv_eps20_phiEst_NEU/Data_u.dat};
      \addlegendentry{adaptive}
      \addplot[color=blue,mark=*,mark size=1pt]
              table[x expr=\thisrow{DOF}/2, y expr=sqrt(\thisrow{ERROR}), col sep=tab]
              {Data/LShape_Conv_eps20_phiEst_NEU/Data_u_uni.dat};
      \addlegendentry{uniform}        
    \end{loglogaxis}
  \end{tikzpicture} 
}
\caption{Convergence in different error measures for L-shape test at $n=200$ with $\epsilon=20$ and adaptive refinement based on the estimator  $\eta^{\varphi}$. Left: energy norm in $\varphi$, Right:
  energy norm in $u$.}
\label{fig:Convergence_Lshape}
\end{figure}

%% file: Numerik-Fig10.tex
\begin{figure}[H]
    \definecolor{col0}{rgb}{0,0,0}
    \definecolor{col1}{rgb}{0.2,0.3,1}
    \definecolor{col2}{rgb}{0.5,1,1}
    \definecolor{col3}{rgb}{0.2,0.8,0.5}
    \definecolor{col4}{rgb}{0.9,0.8,0.3}
    \definecolor{col5}{rgb}{0.9,0.5,0.1}
    \definecolor{col6}{rgb}{0.9,0.1,0}
    \definecolor{col7}{rgb}{0.4,0.2,0.9}
\centering
\subfloat[estimator in $\varphi$]{
  \begin{tikzpicture}
    \begin{axis}[clip marker paths=true,
      ylabel near ticks, xlabel near ticks,
      ylabel={efficiency index},
      xtick={0,5000,10000,15000,20000,25000,30000,35000},
      xtick scale label code/.code={\pgfmathparse{int(#1)} degrees of freedom $\cdot 10^{\pgfmathresult}$},
      every x tick scale label/.style={at={(0.5,-0.1)},anchor=north},
      ytick={0,2,4,6,8,10},
      width=0.38\textwidth,font=\scriptsize,legend cell
      align=left, xmin=0, xmax=35000, ymin=0, ymax=10]
      \pgfplotsset{legend style={at={(0,1)}, anchor=north
          west}, legend columns=1, row sep=-4pt,legend to name=Efficiency_Tension_phi}
      %Step0
      \addplot[thin,color=col0, solid, mark=o, mark size=1pt] table[x index=0, y index=1, col sep=tab]
              {Data/TensionModified_Efficiency/Data_eps0055.dat};
      \addlegendentry{$\epsilon=0.0055$}
      %Step1
       \addplot[thin,color=col1, dashed] table[x index=0, y index=1, col sep=tab]
              {Data/TensionModified_Efficiency/Data_eps011.dat};
      \addlegendentry{$\epsilon=0.011$}
      %Step2
       \addplot[thin,color=col2, solid, mark=o, mark size=1pt] table[x index=0, y index=1, col sep=tab]
              {Data/TensionModified_Efficiency/Data_eps022.dat};
      \addlegendentry{$\epsilon=0.022$}
      %Step3
       \addplot[thin,color=col3, solid, mark=star, mark size=1pt] table[x index=0, y index=1, col sep=tab]
              {Data/TensionModified_Efficiency/Data_eps044.dat};
      \addlegendentry{$\epsilon=0.044$}
      %Step4
       \addplot[thin,color=col4, solid, mark=diamond, mark size=1pt] table[x index=0, y index=1, col sep=tab]
              {Data/TensionModified_Efficiency/Data_eps088.dat};
      \addlegendentry{$\epsilon=0.088$}
      %Step5
       \addplot[thin,color=col5, solid] table[x index=0, y index=1, col sep=tab]
              {Data/TensionModified_Efficiency/Data_eps176.dat};
              \addlegendentry{$\epsilon=0.176$}
      %Step6
       \addplot[thin,color=col6, densely dashed] table[x index=0, y index=1, col sep=tab]
              {Data/TensionModified_Efficiency/Data_eps352.dat};
      \addlegendentry{$\epsilon=0.352$}
 
    \end{axis}
  \end{tikzpicture}
}\hfill
\subfloat[std. estimator]{
  \begin{tikzpicture}
    \begin{axis}[clip marker paths=true,
      ylabel near ticks, xlabel near ticks,
      ylabel={efficiency index},
      xtick={0,5000,10000,15000,20000,25000,30000,35000},
      xtick scale label code/.code={\pgfmathparse{int(#1)} degrees of freedom $\cdot 10^{\pgfmathresult}$},
      every x tick scale label/.style={at={(0.5,-0.1)},anchor=north},
      ytick={0,2,4,6,8,10},
      width=0.38\textwidth,font=\scriptsize,legend cell
      align=left, xmin=0, xmax=35000, ymin=0, ymax=10]
      %Step0
      \addplot[thin,color=col0, solid, mark=o, mark size=1pt] table[x index=0, y index=1, col sep=tab]
              {Data/TensionModified_Efficiency/Data_eps0055_std.dat};
      %Step1
       \addplot[thin,color=col1, dashed] table[x index=0, y index=1, col sep=tab]
              {Data/TensionModified_Efficiency/Data_eps011_std.dat};
      %Step2
       \addplot[thin,color=col2, solid, mark=o, mark size=1pt] table[x index=0, y index=1, col sep=tab]
              {Data/TensionModified_Efficiency/Data_eps022_std.dat};
      %Step3
       \addplot[thin,color=col3, solid, mark=star, mark size=1pt] table[x index=0, y index=1, col sep=tab]
              {Data/TensionModified_Efficiency/Data_eps044_std.dat};
      %Step4
       \addplot[thin,color=col4, solid, mark=diamond, mark size=1pt] table[x index=0, y index=1, col sep=tab]
              {Data/TensionModified_Efficiency/Data_eps088_std.dat};
      %Step5
       \addplot[thin,color=col5, solid] table[x index=0, y index=1, col sep=tab]
              {Data/TensionModified_Efficiency/Data_eps176_std.dat};
      %Step6
       \addplot[thin,color=col6, densely dashed] table[x index=0, y index=1, col sep=tab]
              {Data/TensionModified_Efficiency/Data_eps352_std.dat};
    \end{axis}
  \end{tikzpicture}  
}\hfill
\begin{tikzpicture}[font=\scriptsize]\ref{Efficiency_Tension_phi}
\end{tikzpicture}
\caption{Efficiency index for tension test in time step $n=280$ with estimator $\eta^{\varphi}$ and standard estimator on different meshes}
\label{fig:Efficiency_Tension}
\end{figure}

%% file: Numerik-Fig11.tex
\begin{figure}[H]
    \definecolor{col0}{rgb}{0,0,0}
    \definecolor{col1}{rgb}{0.2,0.3,1}
    \definecolor{col2}{rgb}{0.5,1,1}
    \definecolor{col3}{rgb}{0.2,0.8,0.5}
    \definecolor{col4}{rgb}{0.9,0.8,0.3}
    \definecolor{col5}{rgb}{0.9,0.5,0.1}
    \definecolor{col6}{rgb}{0.9,0.1,0}
    \definecolor{col7}{rgb}{0.4,0.2,0.9}
\centering
\subfloat[estimator in $\varphi$]{
  \begin{tikzpicture}
    \begin{axis}[clip marker paths=true,
      ylabel near ticks, xlabel near ticks,
      ylabel={efficiency index},
      xlabel={degrees of freedom},
      xtick={0,1000,2000,3000,4000,5000},
      ytick={0,0.5,1,1.5,2,2.5,3,3.5,4,4.5},
      width=0.38\textwidth,font=\scriptsize,legend cell
      align=left, xmin=0, xmax=5000, ymin=0, ymax=4.5]
      \pgfplotsset{legend style={at={(0,1)}, anchor=north
          west}, legend columns=1, row sep=-4pt,legend to name=Efficiency_shear_phi}
      %Step0
      \addplot[thin,color=col0, solid, mark=o, mark size=1pt] table[x index=0, y index=1, col sep=tab]
              {Data/Shear_EfficiencyVar/Data_eps0055.dat};
      \addlegendentry{$\epsilon=0.0055$}
      %Step1
       \addplot[thin,color=col1, dashed] table[x index=0, y index=1, col sep=tab]
              {Data/Shear_EfficiencyVar/Data_eps011.dat};
      \addlegendentry{$\epsilon=0.011$}
      %Step2
       \addplot[thin,color=col2, solid, mark=o, mark size=1pt] table[x index=0, y index=1, col sep=tab]
              {Data/Shear_EfficiencyVar/Data_eps022.dat};
      \addlegendentry{$\epsilon=0.022$}
      %Step3
       \addplot[thin,color=col3, solid, mark=star, mark size=1pt] table[x index=0, y index=1, col sep=tab]
              {Data/Shear_EfficiencyVar/Data_eps044.dat};
      \addlegendentry{$\epsilon=0.044$}
      %Step4
       \addplot[thin,color=col4, solid, mark=diamond, mark size=1pt] table[x index=0, y index=1, col sep=tab]
              {Data/Shear_EfficiencyVar/Data_eps088.dat};
      \addlegendentry{$\epsilon=0.088$}
      %Step5
       \addplot[thin,color=col5, solid] table[x index=0, y index=1, col sep=tab]
              {Data/Shear_EfficiencyVar/Data_eps176.dat};
              \addlegendentry{$\epsilon=0.176$}
      %Step6
       \addplot[thin,color=col6, densely dashed] table[x index=0, y index=1, col sep=tab]
              {Data/Shear_EfficiencyVar/Data_eps352.dat};
      \addlegendentry{$\epsilon=0.352$}
 
    \end{axis}
  \end{tikzpicture}
}\hfill
\subfloat[std. estimator]{
  \begin{tikzpicture}
    \begin{axis}[clip marker paths=true,
      ylabel near ticks, xlabel near ticks,
      ylabel={efficiency index},
      xlabel={degrees of freedom},
      xtick={0,500,1000,1500,2000,2500},
      ytick={0,2,4,6,8,10},
      width=0.38\textwidth,font=\scriptsize,legend cell
      align=left, xmin=0, xmax=2500, ymin=0, ymax=6]
      %Step0
      \addplot[thin,color=col0, solid, mark=o, mark size=1pt] table[x index=0, y index=1, col sep=tab]
              {Data/Shear_EfficiencyVar/Data_eps0055_std.dat};
      %Step1
       \addplot[thin,color=col1, dashed] table[x index=0, y index=1, col sep=tab]
              {Data/Shear_EfficiencyVar/Data_eps011_std.dat};
      %Step2
       \addplot[thin,color=col2, solid, mark=o, mark size=1pt] table[x index=0, y index=1, col sep=tab]
              {Data/Shear_EfficiencyVar/Data_eps022_std.dat};
      %Step3
       \addplot[thin,color=col3, solid, mark=star, mark size=1pt] table[x index=0, y index=1, col sep=tab]
              {Data/Shear_EfficiencyVar/Data_eps044_std.dat};
      %Step4
       \addplot[thin,color=col4, solid, mark=diamond, mark size=1pt] table[x index=0, y index=1, col sep=tab]
              {Data/Shear_EfficiencyVar/Data_eps088_std.dat};
      %Step5
       \addplot[thin,color=col5, solid] table[x index=0, y index=1, col sep=tab]
              {Data/Shear_EfficiencyVar/Data_eps176_std.dat};
      %Step6
       \addplot[thin,color=col6, densely dashed] table[x index=0, y index=1, col sep=tab]
              {Data/Shear_EfficiencyVar/Data_eps352_std.dat};
    \end{axis}
  \end{tikzpicture}  
}\hfill
\begin{tikzpicture}[font=\scriptsize]\ref{Efficiency_shear_phi}
\end{tikzpicture}
\caption{Efficiency index for shear test in time step $n= 107$ with estimator $\eta^{\varphi}$ and standard estimator on different meshes}
\label{fig:Efficiency_shear}
\end{figure}

%% file: Numerik-Fig14.tex
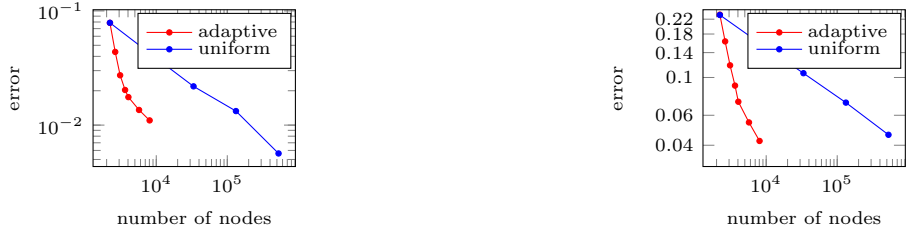
\begin{figure}[H]
\centering
\subfloat{
    \begin{tikzpicture}
    \begin{loglogaxis}[width=0.35\textwidth,
        ylabel near ticks, xlabel near ticks,
        ylabel={error},
        xlabel={number of nodes},
        legend cell align=left,font=\scriptsize,row sep=-4pt,
      ]
      \addplot[color=red,mark=*,mark size=1pt]
              table[x expr=\thisrow{DOF}/2, y expr=sqrt(\thisrow{ERROR}), col sep=tab]
              {Data/TensionModified_Conv_eps088_uphiEst/ConvergenceT280/Data_phi.dat};
      \addlegendentry{adaptive}
      \addplot[color=blue,mark=*,mark size=1pt]
              table[x expr=\thisrow{DOF}/2, y expr=sqrt(\thisrow{ERROR}), col sep=tab]
              {Data/TensionModified_Conv_eps088_uphiEst/ConvergenceT280/Data_phi_uni.dat};
      \addlegendentry{uniform}        
    \end{loglogaxis}
  \end{tikzpicture}
}\hfill
\subfloat{
    \begin{tikzpicture}
    \begin{loglogaxis}[width=0.35\textwidth,
        ylabel near ticks, xlabel near ticks,
        ylabel={error},
        xlabel={number of nodes},
        ytick={0.04,0.06,0.1,0.14,0.18,0.22},
        yticklabel={
        \pgfkeys{/pgf/fpu=true}
        \pgfmathparse{exp(\tick)}
        \pgfmathprintnumber[fixed relative, precision=3]{\pgfmathresult}
        \pgfkeys{/pgf/fpu=false}
        },
        legend cell align=left,font=\scriptsize,row sep=-4pt,
        ymin=0.03, ymax=0.25
      ]
      \addplot[color=red,mark=*,mark size=1pt]
              table[x expr=\thisrow{DOF}/2, y expr=sqrt(\thisrow{ERROR}), col sep=tab]
              {Data/TensionModified_Conv_eps088_uphiEst/ConvergenceT280/Data_u.dat};
      \addlegendentry{adaptive}
      \addplot[color=blue,mark=*,mark size=1pt]
              table[x expr=\thisrow{DOF}/2, y expr=sqrt(\thisrow{ERROR}), col sep=tab]
              {Data/TensionModified_Conv_eps088_uphiEst/ConvergenceT280/Data_u_uni.dat};
      \addlegendentry{uniform}        
    \end{loglogaxis}
  \end{tikzpicture} 
}
\caption{Convergence in different error measures for tension test at $n=280$ with $\epsilon=0.088$ and adaptive refinement based on the estimators $\eta^{\varphi}$ and $\eta^u$. Left: energy norm in $\varphi$, Right:
  energy norm in $u$.}
\label{fig:Convergence_tension_2}
\end{figure}

%% file: Numerik-Fig15.tex
\begin{figure}[H]
\centering
\subfloat{
      \begin{tikzpicture}
    \begin{loglogaxis}[width=0.35\textwidth,
        ylabel near ticks, xlabel near ticks,
        ylabel={error},
        xlabel={number of nodes},
        ytick={0.1,1,10},
        ymax=10,
        legend cell align=left,font=\scriptsize,row sep=-4pt,
      ]
      \addplot[color=red,mark=*,mark size=1pt]
              table[x expr=\thisrow{DOF}/2, y expr=sqrt(\thisrow{ERROR}), col sep=tab]
              {Data/Shear_Conv_eps088_uphiEst/Data_phi.dat};
      \addlegendentry{adaptive}
      \addplot[color=blue,mark=*,mark size=1pt]
              table[x expr=\thisrow{DOF}/2, y expr=sqrt(\thisrow{ERROR}), col sep=tab]
              {Data/Shear_Conv_eps088_uphiEst/Data_phi_uni.dat};
      \addlegendentry{uniform}        
    \end{loglogaxis}
  \end{tikzpicture}
}\hfill
\subfloat{
   \begin{tikzpicture}
    \begin{loglogaxis}[width=0.35\textwidth,
        ylabel near ticks, xlabel near ticks,
        ylabel={error},
        xlabel={number of nodes},
        ytick={0.1,1},
        ymax=1,
        legend cell align=left,font=\scriptsize,row sep=-4pt
      ]
      \addplot[color=red,mark=*,mark size=1pt]
              table[x expr=\thisrow{DOF}/2, y expr=sqrt(\thisrow{ERROR}), col sep=tab]
              {Data/Shear_Conv_eps088_uphiEst/Data_u.dat};
      \addlegendentry{adaptive}
      \addplot[color=blue,mark=*,mark size=1pt]
              table[x expr=\thisrow{DOF}/2, y expr=sqrt(\thisrow{ERROR}), col sep=tab]
              {Data/Shear_Conv_eps088_uphiEst/Data_u_uni.dat};
      \addlegendentry{uniform}        
    \end{loglogaxis}
  \end{tikzpicture}  
}
\caption{Convergence in different error measures for shear test at $n=107$ with $\epsilon=0.088$ and adaptive refinement based on the estimators $\eta^{\varphi}$ and $\eta^u$. Left: energy norm in $\varphi$, Right:
  energy norm in $u$.}
\label{fig:Convergence_shear_2}
\end{figure}
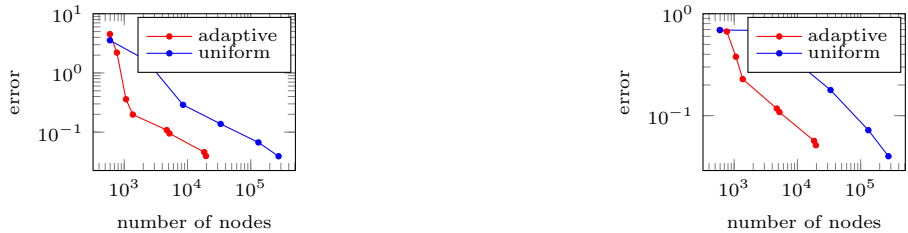

%% file: Main.bbl
\begin{thebibliography}{10}

\bibitem{Ambati_Gerasimov_Lorenzis_2015}
M.~Ambati, T.~Gerasimov, and L.~De~Lorenzis.
\newblock Phase-field modeling of ductile fracture.
\newblock {\em Comput. Mech.}, pages 1--24, 2015.

\bibitem{dealii2019design}
D.~Arndt, W.~Bangerth, D.~Davydov, T.~Heister, L.~Heltai, M.~Kronbichler,
  M.~Maier, J.-P. Pelteret, B.~Turcksin, and D.~Wells.
\newblock The {deal.II} finite element library: Design, features, and insights.
\newblock {\em Comput. Math. Appl.}, 81:407--422, 2021.

\bibitem{Artina_Fornasier_Micheletti_Perotto_2015}
M.~Artina, M.~Fornasier, S.~Micheletti, and S.~Perotto.
\newblock Anisotropic mesh adaptation for crack detection in brittle materials.
\newblock {\em SIAM J. Sci. Comput.}, 37(4):B633--B659, 2015.

\bibitem{Bartels_Carstensen_2004}
S.~Bartels and C.~Carstensen.
\newblock Averaging techniques yield reliable a posteriori finite element error
  control for obstacle problems.
\newblock {\em Numer. Math.}, 99(2):225--249, 2004.

\bibitem{Borden_Hughes_Landis_Verhoosel_2014}
M.~J. Borden, T.~J. Hughes, C.~M. Landis, and C.~V. Verhoosel.
\newblock A higher-order phase-field model for brittle fracture: Formulation
  and analysis within the isogeometric analysis framework.
\newblock {\em Comput. Methods Appl. Mech. Engrg.}, 273:100--118, 2014.

\bibitem{Borden_Verhoosel_Scott_Hughes_Landis_2012}
M.~J. Borden, C.~V. Verhoosel, M.~A. Scott, T.~J.~R. Hughes, and C.~M. Landis.
\newblock A phase-field description of dynamic brittle fracture.
\newblock {\em Comput. Methods Appl. Mech. Engrg.}, 217:77--95, 2012.

\bibitem{Bourdin_Francfort_Marigo_2008}
B.~Bourdin, G.~A. Francfort, and J.-J. Marigo.
\newblock The variational approach to fracture.
\newblock {\em J. Elasticity}, 91(1--3):1--148, 2008.

\bibitem{Burke_Ortner_Sueli_2010}
S.~Burke, C.~Ortner, and E.~S\"uli.
\newblock An adaptive finite element approximation of a variational model of
  brittle fracture.
\newblock {\em SIAM J. Numer. Anal.}, 48(3):980--1012, 2010.

\bibitem{Burke_Ortner_Sueli_2013}
S.~Burke, C.~Ortner, and E.~S\"uli.
\newblock An adaptive finite element approximation of a generalized
  {A}mbrosio-{T}ortorelli functional.
\newblock {\em M3AS}, 23(9):1663--1697, 2013.

\bibitem{Chen_Nochetto_2000}
Z.~Chen and R.~Nochetto.
\newblock Residual type a posteriori error estimates for elliptic obstacle
  problems.
\newblock {\em Numer. Math.}, 84(4):527--548, 2000.

\bibitem{Fierro_Veeser_2003}
F.~Fierro and A.~Veeser.
\newblock A posteriori error estimators for regularized total variation of
  characteristic functions.
\newblock {\em SIAM J. Numer. Anal.}, 41(6):2032--2055, 2003.

\bibitem{Francfort_Marigo_1998}
G.~Francfort and J.-J. Marigo.
\newblock Revisiting brittle fracture as an energy minimization problem.
\newblock {\em J. Mech. Phys. Solids}, 46(8):1319--1342, 1998.

\bibitem{Goll_Wick_Wollner_2017}
C.~Goll, T.~Wick, and W.~Wollner.
\newblock {DOpElib}: {D}ifferential equations and {O}ptimization {E}nvironment;
  {A} goal oriented software library for solving {PDE}s and optimization
  problems with {PDE}s.
\newblock {\em Archive of Numerical Software}, 5(2):1--14, 2017.

\bibitem{Griffith_1921}
A.~Griffith.
\newblock The phenomena of rupture and flow in solids.
\newblock {\em Philos. Trans. R. Soc. Lond.}, 221:163--198, 1921.

\bibitem{Gudi_Porwal_2014}
T.~Gudi and K.~Porwal.
\newblock A posteriori error control of discontinuous {G}alerkin methods for
  elliptic obstacle problems.
\newblock {\em Math. Comp.}, 83(286):579--602, 2014.

\bibitem{Gudi_Porwal_2016}
T.~Gudi and K.~Porwal.
\newblock A posteriori error estimates of discontinuous {G}alerkin methods for
  the {S}ignorini problem.
\newblock {\em J. Comput. Appl. Math.}, 292:257--278, 2016.

\bibitem{Heister_Wheeler_Wick_2015}
T.~Heister, M.~F. Wheeler, and T.~Wick.
\newblock A primal-dual active set method and predictor-corrector mesh
  adaptivity for computing fracture propagation using a phase-field approach.
\newblock {\em Comput. Methods Appl. Mech. Engrg.}, 290:466--495, 2015.

\bibitem{Krause_Veeser_Walloth_2015}
R.~Krause, A.~Veeser, and M.~Walloth.
\newblock An efficient and reliable residual-type a posteriori error estimator
  for the {S}ignorini problem.
\newblock {\em Numer. Math.}, 130(1):151--197, 2015.

\bibitem{Mang_Walloth_Wick_Wollner_2019}
K.~Mang, M.~Walloth, T.~Wick, and W.~Wollner.
\newblock Mesh adaptivity for quasi-static phase-field fractures based on a
  residual-type a posteriori error estimator.
\newblock {\em GAMM-Mitt.}, 43(1):e202000003, 22, 2020.

\bibitem{Miehe_Welschinger_Hofacker_2010a}
C.~Miehe, M.~Hofacker, and F.~Welschinger.
\newblock A phase field model for rate-independent crack propagation: robust
  algorithmic implementation based on operator splits.
\newblock {\em Comput. Methods Appl. Mech. Engrg.}, 199(45-48):2765--2778,
  2010.

\bibitem{Miehe_Welschinger_Hofacker_2010}
C.~Miehe, F.~Welschinger, and M.~Hofacker.
\newblock Thermodynamically consistent phase-field models of fracture:
  variational principles and multi-field {FE} implementations.
\newblock {\em Internat. J. Numer. Methods Engrg.}, 83(10):1273--1311, 2010.

\bibitem{Moon_Nochetto_Petersdorff_Zhang_2007}
K.-S. Moon, R.~H. Nochetto, T.~von Petersdorff, and C.-S. Zhang.
\newblock A posteriori error analysis for parabolic variational inequalities.
\newblock {\em M2AN Math. Model. Numer. Anal.}, 41(3):485--511, 2007.

\bibitem{Nochetto_Siebert_Veeser_2005}
R.~H. Nochetto, K.~G. Siebert, and A.~Veeser.
\newblock Fully localized a posteriori error estimators and barrier sets for
  contact problems.
\newblock {\em SIAM J. Numer. Anal.}, 42(5):2118--2135, 2005.

\bibitem{Schlueter_Willenbuecher_Kuhn_Mueller_2014}
A.~Schl\"uter, A.~Willenb\"ucher, C.~Kuhn, and R.~M\"uller.
\newblock Phase field approximation of dynamic brittle fracture.
\newblock {\em Comput. Mech.}, 54(5):1141--1161, 2014.

\bibitem{Veeser_2001}
A.~Veeser.
\newblock Efficient and reliable a posteriori error estimators for elliptic
  obstacle problems.
\newblock {\em SIAM J. Numer. Anal.}, 39(1):146--167, 2001.

\bibitem{Verfuerth_1998b}
R.~Verf\"{u}rth.
\newblock Robust a posteriori error estimators for a singularly perturbed
  reaction-diffusion equation.
\newblock {\em Numer. Math.}, 78(3):479--493, 1998.

\bibitem{Verfuerth_1998a}
R.~Verf\"{u}rth.
\newblock A review of a posteriori error estimation techniques for elasticity
  problems.
\newblock {\em Comput. Methods Appl. Mech. Engrg.}, 176(1-4):419--440, 1999.
\newblock New advances in computational methods (Cachan, 1997).

\bibitem{Verfuerth_2013}
R.~Verf\"{u}rth.
\newblock {\em A posteriori error estimation techniques for finite element
  methods}.
\newblock Numerical Mathematics and Scientific Computation. Oxford University
  Press, Oxford, 2013.

\bibitem{Walloth_2018b}
M.~Walloth.
\newblock Residual-type a posteriori estimators for a singularly perturbed
  reaction-diffusion variational inequality -- reliability, efficiency and
  robustness.
\newblock Preprint 1812.01957, arXiv, 2018.

\bibitem{Walloth_2019}
M.~Walloth.
\newblock A reliable, efficient and localized error estimator for a
  discontinuous {G}alerkin method for the {S}ignorini problem.
\newblock {\em Appl. Numer. Math.}, 135:276--296, 2019.

\bibitem{Walloth_2020}
M.~Walloth.
\newblock Residual-type a posteriori error estimator for a quasi-static
  {S}ignorini contact problem.
\newblock {\em IMA J. Numer. Anal.}, 40(3):1937--1971, 2020.

\bibitem{Weiss_Wohlmuth_2010}
A.~Weiss and B.~I. Wohlmuth.
\newblock A posteriori error estimator for obstacle problems.
\newblock {\em SIAM J. Sci. Comput.}, 32(5):2627--2658, 2010.

\bibitem{Wick_2016}
T.~Wick.
\newblock Goal functional evaluations for phase-field fracture using pu-based
  dwr mesh adaptivity.
\newblock {\em Comput. Mech.}, 57(6):1017--1035, 2016.

\bibitem{Winkler2001}
B.~Winkler.
\newblock {\em Traglastuntersuchungen von unbewehrten und bewehrten
  Betonstrukturen auf der Grundlage eines objektiven Werkstoffgesetzes f{\"u}r
  Beton}.
\newblock PhD thesis, Universit{\"a}t Insbruck, 2001.

\bibitem{Zou_Veeser_Kornhuber_Graeser_2011}
Q.~Zou, A.~Veeser, R.~Kornhuber, and C.~Gr{\"a}ser.
\newblock Hierarchical error estimates for the energy functional in obstacle
  problems.
\newblock {\em Numer. Math.}, 117(4):653--677, 2011.

\end{thebibliography}
